\documentclass[12pt,english,reqno]{amsart}
\usepackage{a4wide}
\usepackage{amsmath,amssymb,amsfonts,amsthm}

\usepackage[
  pdfauthor={Bernd C. Kellner},
  pdfkeywords={p-adic functions, p-adic Kummer spaces,
    Kummer congruences, finite differences,
    Bernoulli numbers, Euler numbers, Fermat quotients},
  pdftitle={Classification of p-adic functions satisfying Kummer type congruences},
  linktocpage,colorlinks,bookmarksnumbered]{hyperref}

\newcommand{\pdiv}{\mid}

\newcommand{\notdiv}{\nmid}

\DeclareMathOperator{\real}{Re}

\DeclareMathOperator{\ord}{ord}
\DeclareMathOperator{\Lip}{Lip}
\DeclareMathOperator{\id}{id}

\newcommand{\RR}{\mathbb{R}}
\newcommand{\CC}{\mathbb{C}}
\newcommand{\QQ}{\mathbb{Q}}
\newcommand{\ZZ}{\mathbb{Z}}
\newcommand{\NN}{\mathbb{N}}
\newcommand{\PP}{\mathbb{P}}

\newcommand{\ff}{\mathfrak{f}}

\newcommand{\Dop}{\nabla}
\newcommand{\Doph}[2]{\nabla_{\hspace{-0.5ex}#1}^{\hspace{-0.03ex}#2}}
\newcommand{\Dophh}{\Doph{h}{}}
\newcommand{\Diffop}{\mathcal{D}}
\newcommand{\KSop}{\eth}
\newcommand{\binomh}[3]{\binom{#1}{#2}_{\hspace{-0.5ex}#3}}
\newcommand{\valueat}[1]{{\,}_{\big|\, #1}}

\newcommand{\Sp}{S}
\newcommand{\norm}[1]{\left|#1\right|}
\newcommand{\norml}[1]{|#1|}
\newcommand{\normf}[1]{\left\|#1\right\|}
\newcommand{\bracket}[1]{\left\langle #1\right\rangle}
\newcommand{\eulerphi}{\varphi}
\newcommand{\EF}[1]{\varepsilon_{#1}}

\newcommand{\KS}{\mathcal{K}}
\newcommand{\KSS}{\mathcal{K}^\sharp}
\newcommand{\WKS}{\widehat{\mathcal{K}}}
\newcommand{\BKS}{\mathcal{B}}

\newcommand{\LS}{\mathcal{O}}
\newcommand{\IRR}{\Psi^{\mathrm{irr}}}
\newcommand{\EXC}{\Psi^{\mathrm{exc}}}

\newcommand{\CZ}{\mathcal{C}(\ZZ_p)}

\newcommand{\SD}{\mathcal{S}^1}
\newcommand{\SE}{\mathcal{E}}
\newcommand{\SU}{\mathcal{U}}

\newtheorem{prop}{Proposition}[section]
\newtheorem{lemma}[prop]{Lemma}
\newtheorem{theorem}[prop]{Theorem}
\newtheorem{corl}[prop]{Corollary}
\newtheorem{algo}[prop]{Algorithm}
\newtheorem{defin}[prop]{Definition}
\newtheorem{conj}[prop]{Conjecture}

\theoremstyle{remark}
\newtheorem*{remark*}{Remark}
\newtheorem{remark}[prop]{Remark}

\theoremstyle{definition}
\newtheorem{tbl}[prop]{Table}
\newtheorem{expl}[prop]{Example}

\numberwithin{equation}{section}
\newcommand{\textdef}[1]{\textbf{#1}}

\title[Classification of $p$-adic functions satisfying Kummer type congruences]
{Classification of $p$-adic functions satisfying\\ Kummer type congruences}
\author{Bernd C. Kellner}
\date{}
\address{Mathematisches Institut, Universit\"at G\"ottingen, Bunsenstr.\ 3–-5,
37073 G\"ottingen, Germany}
\email{bk@bernoulli.org}

\begin{document}

\begin{abstract}
We introduce $p$-adic Kummer spaces of continuous functions on $\mathbb{Z}_p$
that satisfy certain Kummer type congruences. We will classify these spaces and
show their properties, for instance, ring properties and certain
decompositions. As a result, these functions have always a fixed point,
functions of certain subclasses have always a unique simple zero in
$\mathbb{Z}_p$. The fixed points and the zeros are effectively computable by
given algorithms. This theory can be transferred to values of Dirichlet
$L$-functions at negative integer arguments. That leads to a conjecture about
their structure supported by several computations. In particular we give an
application to the classical Bernoulli and Euler numbers. Finally, we present a
link to $p$-adic functions that are related to Fermat quotients.
\end{abstract}

\maketitle

\tableofcontents

\section{Introduction}

Throughout this paper $p$ denotes a prime. The author \cite{Kellner:2007}
showed some special results for $p$-adic zeta functions, introduced by
Kubota and Leopoldt \cite{Kubota&Leopoldt:1964}, see especially Koblitz
\cite{Koblitz:1996}. These functions interpolate values of divided Bernoulli
numbers in certain residue classes, which are values of the Riemann zeta
function at negative integers, modified by an Euler factor.

To avoid confusion, the $p$-adic $L$-function $L_p(s,\chi)$
in context of Iwasawa theory is the second construction of Kubota-Leopoldt,
while we only consider their first construction.

Although there is a vast literature about Kummer congruences and its
generalizations for Bernoulli numbers and other special sequences, commonly
called Kummer type congruences, the results are presented in their contexts.

We will here establish a generalized theory, using new methods, for arbitrary
$p$-adic functions that satisfy certain Kummer type congruences. This is
embedded in the theory of continuous functions on $\ZZ_p$, which always have a
Mahler expansion. Therefore we introduce the $p$-adic Kummer spaces $\KS_{p,1}$
and $\KS_{p,2}$ of such functions and show their relations and properties, for
instance, ring properties, certain decompositions, and that $\KS_{p,2}
\subsetneq \KS_{p,1}$.

As a result, functions of $\KS_{p,1}$ and $\KS_{p,2}$ have always a fixed point
in $\ZZ_p$. Functions of an important subclass $\WKS^0_{p,2} \subset \KS_{p,2}$
have always a unique simple zero in $\ZZ_p$. A product of the latter functions
provides a product of linear terms when viewed in the $p$-adic norm. We
present two algorithms, which can effectively compute a truncated $p$-adic
expansion of the fixed point of a function of $\KS_{p,2}$, resp., of the zero
of a function of $\WKS^0_{p,2}$.

All results of the $p$-adic Kummer spaces can be transferred back to values of
ordinary $L$-functions at negative integer arguments, which are associated with
a real Dirichlet character. Since these functions, modified by an Euler factor,
obey the Kummer type congruences in certain residue classes, we obtain $p$-adic
$L$-functions of $\KS_{p,2}$. In contrast, the construction of Kubota-Leopoldt
yields a $p$-adic $L$-function of $\KS_{p,1}$ in our terminology. As a special
case, we apply these results to the classical Bernoulli and Euler numbers.

At the end, we construct $p$-adic functions using Fermat quotients, which have
a similar behavior as the $p$-adic zeta and $L$-functions mentioned above.

\section{Preliminaries}

Let $\NN$, $\PP$, $\ZZ$, $\QQ$, $\RR$, and $\CC$ be the set of positive
integers, the set of primes, the ring of integers, the field of rational, real,
and complex numbers, respectively. Let $\ZZ_p$ be the ring of $p$-adic integers
and $\QQ_p$ be the field of $p$-adic numbers. The ultrametric absolute value
$\norm{\cdot}_p$ is defined by $\norm{s}_p = p^{-\ord_p s}$ on $\QQ_p$. Let
$\norm{\cdot}_\infty = \norm{\cdot}$ be the usual norm on $\QQ_\infty = \RR$
and $\CC$. Following \cite[Ch.~4--5]{Robert:2000}, we denote $\CZ$,
$\Lip(\ZZ_p)$, and $\SD(\ZZ_p)$ as the space of continuous, Lipschitz, and
strictly differentiable functions $f: \ZZ_p \to \ZZ_p$, respectively.
For continuous functions $f \in \CZ$ define the norm
$\normf{f}_p = \sup_{s \in \ZZ_p} \norm{f(s)}_p = \max_{s \in \ZZ_p} \norm{f(s)}_p$
on the compact space $\ZZ_p$. Let $\LS$ denote the Landau symbol.

\begin{defin} \label{def:diff-oper}
The linear forward difference operator $\Doph{h}{}$ with increment $h$
and its powers are defined by
\[
   \Doph{h}{n} f(s) = \sum_{\nu=0}^n \binom{n}{\nu} (-1)^{n-\nu} f(s+\nu h)
\]
for integers $n \geq 0$, $h \geq 1$, and any function $f: \ZZ_p \to \ZZ_p$.
For brevity we write $\Dop^n$ instead of $\Doph{1}{n}$. In case of ambiguity
we explicitly indicate the variable with a possible start value, for example
$t=0$, by the expression
\[
   \Doph{h}{n} f(s+t) \valueat{t = 0}.
\]
The falling factorials are defined by
\[
    (s)_0 = 1, \quad (s)_n = s(s-1) \cdots (s-n+1) \quad \mbox{for } n \geq 1.
\]
As usual, let $\binom{s}{n} = (s)_n/n!$ for $n \geq 0$ be the binomial
polynomial, which is a function on $\ZZ_p$. A series
\[
   f(s) = \sum_{\nu \geq 0} a_\nu \binom{s}{\nu}
\]
with coefficients $a_\nu \in \ZZ_p$, where $|a_\nu|_p \to 0$, is called a
Mahler series, which defines a continuous function $f: \ZZ_p \to \ZZ_p$.
\end{defin}

\begin{theorem}[Mahler \cite{Mahler:1958}] \label{thm:mahler}
If $f \in \CZ$, then $f$ has a unique Mahler expansion
\[
   f(s) = \sum_{\nu \geq 0} a_\nu \binom{s}{\nu},
\]
where the coefficients $a_\nu \in \ZZ_p$ are given by
$a_\nu = \Dop^\nu f(0)$ with $|a_\nu|_p \to 0$.
\end{theorem}

\begin{lemma}[{\cite[Ch.~5.1, p.~227]{Robert:2000}}] \label{lem:delta-binom}
Let $k \geq 1$ and $p^j \leq k < p^{j+1}$. We have
\[
   \left| \binom{s}{k} - \binom{t}{k} \right|_p \leq p^j \, |s-t|_p,
     \quad s, t \in \ZZ_p.
\]
\end{lemma}

\begin{theorem} \label{thm:delta-binom-ord}
Let $h, m, n$ be positive integers with $m \geq n$. Then
\[
   \ord_p \left( p^m \, \Doph{h}{n} \, \binom{s}{m} \right)
     \geq n ( 1 + \ord_p h ), \quad s \in \ZZ_p.
\]
\end{theorem}

We will prove this theorem in the end of this section,
since we shall need several preparations.
For basic properties of differences see, for instance,
\cite{Graham&others:1994} and \cite{Robert:2000}.

\begin{lemma}[{\cite[Ch.~3.1, p.~241]{Robert:2000}}] \label{lem:padic-fact}
Let $n$ be a positive integer. We have
\[
   \ord_p n! = \frac{n-\Sp_p(n)}{p-1} \geq 0 \quad \text{and} \quad
     \ord_p \left( \frac{p^n}{n!} \right) = \frac{p-2}{p-1} n
     + \frac{\Sp_p(n)}{p-1} \geq 1,
\]
where $\Sp_p(n)$ is the sum of the digits of the $p$-adic expansion of $n$.
\end{lemma}

\begin{lemma} \label{lem:delta-poly}
Let $f \in \QQ_p[s]$ a function $f: \ZZ_p \to \ZZ_p$ with
$m = \deg f$ and $a_m \in \QQ_p$ be the highest coefficient of $f$.
For positive integers $h, n$ and $s \in \ZZ_p$ we have
\[
   \Doph{h}{n} \, f(s) = \left\{
     \begin{array}{rl}
       h^n \, g(s), & n < m, \\
       h^m \, m! \, a_m, & n = m, \\
       0, & n > m,  \\
     \end{array} \right.
\]
where $m! \, a_m, h^n \, g(s) \in \ZZ_p$ and in the latter case
$g \in \QQ_p[s]$ with $\deg g = m-n$.
If $f \in \ZZ_p[s]$, then $g \in \ZZ_p[s]$.
\end{lemma}

\begin{proof}
We have $\Dophh \, s = h$ and $\Dophh \, s^m = h ( m s^{m-1} + \LS(s^{m-2}) )$
for $m \geq 2$, while constant terms vanish under $\Doph{h}{}$.
Case $n < m$: We get $\Dophh \, f(s) = h \, \tilde{g}(s)$ with
$\deg \tilde{g} = m-1$ and by iteration that
$\Doph{h}{n} \, f(s) = h^n \, g(s)$ with $\deg g = m-n$. Since
$f$ takes only values in $\ZZ_p$, so also its differences. Thus
$h^n \, g(s) \in \ZZ_p$. If $f \in \ZZ_p[s]$, then $\Dophh$ maps coefficients
from $\ZZ_p$ onto $\ZZ_p$ and this provides that $g \in \ZZ_p[s]$.
Case $n=m$: Since lower terms vanish, we obtain a constant term
$\Doph{h}{m} \, f(s) = \Doph{h}{m} \, a_m \, s^m = h^m \, m! \, a_m \in \ZZ_p$.
For $h=1$ this implies $m! \, a_m \in \ZZ_p$.
Case $n>m$: The constant terms vanish.
\end{proof}

\begin{lemma} \label{lem:delta-h-binom}
Let $h, m, n$ be positive integers with $m \geq n$. Then
\[
   \Doph{h}{n} \, \binom{s}{m} = \sum_{\nu=0}^{n(h-1)} \!\! \binomh{n}{\nu}{h}
     \binom{s+\nu}{m-n}, \quad s \in \ZZ_p,
\]
where the $h$-nomial coefficients of order $n$ coincide with
\begin{equation} \label{eq:delta-h-binom}
   (1+x+\cdots+x^{h-1})^n = \sum_{\nu=0}^{n(h-1)} \binomh{n}{\nu}{h} \, x^\nu
     \quad \text{and} \quad \sum_{\nu=0}^{n(h-1)} \binomh{n}{\nu}{h} = h^n.
\end{equation}
\end{lemma}

\begin{proof}
Let $s \in \ZZ_p$. For $h=1$ this gives the usual differences
$\Dop^n \binom{s}{m} = \binom{s}{m-n}$ for $m \geq n \geq 1$. Now let $h > 1$.
Using $\binom{s+1}{m} = \binom{s}{m} + \binom{s}{m-1}$ successively, we get
\[
   \Dophh \, \binom{s}{m} = \sum_{\nu=0}^{h-1} \binom{s+\nu}{m-1}.
\]
Applying the above equation repeatedly yields the result,
whereas the coefficients are mapped, step by step for $r=1,\ldots,n$,
in the same way as $(1+x+\cdots+x^{h-1})^{r-1} \mapsto (1+x+\cdots+x^{h-1})^r$.
Then taking $x=1$ shows that the sum of the $h$-nomial coefficients
of order $n$ equals $h^n$.
\end{proof}

\begin{prop} \label{prop:delta-h-binom-zp}
Let $h, k, n$ be integers with $h,n \geq 1$ and $k \geq 0$. Then
\[
   \theta(n,k) = h^{-n} p^k \sum_{\nu=0}^{n(h-1)} \!\!
     \binomh{n}{\nu}{h} \binom{\nu}{k} \in \ZZ_p.
\]
\end{prop}

\begin{proof}
Note that $\theta(n,k) = 0$ for $k > n(h-1)$ and $\theta(n,0)=1$ by
\eqref{eq:delta-h-binom}. According to Lemma \ref{lem:delta-h-binom} set
$f(x) = 1+x+\cdots+x^{h-1}$. We will evaluate formal derivatives of both sides
of \eqref{eq:delta-h-binom}. Define the differential operator
$\Diffop^r = (d/dx)^r / r!$ for $r \geq 0$. We firstly have
\[
   \Diffop^r f(x) \valueat{x=1}
     = \sum_{\nu=r}^{h-1} \binom{\nu}{r} = \binom{h}{r+1},
\]
which is valid for all $r \geq 0$. Let $k \geq 1$. We then deduce that
\begin{alignat}{1}
   \Diffop^k f(x)^n \valueat{x=1}
     &= \sum_{\nu_1+\cdots+\nu_n=k} \frac{1}{k!} \binom{k}{\nu_1,\ldots,\nu_n}
        f^{(\nu_1)}(x) \cdots f^{(\nu_n)}(x) \valueat{x=1} \nonumber \\
     &= \sum_{\nu_1+\cdots+\nu_n=k} \binom{h}{\nu_1+1} \cdots \binom{h}{\nu_n+1}
        \nonumber \\
     &= \sum_{\nu_1+\cdots+\nu_n=k} \frac{h^n}{(\nu_1+1)\cdots(\nu_n+1)}
        \binom{h-1}{\nu_1} \cdots \binom{h-1}{\nu_n}.
        \label{eq:loc-delta-h-binom-1}
\end{alignat}
Since $p^l/(l+1) \in \ZZ_p$ for $l \geq 0$, we obtain for
$\nu_1+\cdots+\nu_n=k$ that
\begin{equation} \label{eq:loc-delta-h-binom-2}
   \frac{p^k}{(\nu_1+1)\cdots(\nu_n+1)}
     = \frac{p^{\nu_1}}{\nu_1+1} \cdots \frac{p^{\nu_n}}{\nu_n+1} \in \ZZ_p.
\end{equation}
Thus, replacing $h^n$ by $p^k$ in \eqref{eq:loc-delta-h-binom-1} and using
\eqref{eq:loc-delta-h-binom-2}, this shows that
\[
   h^{-n} p^k \, \Diffop^k f(x)^n \valueat{x=1} \in \ZZ_p.
\]
Now, consider the right hand side of \eqref{eq:delta-h-binom}.
We finally achieve that
\[
    h^{-n} p^k \, \Diffop^k f(x)^n \valueat{x=1}
     =  h^{-n} p^k \, \sum_{\nu=0}^{n(h-1)} \!\!
     \binomh{n}{\nu}{h} \binom{\nu}{k} \in \ZZ_p. \qedhere
\]
\end{proof}

\begin{prop} \label{prop:delta-h-binom-zp-2}
Let $h, m, n$ be positive integers. Then
\[
   f_n(s) = h^{-n} p^{m-n} \sum_{\nu=0}^{n(h-1)} \!\!
     \binomh{n}{\nu}{h} \binom{s+\nu}{m-n} \in \ZZ_p[s]
\]
for $n=1,\ldots,m$.
\end{prop}

\begin{proof}
We use the Vandermonde's convolution identity,
cf.~\cite[Ch.~5.1, p.~170]{Graham&others:1994}, that
\[
   \binom{s+a}{n} = \sum_{k=0}^n \binom{a}{k} \binom{s}{n-k}.
\]
Hence
\begin{align*}
   f_n(s)
     &= h^{-n} p^{m-n} \sum_{\nu=0}^{n(h-1)} \!\!
        \binomh{n}{\nu}{h} \binom{s+\nu}{m-n} \\
     &= h^{-n} p^{m-n} \sum_{\nu=0}^{n(h-1)} \!\!
        \binomh{n}{\nu}{h} \sum_{k=0}^{m-n}
        \binom{\nu}{k} \binom{s}{m-n-k} \\
     &= \sum_{k=0}^{m-n} p^{m-n-k} \binom{s}{m-n-k} \,
        h^{-n} p^k \sum_{\nu=0}^{n(h-1)} \binomh{n}{\nu}{h} \binom{\nu}{k}.
\end{align*}
Lemma \ref{lem:padic-fact} and Proposition \ref{prop:delta-h-binom-zp}
provide that $f_n \in \ZZ_p[s]$.
\end{proof}
\smallskip

\begin{proof}[Proof of Theorem \ref{thm:delta-binom-ord}]
Using Lemma \ref{lem:delta-h-binom} and Proposition
\ref{prop:delta-h-binom-zp-2}, we obtain
\[
   \ord_p \left( p^m \, \Doph{h}{n} \, \binom{s}{m} \right)
     = \ord_p \left( (ph)^n \, f_n(s) \right)
   \geq n ( 1 + \ord_p h ). \qedhere
\]
\end{proof}
\medskip

\section{$p$-adic Kummer spaces}

\begin{defin} \label{def:kummer-space}
We introduce the spaces $\KS_{p,1}$, $\KS_{p,2}$, and $\KSS_{p,2}$, which
we call $p$-adic Kummer spaces. Furthermore we define the sets $\WKS_{p,1}$
and $\WKS_{p,2}$. We distinguish between the following congruences of a
function $f: \ZZ_p \to \ZZ_p$ for any $s, t \in \ZZ_p$, $s \neq t$, and any
$n \geq 0$:
\begin{enumerate}
\item \textdef{Kummer congruences}: If $f \in \KS_{p,1}$, then
\[
   s \equiv t \pmod{p^n \ZZ_p} \quad \Longrightarrow \quad
     f(s) \equiv f(t) \pmod{p^{n+1} \ZZ_p}.
\]
If the converse also holds, then $f \in \WKS_{p,1}$.
\item \textdef{Kummer type congruences I}: If $f \in \KS_{p,2}$, then
\[
   \Dop^n f(s) \equiv 0 \pmod{p^n \ZZ_p}.
\]
We write $\Delta_f(n) = \Dop^n f(0)/p^n$, where $\Delta_f(n) \in \ZZ_p$.
Furthermore we write
\[
   \Delta_f \equiv \Delta_f(1) \pmod{p\ZZ_p}, \quad 0 \leq \Delta_f < p.
\]
If $\Delta_f \neq 0$, additionally $2 \pdiv \Delta_f(2)$ in case $p = 2$,
then $f \in \WKS_{p,2}$.
\item \textdef{Kummer type congruences II}: If $f \in \KSS_{p,2}$, then
\[
   \Doph{h}{n} f(s) \equiv 0 \pmod{p^{nr} \ZZ_p}
\]
for any $h \geq 1$, where $r=1+\ord_p h$.
\end{enumerate}
\end{defin}
\smallskip

By definition we have $\WKS_{p,1} \subset \KS_{p,1}$,
$\WKS_{p,2} \subset \KS_{p,2}$, and $\KSS_{p,2} \subset \KS_{p,2}$.
Clearly, a function $f \in \KS_{p,\nu}$, $\nu = 1,2$, resp.,
$f \in \KSS_{p,2}$ is continuous on $\ZZ_p$. For $f \in \KS_{p,1}$ this
follows by the usual $\varepsilon$-$\delta$ criterion of continuity.
For $f \in \KS_{p,2}$, resp., $f \in \KSS_{p,2}$ this is a consequence
of its Mahler expansion. The definition of $\KS_{p,2}$ above can be
weakened as follows.

\begin{lemma} \label{lem:ks2-kummer-congr-0}
Let the function $f: \ZZ_p \to \ZZ_p$ satisfy $\Dop^n f(0) \in p^n \ZZ_p$
for all $n \geq 0$. Then $f \in \KS_{p,2}$.
\end{lemma}

\begin{proof}
The Mahler expansion of $f$ is easily given by
\begin{equation} \label{eq:mahler-exp-ks2}
   f(s) = \sum_{\nu \geq 0} \Delta_f(\nu) \, p^\nu \binom{s}{\nu},
          \quad s \in \ZZ_p.
\end{equation}
Since $\Dop^n \, p^\nu \binom{s}{\nu} \equiv 0 \pmod{p^n\ZZ_p}$
for all $n, \nu \geq 0$, it follows that
$\Dop^n f(s) \equiv 0 \pmod{p^n \ZZ_p}$ for all $n \geq 0$
independently of $s \in \ZZ_p$.
\end{proof}

\begin{lemma}
The space $\KS_{p,2}$ has the basis
\[
   \BKS_p = \left\{ p^\nu \binom{s}{\nu} \right\}_{\nu \geq 0}.
\]
\end{lemma}

\begin{proof}
Clearly, $\Dop^n \, p^\nu \binom{s}{\nu} \equiv 0 \pmod{p^n\ZZ_p}$ for all
$n, \nu \geq 0$. The functions of $\BKS_p$ are linearly independent over $\ZZ_p$.
Let $f \in \KS_{p,2}$. The unique Mahler expansion of $f$ is given as in
\eqref{eq:mahler-exp-ks2}. Thus $(\Delta_f(0),\Delta_f(1),\Delta_f(2),\ldots)$
are the coordinates with respect to the basis $\BKS_p$. Conversely, prescribed
values $\Delta_f(\nu) \in \ZZ_p$ define a uniformly convergent Mahler series,
which then equals the Mahler expansion in \eqref{eq:mahler-exp-ks2}.
\end{proof}

\begin{lemma} \label{lem:oper-ks2}
We define the operator $\KSop^r = p^{-r} \Dop^r$ on $\KS_{p,2}$ for $r \geq 0$.
If $f \in \KS_{p,2}$, then $g = \KSop^r f \in \KS_{p,2}$,
where we have a shift of coefficients such that
$\Delta_g(\nu) = \Delta_f(\nu+r)$ for all $\nu \geq 0$.
Moreover, we have the relation $\Delta_f(r) = \KSop^r f(0)$.
\end{lemma}

\begin{proof}
Applying $\KSop^r$ to \eqref{eq:mahler-exp-ks2}, this yields
\[
   \KSop^r f(s) = \sum_{\nu \geq 0} \Delta_f(\nu+r) \, p^\nu \binom{s}{\nu}
                = \sum_{\nu \geq 0} \Delta_g(\nu) \, p^\nu \binom{s}{\nu}
                = g(s) \in \KS_{p,2}.
\]
By definition follows that $\Delta_f(r) = \Dop^r f(0)/p^r = \KSop^r f(0)$.
\end{proof}

The next theorem shows among other relations that the Kummer type congruences
imply the Kummer congruences, but the converse does not hold.
We will prove the theorem by the following propositions and corollaries.

\begin{theorem} \label{thm:ks-relations}
We have the following relations:
\begin{enumerate}
\item $\KS_{p,2} \subsetneq \KS_{p,1}$.
\item $\WKS_{p,2} \subsetneq \WKS_{p,1}$.
\item $\KS_{p,2} = \KSS_{p,2}$.
\item $\KS_{p,2} \subset \SD(\ZZ_p)$.
\end{enumerate}
\end{theorem}

\begin{prop} \label{prop:ks2-eq-ks3}
We have $\KS_{p,2} = \KSS_{p,2}$.
\end{prop}

\begin{proof}
By definition we have $\KSS_{p,2} \subset \KS_{p,2}$.
Now, we have to show that $\KS_{p,2} \subset \KSS_{p,2}$.
Let $h, n$ be positive integers and $f \in \KS_{p,2}$.
Since the sequence $(f_m)_{m \geq 1}$ of the partial sums $f_m$ of
the Mahler expansion of $f$ is uniformly convergent to $f$,
we can apply the operator $\Doph{h}{n}$ term by term:
\[
   \Doph{h}{n} f(s) = \sum_{\nu \geq 0} \Delta_f(\nu) \, p^\nu \,
     \Doph{h}{n} \binom{s}{\nu} .
\]
The lower terms for $\nu < n$ vanish by Lemma \ref{lem:delta-poly},
while Theorem \ref{thm:delta-binom-ord} gives
an estimate for the other terms where $\nu \geq n$:
\[
   \ord_p \left( p^\nu \, \Doph{h}{n} \, \binom{s}{\nu} \right) \geq n r,
     \quad r = 1 + \ord_p h.
\]
Thus $\Doph{h}{n} f(s) \equiv 0 \pmod{p^{nr}}$,
which shows that $f \in \KSS_{p,2}$.
\end{proof}

\begin{remark*}
The fact that $\KS_{p,2} = \KSS_{p,2}$ is only caused by properties of
binomial polynomials given in Theorem \ref{thm:delta-binom-ord}.
These properties ensure that Kummer type congruences I already imply
Kummer type congruences II.
\end{remark*}
\bigskip

\begin{prop} \label{prop:f-quod-delta}
Let $f \in \KS_{p,2}$ and $s,t \in \ZZ_p$, $s \neq t$.
We have the following statements:
\begin{enumerate}
\item
  \begin{equation} \label{eq:frac-f-zero}
     \frac{f(s)-f(t)}{s-t} \equiv 0 \pmod{p \ZZ_p}.
  \end{equation}
\item
  \begin{equation} \label{eq:frac-f-delta}
     \frac{f(s)-f(t)}{p\,(s-t)} \equiv \Delta_f \pmod{p \ZZ_p}.
  \end{equation}
\item
  \[
     f'(s) \equiv p \,\Delta_f \pmod{p^2 \ZZ_p}.
  \]
\end{enumerate}
For the cases (2) and (3) we additionally require that
$2 \pdiv \Delta_f(2)$ when $p=2$.
\end{prop}

\begin{proof}
Since $f \in \KS_{p,2}$, we make use of the Mahler expansion of $f$
and show for $s \neq t$ that
\begin{equation} \label{eq:frac-binom-delta}
   \frac{\sum_{\nu \geq 0} \Delta_f(\nu) \, p^\nu
     \left[ \binom{s}{\nu} - \binom{t}{\nu} \right] }{p\,(s-t)} =
   \Delta_f(1) + \Delta_f(2) \frac{p}{2}(s+t-1) + \LS(p) .
\end{equation}
The lower terms are easily given.
For the higher terms with $\nu \geq 3$, we have the following estimate by
Lemma \ref{lem:delta-binom}:
\[
   \ord_p \left( \binom{s}{\nu} - \binom{t}{\nu} \right) \geq
     \ord_p ( s-t ) - \lfloor \log_p \nu \rfloor,
\]
where $\log_p$ is the real-valued logarithm to base $p$.
Since $r = \nu - 1 - \lfloor \log_p \nu \rfloor \geq 1$ for $\nu \geq 3$
and all primes $p$, we obtain in these cases that
\begin{equation} \label{eq:frac-binom-delta-2}
   p^\nu \left( \binom{s}{\nu} - \binom{t}{\nu} \right) \Big/ p\,(s-t)
     \in p^r\ZZ_p ,
\end{equation}
where $r \to \infty$ as $\nu \to \infty$.
Therefore \eqref{eq:frac-f-zero} follows by \eqref{eq:frac-binom-delta}.
By definition $\Delta_f(1) \equiv \Delta_f \pmod{p \ZZ_p}$ and
in case $p = 2$ we have $2 \pdiv \Delta_f(2)$,
thus \eqref{eq:frac-f-delta} follows by \eqref{eq:frac-binom-delta}.
Note that \eqref{eq:frac-binom-delta} is responsible for the extra
condition in the case $p = 2$.
Now, taking any sequence $(s_\nu,t_\nu)_{\nu \geq 1} \!\to (s,s)$
where $s_\nu \neq t_\nu$, \eqref{eq:frac-binom-delta} and
\eqref{eq:frac-binom-delta-2} show the existence
of a limit: $f'(s) \equiv p \,\Delta_f \pmod{p^2 \ZZ_p}$.
\end{proof}

\begin{corl} \label{corl:ks2-strong-kummer-congr}
We have $\KS_{p,2} \subsetneq \KS_{p,1}$ and
$\WKS_{p,2} \subsetneq \WKS_{p,1}$. If $f \in \WKS_{p,2}$, then a strong
version of the Kummer congruences holds, that
\[
   \norm{f(s)-f(t)}_p = \norm{p \, (s-t)}_p, \quad s, t \in \ZZ_p.
\]
\end{corl}

\begin{proof}
Let $f \in \KS_{p,2}$ and $s, t \in \ZZ_p$.
We omit the trivial case \mbox{$s = t$} and assume that \mbox{$s \neq t$}.
Eq.~\eqref{eq:frac-binom-delta} shows that
$\norm{f(s)-f(t)}_p \leq \norm{p \, (s-t)}_p$.
This implies the Kummer congruences and that $f \in \KS_{p,1}$.
For $f \in \WKS_{p,2}$, we have $\Delta_f \neq 0$
and additionally in case $p=2$ that $2 \pdiv \Delta_f(2)$.
Then \eqref{eq:frac-f-delta} yields $\norm{(f(s)-f(t))/(p\,(s-t))}_p = 1$,
which gives the equation above and shows that $f \in \WKS_{p,1}$.
It remains to show that $\KS_{p,2} \neq \KS_{p,1}$ and
$\WKS_{p,2} \neq \WKS_{p,1}$. We construct a function
$f \in \KS_{p,2}$, resp., $f \in \WKS_{p,2}$ by choosing suitable
coefficients $\Delta_f(\nu)$.
Therefore we can assume that an index $n \geq 5$ exists, where
$\Delta_f(n) \in \ZZ^*_p$. Now define
\[
   \tilde{f}(s) = \sum_{\nu \geq 0} \Delta_f(\nu) \,
     p^{\nu-\delta_{n,\nu}} \binom{s}{\nu}
\]
with $\delta_{n,\nu}$ as Kronecker's delta. Then $\tilde{f}$ also satisfies
\eqref{eq:frac-binom-delta}, since higher terms
do not play a role in view of \eqref{eq:frac-binom-delta-2}.
This implies that $\tilde{f} \in \KS_{p,1}$, resp., $\tilde{f} \in \WKS_{p,1}$.
By construction
\[
   \Dop^n \tilde{f}(s) \equiv \Delta_f(n) \,
     p^{n-1} \not\equiv 0 \pmod{p^n \ZZ_p},
\]
consequently $\tilde{f} \notin \KS_{p,2}$, resp., $\tilde{f} \notin \WKS_{p,2}$.
\end{proof}

\begin{corl} \label{corl:ks2-strict-diff}
We have
\[
   \KS_{p,2} \subset \SD(\ZZ_p) \subset \Lip(\ZZ_p) \subset \CZ .
\]
If $f \in \KS_{p,2}$, then the Volkenborn integral of $f$ is given by
\[
   \int_{\ZZ_p} f(s) \, ds = \sum_{\nu \geq 0} (-1)^\nu
     \frac{\Delta_f(\nu) \, p^\nu}{\nu+1} \in \ZZ_p .
\]
\end{corl}

\begin{proof}
As seen in Proposition \ref{prop:f-quod-delta}, a function $f \in \KS_{p,2}$
is Lipschitz and moreover strictly differentiable. The latter is also a
consequence that the coefficients of the Mahler expansion obey that
$\norm{\Delta_f(\nu) \, p^\nu / \nu}_p \to 0$ as $\nu \to \infty$,
see \cite[Ch.~5.1, p.~227]{Robert:2000}. Therefore
$\KS_{p,2} \subset \SD(\ZZ_p) \subset \Lip(\ZZ_p) \subset \CZ$.
As a further consequence that $f \in \SD(\ZZ_p)$, the Volkenborn integral
is given as above, see \cite[Ch.~5.2, p.~265]{Robert:2000}.
Since $p^\nu/(\nu+1)$ is $p$-integral, the sum lies in $\ZZ_p$.
\end{proof}

This proves Theorem \ref{thm:ks-relations}. \qedsymbol \quad
Functions of $\KS_{p,1}$ and $\KS_{p,2}$ obey the Kummer congruences.
Moreover, one can easily calculate any
values$\pmod{p^{n+1} \ZZ_p}$ of functions of $\KS_{p,2}$ by a finite
Mahler expansion. Additionally, we give a
formulation in terms like the Kummer congruences.

\begin{prop} \label{prop:kummer-congr-value}
Let $f \in \KS_{p,2}$. For $n \geq 0$ and $s \in \ZZ_p$ we have
\[
   f(s) \equiv \sum_{\nu = 0}^n \Delta_f(\nu) \, p^\nu \binom{s}{\nu}
     \pmod{p^{n+1} \ZZ_p},
\]
resp.,
\[
   f(s) \equiv \sum_{\nu = 0}^{n} f(\nu) \binom{s}{\nu} \binom{n-s}{n-\nu}
     \pmod{p^{n+1} \ZZ_p}.
\]
\end{prop}

\begin{proof}
For $n \geq 0$ we have a finite Mahler expansion
\[
   f(s) \equiv \sum_{j = 0}^n \Delta_f(j) \, p^j \binom{s}{j}
     \equiv \sum_{j = 0}^n \Dop^j f(0) \, \binom{s}{j}
     \pmod{p^{n+1} \ZZ_p}.
\]
By definition we have
\[
   \Dop^j f(0)
     = \sum_{k=0}^j \binom{j}{k} (-1)^{j-k} f(k)
     = \sum_{k=0}^n \binom{j}{k} (-1)^{j-k} f(k) .
\]
We rearrange the finite sums and omit the vanishing terms. Hence
\[
   f(s) \equiv \sum_{k = 0}^{n} f(k) \sum_{j = k}^n
     (-1)^{j-k} \binom{j}{k} \binom{s}{j}
     \pmod{p^{n+1} \ZZ_p}.
\]
We use the following identities,
cf.~\cite[Ch.~5.1, pp.~164--168]{Graham&others:1994}, for $j \geq k$:
\[
   \binom{s}{j} \binom{j}{k} = \binom{s}{k} \binom{s-k}{j-k}
\]
and
\[
   \sum_{j = k}^n (-1)^{j-k} \binom{s-k}{j-k}
     = \sum_{j = 0}^{n-k} (-1)^j \binom{s-k}{j}
     = (-1)^{n-k} \binom{s-k-1}{n-k} = \binom{n-s}{n-k}.
\]
This gives the result.
\end{proof}

As a consequence we obtain a kind of a reflection formula.

\begin{corl}
Let $f \in \KS_{p,2}$, $n \geq 0$, and $s \in \ZZ_p$. Then there exist
coefficients $a_\nu \in \ZZ_p$ depending on $s$ such that
\begin{alignat*}{2}
   f(s)   &\equiv \sum_{\nu = 0}^{n} a_\nu \, f(\nu)     && \pmod{p^{n+1} \ZZ_p}, \\
   f(n-s) &\equiv \sum_{\nu = 0}^{n} a_{n-\nu} \, f(\nu) && \pmod{p^{n+1} \ZZ_p}.
\end{alignat*}
\end{corl}

\begin{proof}
Set $a_\nu = \binom{s}{\nu} \binom{n-s}{n-\nu}$ for $\nu=0,\ldots,n$.
\end{proof}

As a further consequence, we get an expression via differences for
values at negative integer arguments.

\begin{corl}
Let $f \in \KS_{p,2}$. Let $n, r$ be integers, where $n \geq 0$ and $r > 0$.
Then
\[
   f( -r ) \equiv (-1)^n \, r \binom{n+r}{r} \, \Dop^n \frac{f(s)}{s+r}
     \valueat{s=0} \pmod{p^{n+1} \ZZ_p}.
\]
\end{corl}

\begin{proof}
From Proposition \ref{prop:kummer-congr-value} we have
\[
   f(-r) \equiv \sum_{\nu = 0}^{n} f(\nu) \binom{-r}{\nu} \binom{n+r}{n-\nu}
     \pmod{p^{n+1} \ZZ_p}.
\]
The result follows by
\[
   \binom{-r}{\nu} \binom{n+r}{n-\nu} = (-1)^\nu \binom{\nu+r-1}{r-1}
     \binom{n+r}{\nu+r}
     = (-1)^\nu \binom{n}{\nu} \binom{n+r}{r} \frac{r}{\nu+r}. \qedhere
\]
\end{proof}

\begin{lemma} \label{lem:f-comp-lin-ks2}
Let $f \in \KS_{p,2}$ and $\lambda(s) = a+bs$, where $a, s \in \ZZ_p$ and
$b \in \ZZ$. Then $f \circ \lambda \in \KS_{p,2}$, where
$(f \circ \lambda)(s) = f(\lambda(s))$. Moreover, if $b \neq 0$ then
$\Dop^n (f \circ \lambda)(s) \equiv 0 \pmod{p^{nr}}$,
where $r = 1 + \ord_p b$.
\end{lemma}

\begin{proof}
Case $b=0$: This gives a constant function $f(a) \in \KS_{p,2}$.
Case $b>0$: We have $\lambda(s+\nu) = a + bs + b\nu$.
Since $\KSS_{p,2} = \KS_{p,2}$, we obtain
\[
   \Dop^n (f \circ \lambda)(s) = \Doph{b}{n} f(t)
     \valueat{t = a + bs} \equiv 0 \pmod{p^{nr}}
\]
for all $n \geq 1$, where $r = 1 + \ord_p b$.
Thus $f \circ \lambda \in \KS_{p,2}$.
Case $b<0$: Set $b' = |b|$, then $\lambda(s+\nu) = a - b's - b'\nu$.
Recall Definition \ref{def:diff-oper} and note the symmetry of the
binomial coefficients. We then get
\[
   \Dop^n (f \circ \lambda)(s) = (-1)^n \, \Doph{b'}{n} f(t)
     \valueat{t = a - b's - b'n} \equiv 0 \pmod{p^{nr}}
\]
for all $n \geq 1$, where $r = 1 + \ord_p b' = 1 + \ord_p b$.
Consequently $f \circ \lambda \in \KS_{p,2}$.
\end{proof}
\smallskip

\section{Zeros and fixed points}

\begin{defin} \label{def:func-level}
Let $f \in \CZ$ and $n \geq 0$. We call
\[
   f_n(s) = p^{-n} f( \xi_n + s \, p^n ), \quad s \in \ZZ_p,
\]
a function of level $n$ of $f$, when $f_n$ defines a function
$f_n: \ZZ_p \to \ZZ_p$,
where the $p$-adic expansion of $\xi_n$ is given by
\[
   \xi_0 = 0, \quad \xi_n = s_0 + s_1 \, p + \cdots + s_{n-1} \, p^{n-1}
     \quad \text{for } n \geq 1.
\]
\end{defin}

\begin{prop} \label{prop:f-lip-zero}
Let $f \in \Lip(\ZZ_p)$ satisfy
\[
   \frac{f(s)-f(t)}{s-t} \equiv \Delta \pmod{p \ZZ_p}
     \quad \text{for } s \neq t, \quad s, t \in \ZZ_p,
\]
such that $\Delta$ is a fixed integer where $0 < \Delta < p$.
Then $f$ has a unique simple zero $\xi \in \ZZ_p$ and
\[
   f(s) = (s-\xi) \, f^*(s), \quad \norm{f(s)}_p = \norm{s-\xi}_p,
     \quad s \in \ZZ_p,
\]
where $f^*(s) \equiv \Delta \pmod{p \ZZ_p}$ and $f^* \in \CZ^*$.
\end{prop}

\begin{proof}
We show on induction that there exists a sequence $(f_n)_{n \geq 0}$
of functions of level $n$ of $f$, such that the sequence $(\xi_n)_{n \geq 0}$
is uniquely determined. Basis of induction $n=0$: $f_0(s) = f(s)$.
Inductive step $n \mapsto n+1$: Assume this is true for $n$ prove for $n+1$.
We have
\begin{equation} \label{eq:delta-fn}
\begin{split}
   \Dop f_n(s) &= p^{-n} \, ( f( \xi_n + (s+1) \, p^n )
      - f( \xi_n + s \, p^n ) ) \\
   &= \frac{f( \xi_n + (s+1) \, p^n ) - f( \xi_n + s \, p^n )}
      {( \xi_n + (s+1) \, p^n ) - ( \xi_n + s \, p^n )}.
\end{split}
\end{equation}
By assumption we get $\Dop f_n(s) \equiv \Delta \not\equiv 0 \pmod{p\ZZ_p}$.
Thus, we can uniquely determine the value $s_n$ by
\begin{equation} \label{eq:sn-fn}
   s_n \equiv -f_n(0) / \Delta \pmod{p\ZZ_p}, \quad 0 \leq s_n < p,
\end{equation}
such that $f_n(s_n) \equiv 0 \pmod{p\ZZ_p}$. It also follows that
\[
   f_n(s_n + s \, p ) \equiv 0 \pmod{p\ZZ_p} \quad \text{for } s \in \ZZ_p.
\]
Setting $\xi_{n+1} = \xi_n + s_n \, p^n$, we obtain the function
$f_{n+1}: \ZZ_p \to \ZZ_p$ by
\[
   f_{n+1}( s ) = p^{-1} \, f_n(s_n + s \, p ).
\]
Existence of the zero:
We achieve that $\lim_{n \to \infty} \norm{f(\xi_n)}_p = 0$.
Define $\xi = \lim_{n \to \infty} \xi_n$, then $\xi$ is a zero of $f$,
due to the fact that $f \in \Lip(\ZZ_p) \subset \CZ$.
Uniqueness of the zero:
Assume to the contrary that $\xi$ and $\xi'$ are different zeros of $f$.
Then
\[
   0 = \frac{f(\xi)-f(\xi')}{\xi - \xi'}
     \equiv \Delta \not\equiv 0 \pmod{p\ZZ_p}
\]
yields a contradiction. Representation of $f$: Since $f(\xi)=0$, we obtain
\[
   f^*(s) = \frac{f(s)}{s-\xi} \equiv \Delta \pmod{p\ZZ_p}
     \quad \text{for } s \neq \xi, \quad s \in \ZZ_p.
\]
We get $\lim_{s \to \xi} f^*(s) = f'(\xi)$
where $f^*(\xi) = f'(\xi) \equiv \Delta \pmod{p\ZZ_p}$.
This implies that $f^* \in \CZ$. Since $1/f^* \in \CZ$,
we even have $f^* \in \CZ^*$.
Finally, $f(s) = (s-\xi) \, f^*(s)$ and $\norm{f(s)}_p = \norm{s-\xi}_p$
for $s \in \ZZ_p$.
\end{proof}

\begin{remark*}
This result is similar to Hensel's Lemma, which predicts a zero $\xi$ of
a polynomial $g \in \ZZ_p[s]$, when
$\norm{g(\tilde{s})}_p < \norm{g'(\tilde{s})}^2_p$
for some $\tilde{s} \in \ZZ_p$.
Then $\norm{\xi-\tilde{s}}_p = \norm{g(\tilde{s})/g'(\tilde{s})}_p$,
cf.~\cite[Ch.~2.1, p.~80]{Robert:2000}.
But in this context, a function $f \in \Lip(\ZZ_p)$,
that satisfies the conditions of Proposition \ref{prop:f-lip-zero},
can have an infinite Mahler expansion in view of
Proposition \ref{prop:f-quod-delta}. Moreover this function has only one zero.
Note also that this result cannot be derived by the $p$-adic Weierstrass
Preparation Theorem, cf.~\cite[Thm.~7.3, p.~115]{Washington:1997},
since $s-\xi$ is not a distinguished polynomial when $\xi \in \ZZ^*_p$.
\end{remark*}

\begin{prop} \label{prop:f-lip-fixpnt}
Let $f \in \Lip(\ZZ_p)$ satisfy
\[
   \frac{f(s)-f(t)}{s-t} \equiv 0 \pmod{p \ZZ_p}
     \quad \text{for } s \neq t, \quad s, t \in \ZZ_p.
\]
We have the following statements:
\begin{enumerate}
\item The function $f$ has a fixed point $\tau \in \ZZ_p$.
\item If there exists a $\xi \in \ZZ_p$ such that $f(\xi) \in p^n\ZZ_p$ with
$n \geq 1$, then there exists a function $f_n$ of level $n$
where $\xi_n \equiv \xi \pmod{p^n\ZZ_p}$ and
$f_n(s) \equiv p^{-n} f(\xi) \pmod{p \ZZ_p}$ for $s \in \ZZ_p$.
\end{enumerate}
\end{prop}

\begin{proof}
(1): Since $(\ZZ_p,\norm{\cdot}_p)$ is a Banach space and the function $f$
defines a contractive mapping by $\norm{f(s)-f(t)}_p \leq p^{-1} \norm{s-t}_p$
for $s \neq t$, the Banach fixed point theorem provides
a $\tau \in \ZZ_p$ such that $f(\tau) = \tau$.
(2): We have $\norm{f(s)-f(t)}_p \leq \norm{p(s-t)}_p$, which implies the
Kummer congruences. Assume that $f(\xi) \in p^n\ZZ_p$.
According to Definition \ref{def:func-level}, we define $f_n$ and determine
$\xi_n \in \ZZ_p$ by $\xi_n \equiv \xi \pmod{p^n\ZZ_p}$.
Then we get $f(\xi_n) \equiv f(\xi) \pmod{p^{n+1}\ZZ_p}$.
Similar to \eqref{eq:delta-fn} we obtain $\Dop f_n(s) \equiv 0 \pmod{p\ZZ_p}$
for $s \in \ZZ_p$.
Consequently, $f_n(s) \equiv p^{-n} f(\xi) \pmod{p \ZZ_p}$ for $s \in \ZZ_p$.
\end{proof}

\begin{defin} \label{def:decomp-spaces-0}
We define for $\mathcal{T} = \KS_{p,\nu}, \WKS_{p,\nu}$, $\nu=1,2$, the
decomposition $\mathcal{T} = \mathcal{T}^0 \cup \mathcal{T}^*$ where
\[
   \mathcal{T}^0 = \{ f \in \mathcal{T} : f(0) \in p\ZZ_p \}, \quad
     \mathcal{T}^* = \{ f \in \mathcal{T} : f(0) \in \ZZ^*_p \}.
\]
\end{defin}

\begin{theorem} \label{thm:ks2-zero-fixpnt}
We have the following statements:
\begin{enumerate}
\item If $f \in \KS_{p,1}$ or $f \in \KS_{p,2}$,
      then $f$ has a fixed point $\tau \in \ZZ_p$.
\item If $f \in \KS^*_{p,1}$ or $f \in \KS^*_{p,2}$,
      then $\norm{f(s)}_p = 1$ for $s \in \ZZ_p$.
\item If $f \in \WKS^0_{p,2}$, then $f$ has a unique simple zero
      $\xi \in \ZZ_p$ and
\[
   f(s) = p \, (s-\xi) \, f^*(s), \quad \norm{f(s)}_p = \norm{p \, (s-\xi)}_p,
     \quad s \in \ZZ_p,
\]
where $f^*(s) \equiv \Delta_f \pmod{p \ZZ_p}$ and $f^* \in \CZ^*$.
\end{enumerate}
\end{theorem}

\begin{proof}
(1): Case $f \in \KS_{p,1}$: Since $\norm{f(s)-f(t)}_p \leq \norm{p(s-t)}_p$
for $s \neq t$, this shows the existence of a fixed point as argued in the
proof of Proposition \ref{prop:f-lip-fixpnt}.
Case $f \in \KS_{p,2}$: Proposition \ref{prop:f-quod-delta} provides
that $f$ satisfies the condition of Proposition \ref{prop:f-lip-fixpnt}.
(2): Clearly, since $f(0) \equiv f(s) \pmod{p\ZZ_p}$ for $s \in \ZZ_p$.
(3): Proposition \ref{prop:f-quod-delta} shows that $f/p$ satisfies the
conditions of Proposition \ref{prop:f-lip-zero}.
\end{proof}

Functions $f \in \WKS^0_{p,2}$ play a
significant role as we will later see in the next sections.
Another characterization of $\WKS^0_{p,2}$ is given by
\[
   \WKS^0_{p,2} = \left\{ f \in \KS_{p,2}: f \notin \KS^*_{p,2},
     \KSop f \in \KS^*_{p,2} \right\}.
\]
The following algorithm shows how to compute an approximation$\pmod{p^n\ZZ_p}$
of the zero of a function $f \in \WKS^0_{p,2}$. For this task we need the values
\[
   f(0)/p, \ldots, f(n)/p \pmod{p^n\ZZ_p}.
\]
We present two possible methods: The first uses the values of $f$,
the second its Mahler coefficients.

\begin{algo} \label{alg:zero}
Let $f \in \WKS^0_{p,2}$ and $\xi$ be the unique zero of $f$.
Let $n \geq 1$ be fixed.
Initially set $\xi_0 = 0$ and $\delta \equiv -\Delta_f^{-1} \pmod{p\ZZ_p}$.
Further compute for $\nu = 0,\ldots,n$ the values
\[
   \tilde{f}_\nu \equiv f(\nu)/p \pmod{p^n\ZZ_p},
\]
resp.,
\[
   \Delta_f(\nu) \pmod{p^{n-\nu+1}\ZZ_p}.
\]
For each step $r=1,\ldots,n$ proceed as follows. Compute
\[
   \gamma_{r-1} \equiv \sum_{\nu = 0}^{r} \tilde{f}_\nu \binom{\xi_{r-1}}{\nu}
     \binom{r-\xi_{r-1}}{r-\nu} \pmod{p^r \ZZ_p},
\]
resp.,
\begin{equation} \label{eq:algo-sum2}
   \gamma_{r-1} \equiv \sum_{\nu = 0}^{r} \Delta_f(\nu) \, p^{\nu-1}
     \binom{\xi_{r-1}}{\nu} \pmod{p^r \ZZ_p}.
\end{equation}
Then $\gamma_{r-1} \in p^{r-1}\ZZ_p$.
Set $s_{r-1} \equiv \gamma_{r-1} \delta / p^{r-1} \pmod{p\ZZ_p}$
where $0 \leq s_{r-1} < p$. Set $\xi_r = \xi_{r-1} + s_{r-1} \, p^{r-1}$
and go to the next step while $r < n$.
Finally, $\xi_n \equiv \xi \pmod{p^n\ZZ_p}$.
\end{algo}

\begin{proof}
The function $f/p$ satisfies the conditions of Proposition \ref{prop:f-lip-zero}.
We have to adapt the procedure given there to compute the zero of $f/p$.
For each step we use \eqref{eq:sn-fn} to get the next digit of the $p$-adic
expansion of $\xi$; to compute the term $f_{r-1}(0)$ we make use of
Proposition \ref{prop:kummer-congr-value} modified for $f/p$. Note that
$\Delta_f(0) = f(0) \in p\ZZ_p$, so \eqref{eq:algo-sum2} is valid.
\end{proof}

\begin{remark} \label{rem:lucas-binom}
The second method can be further optimized. The binomial coefficient
$\binom{\xi_{r-1}}{\nu}$ can be effectively computed$\pmod{p^{r-\nu+1} \ZZ_p}$,
since we already know the $p$-adic expansion of $\xi_{r-1}$. For the last term
of the sum of \eqref{eq:algo-sum2} we can apply the well known theorem of Lucas:
\[
   \binom{\xi_{r-1}}{r} \equiv \binom{s_0}{r_0} \cdots \binom{s_{r-2}}{r_{r-2}}
     \pmod{p \ZZ_p},
\]
where $r_\nu$ are the digits of the $p$-adic expansion of $r$.
Let $l = \lfloor \log_p r \rfloor$, then
\[
   \Delta_f(r) \, p^{r-1} \binom{\xi_{r-1}}{r} \equiv
     \Delta_f(r) \, p^{r-1} \binom{s_0}{r_0} \cdots \binom{s_l}{r_l}
     \pmod{p^r \ZZ_p}.
\]
Lucas' theorem can be extended to higher prime powers. Davis and Webb
\cite{Davis&Webb:1990} showed a similar formula to compute binomial coefficients
modulo $p^m$, $m \geq 2$, which uses slightly modified binomial coefficients
that are evaluated on blocks of $m$ digits.
\end{remark}

In the next algorithm, which computes an approximation$\pmod{p^n\ZZ_p}$ of
the fixed point of a function $f \in \KS_{p,2}$,
we can apply Lucas' theorem.

\begin{algo} \label{alg:fixed-point}
Let $f \in \KS_{p,2}$ and $\tau$ be the fixed point of $f$.
Let $n \geq 1$ be fixed.
Initially set $\tau_1 = t_0 \equiv f(0) \pmod{p\ZZ_p}$, where $0 \leq t_0 < p$.
Further compute the values
\[
   \Delta_f(\nu) \pmod{p^{n-\nu}\ZZ_p}, \quad \nu = 0,\ldots,n-1.
\]
For each step $r=1,\ldots,n-1$ proceed as follows. Compute
\[
   t_r \equiv p^{-r} \left( \sum_{\nu=0}^{r-1}
     \Delta_f(\nu) \, p^\nu \binom{\tau_r}{\nu} - \tau_r \right)
     + \Delta_f(r) \binom{t_0}{r_0} \cdots \binom{t_l}{r_l} \pmod{p\ZZ_p},
\]
where $0 \leq t_r < p$, $l = \lfloor \log_p r \rfloor$, and $r_\nu$ are the
$p$-adic digits of $r$. Set $\tau_{r+1} = \tau_r + t_r \, p^r$
and go to the next step while $r < n-1$.
Finally, $\tau_n \equiv \tau \pmod{p^n\ZZ_p}$.
\end{algo}

\begin{proof}
By Theorem \ref{thm:ks2-zero-fixpnt} $f$ has a fixed point $\tau$,
which solves simultaneously the congruences
\[
   \tau \equiv \sum_{\nu=0}^{r-1} \Delta_f(\nu) \, p^\nu \binom{\tau}{\nu}
     \pmod{p^r\ZZ_p}, \quad r \geq 1.
\]
Let $\tau_r = t_0 + t_1 \, p + \cdots + t_{r-1} \, p^{r-1}$ be the truncated
$p$-adic expansion of $\tau$ for $r \geq 1$ and set $\tau_0 = 0$.
Using Lemma \ref{lem:padic-fact}, we observe that
\begin{equation} \label{eq:loc-algo-fixpnt-1}
   p^\nu \binom{\tau_r}{\nu} \equiv \frac{p^\nu}{\nu!} \, (\tau_{r-1})_\nu
     \pmod{p^r\ZZ_p}, \quad \nu \geq 0, r \geq 1.
\end{equation}
Thus
\begin{equation} \label{eq:loc-algo-fixpnt-2}
   \tau_r \equiv \sum_{\nu=0}^{r-1} \Delta_f(\nu) \, p^\nu
     \binom{\tau_{r-1}}{\nu} \pmod{p^r\ZZ_p}, \quad r \geq 1.
\end{equation}
Now we use induction on $r$ to compute $t_r$. Basis of induction $r=1$:
By \eqref{eq:loc-algo-fixpnt-2} we get $\tau_1 = t_0 \equiv f(0) \pmod{p\ZZ_p}$.
Inductive step $r \mapsto r+1$: Assume this is true for $r$ prove for $r+1$.
We can use \eqref{eq:loc-algo-fixpnt-2} for $r+1$ to obtain
\[
   \tau_{r+1} = \tau_r + t_r \, p^r
     \equiv \sum_{\nu=0}^{r-1} \Delta_f(\nu) \, p^\nu \binom{\tau_r}{\nu}
     + \Delta_f(r) \, p^r \binom{\tau_r}{r} \pmod{p^{r+1}\ZZ_p}.
\]
Considering Remark \ref{rem:lucas-binom} and using
\eqref{eq:loc-algo-fixpnt-1} and \eqref{eq:loc-algo-fixpnt-2}, this yields
\[
   t_r \, p^r \equiv
     \underbrace{\sum_{\nu=0}^{r-1} \Delta_f(\nu)
     \, p^\nu \binom{\tau_r}{\nu} - \tau_r}_{\in \, p^r\ZZ_p}
     \,+\,\, \Delta_f(r) \, p^r \binom{t_0}{r_0} \cdots \binom{t_l}{r_l}
     \pmod{p^{r+1}\ZZ_p},
\]
where $l = \lfloor \log_p r \rfloor$.
Dividing the above congruence by $p^r$ gives the result.
\end{proof}

We have a close relation between the zero and the fixed point of
$f \in \WKS^0_{p,2}$.

\begin{lemma} \label{lem:rel-zero-fixpnt}
If $f \in \WKS^0_{p,2}$ with $f(0) \neq 0$, then
$\ord_p f(0) = \ord_p \tau = 1 + \ord_p \xi$. More precisely,
\[
   \tau / p \, \xi \equiv - \Delta_f \pmod{p\ZZ_p} \quad \text{and} \quad
   f(0)/\tau \equiv 1 \pmod{p\ZZ_p},
\]
where $\tau$ is the fixed point and $\xi$ is the zero of $f$.
\end{lemma}

\begin{proof}
The case $f(0)=0$ implies that $\xi=\tau=0$ and vice versa, which we have
excluded. Using \eqref{eq:frac-f-delta} of Proposition \ref{prop:f-quod-delta}
yields that
\[
   \frac{\tau}{p\,(\tau-\xi)} \equiv \Delta_f \pmod{p\ZZ_p}.
\]
Since $\Delta_f \neq 0$, we can invert the congruence such that
\[
   \Delta_f^{-1} \equiv \frac{p\,(\tau-\xi)}{\tau} \equiv -p \, \xi / \tau
     \pmod{p\ZZ_p}.
\]
This shows the claimed congruence and $\norm{\tau}_p = \norm{p \, \xi}_p$.
Furthermore Theorem \ref{thm:ks2-zero-fixpnt} shows that
$\norm{f(0)}_p = \norm{p \, \xi}_p$ and also
$f(0)/p \, \xi \equiv - \Delta_f \pmod{p\ZZ_p}$.
Thus $f(0)/\tau \equiv 1 \pmod{p\ZZ_p}$.
\end{proof}

Now, we revisit Theorem \ref{thm:ks2-zero-fixpnt} to show that one cannot
improve the result as follows.

\begin{prop} \label{prop:ks2-cntexpl}
If $f \in \WKS^0_{p,2}$, then $f$ can be decomposed as
\[
   f(s) = p \, (s-\xi) \, f^*(s), \quad s \in \ZZ_p,
\]
with $\xi \in \ZZ_p$, $f^*(s) \equiv \Delta_f \pmod{p \ZZ_p}$, and
$f^* \in \CZ^*$. But in general $f^* \notin \KS^*_{p,2}$.
\end{prop}

\begin{proof}
By Theorem \ref{thm:ks2-zero-fixpnt} we have the decomposition of $f$ as
above. We construct the following functions for $p \geq 3$ using the
binomial expansion in $\ZZ_p$, cf.\ Proposition \ref{prop:exp-func} later.
\[
   f_p(s) = (1+p)^s - 1 = \sum_{\nu \geq 1} p^\nu \binom{s}{\nu},
     \quad s \in \ZZ_p.
\]
Clearly, $f_p \in \WKS^0_{p,2}$ with $\Delta_{f_p} = 1$ and $f_p$ has a zero
$\xi = 0$. Set $f^*_p(s) = f_p(s)/ps$. We easily obtain
\[
   f^*_p(s) = 1 + \sum_{\nu \geq 1} \frac{p^\nu}{\nu+1} \binom{s-1}{\nu},
     \quad s \in \ZZ_p.
\]
It follows that $\Dop^{p-1} f^*_p(1) / p^{p-1} = \frac{1}{p}$ and
consequently $f^*_p \notin \KS^*_{p,2}$. Now, we consider the remaining
case $p=2$. We have to modify $f_p$ in the following way:
\[
   \tilde{f}_p(s) = (1+p)^s + p^2 \binom{s}{2} - 1
     = \sum_{\nu \geq 1} (1 + \delta_{2,\nu}) p^\nu \binom{s}{\nu},
     \quad s \in \ZZ_p.
\]
Then $\tilde{f}_p$ has the properties $\Delta_{\tilde{f}_p} = 1$ and
$2 \pdiv \Delta_{\tilde{f}_p}(2)$. Thus $\tilde{f}_p \in \WKS^0_{p,2}$
and $\xi = 0$. Similarly set $\tilde{f}^*_p(s) = \tilde{f}_p(s)/ps$ and
as usual $q=2p$. We derive in this case that
$\Dop^{q-1} \tilde{f}^*_p(1) / p^{q-1} = \frac{1}{q}$
and finally $\tilde{f}^*_p \notin \KS^*_{p,2}$.
\end{proof}

The proof above works with functions $f \in \KS_{p,2}$, that have a zero
at $s=0$. These functions have the following property.

\begin{lemma} \label{lem:ks2-zero0}
Let $f \in \KS_{p,2}$ with $f(0)=0$. Then $f(s) = ps \, g(s)$
for $s \in \ZZ_p$, where $g \in \CZ$ and
\begin{equation} \label{eq:ks2-zero0-g0}
   g(0) = \int_{\ZZ_p} \KSop f(s) \, ds = \sum_{\nu \geq 0} (-1)^\nu
     \frac{\Delta_f(\nu+1) \, p^\nu}{\nu+1}.
\end{equation}
\end{lemma}

\begin{proof}
Using the Mahler expansion and shifting the index, we obtain
\begin{equation} \label{eq:ks2-zero0-g}
   g(s) = f(s)/ps = \sum_{\nu > 0} \frac{\Delta_f(\nu) \, p^\nu
          \binom{s}{\nu}}{ps}
        = \sum_{\nu \geq 0} \frac{\Delta_f(\nu+1) \, p^\nu}{\nu+1}
          \binom{s-1}{\nu}.
\end{equation}
Since $p^\nu/(\nu+1) \to 0$ as $\nu \to \infty$ and $p^\nu/(\nu+1)$
is $p$-integral, we deduce that $g \in \CZ$.
Comparing the value of $g(0)$ and the Volkenborn integral
of $\KSop f$ by using Lemma \ref{lem:oper-ks2}
and Corollary \ref{corl:ks2-strict-diff} gives the result.
\end{proof}
\smallskip

\section{Degenerate functions}

Comparing the spaces $\KS_{p,1}$ and $\KS_{p,2}$, a function
$f \in \KS_{p,2}$ obeys a very strong law regarding its Mahler
expansion. We can think of degenerate functions as follows.

\begin{defin} \label{def:deg-func}
We call a function $f: \ZZ_p \to \ZZ_p$ $\delta$-degenerate, if
$f$ has a Mahler series such that
\[
   f(s) = \sum_{\nu \geq 0} \Delta'_{f,\delta}(\nu) \, p^{\delta(\nu)}
          \binom{s}{\nu}, \quad s \in \ZZ_p,
\]
where $\Delta'_{f,\delta}(\nu) \in \ZZ_p$ and
$\delta: \NN_0 \to \NN_0$ is a monotonically increasing function
with $\delta(\nu) \to \infty$ as $\nu \to \infty$.
We further define
\begin{alignat*}
  \eta_{f,\delta}(n)   &= \min_{\nu \geq 0} \{ \nu : \delta(\nu) \geq n \}, \\
  \vartheta_{f,\delta} &= \min_{\nu \geq 0} \{ \nu : \delta(\nu) < \nu \}.
\end{alignat*}
\end{defin}

The parameter $\vartheta_{f,\delta}$ determines the first index, where $f$ has
a \textit{defect} compared to a Mahler expansion of a function of $\KS_{p,2}$.
Note that $\delta$, depending on $f$, is not uniquely defined and has to be
chosen suitably. This has the following reason. Demanding that
$\Delta'_{f,\delta}(\nu) \in \ZZ_p^*$, we possibly obtain a non-monotonically
increasing function $\delta$, which is difficult to handle. Now, we have the
following properties and weaker conditions, where we can adapt some results of
the previous sections.

\begin{prop} \label{prop:deg-func}
Let $f$ be a $\delta$-degenerate function. We have the following statements:
\begin{enumerate}
\item If $\vartheta_{f,\delta} = \infty$, then $f \in \KS_{p,2}$.
\item If $\vartheta_{f,\delta} < \infty$, then
  \[
     f(s) = \sum_{\nu = 0}^{\vartheta_{f,\delta}-1} \Delta_f(\nu) \, p^\nu
            \binom{s}{\nu} + \sum_{\nu \geq \vartheta_{f,\delta}}
            \Delta'_{f,\delta}(\nu) \, p^{\delta(\nu)}
            \binom{s}{\nu}, \quad s \in \ZZ_p.
  \]
\item If $\vartheta_{f,\delta} \geq 3$ and
  $\delta(\nu) \geq 2 + \lfloor \log_p \nu \rfloor$ for $\nu \geq 3$,
  then $f \in \KS_{p,1}$. Moreover, we have for
  $s \neq t$, $s, t \in \ZZ_p$, that
  \[
     \frac{f(s)-f(t)}{s-t} \equiv 0 \pmod{p \ZZ_p}
  \]
  and
  \[
     \frac{f(s)-f(t)}{p\,(s-t)} \equiv \Delta_f \pmod{p \ZZ_p},
  \]
  where in the latter case we additionally require
  that $2 \pdiv \Delta_f(2)$ when $p=2$.
\end{enumerate}
\end{prop}

\begin{proof}
(1): Clearly by definition of $\KS_{p,2}$.
(2): This follows by comparing the Mahler expansion of $f$ up to index
$\vartheta_{f,\delta}-1$.
(3): We modify the proof of Proposition \ref{prop:f-quod-delta},
where we have to replace the term $\Delta_f(\nu) \, p^\nu$ by
$\Delta'_{f,\delta}(\nu) \, p^{\delta(\nu)}$ for $\nu \geq 3$.
Since $\vartheta_{f,\delta} \geq 3$ and using (2),
$f$ satisfies \eqref{eq:frac-binom-delta} by the following arguments.
We consider the inequality
\[
   r = \delta(\nu) - 1 - \lfloor \log_p \nu \rfloor \geq 1, \quad \nu \geq 3,
\]
which is satisfied by assumption. This gives the condition in
\eqref{eq:frac-binom-delta-2}, where we only need that $r \geq 1$.
The congruences above follow similarly as in the proof of
Proposition \ref{prop:f-quod-delta}, which imply that $f \in \KS_{p,1}$.
\end{proof}

This gives the notion to define the following classes of $\delta$-degenerate
functions.

\begin{defin}
We define the sets
\begin{align*}
  \KS^{\rm d}_{p,1} &= \{ f \in \CZ : f \text{\ is $\delta$-degenerate},
    \vartheta_{f,\delta} \geq 3, \delta(\nu) \geq 2 + \lfloor \log_p \nu \rfloor
    \text{\ for\ }\, \nu \geq 3 \}, \\
  \WKS^{\rm d}_{p,1} &= \{ f \in \KS^{\rm d}_{p,1} :
    \Delta_f \neq 0, \text{additionally\ } 2 \pdiv \Delta_f(2)
    \text{\ if\ } p=2 \}.
\end{align*}
\end{defin}

\begin{corl}
We have $\KS_{p,2} \subsetneq \KS^{\rm d}_{p,1} \subset \KS_{p,1}$ and
$\WKS_{p,2} \subsetneq \WKS^{\rm d}_{p,1} \subset \WKS_{p,1}$.
\end{corl}

\begin{proof}
Proposition \ref{prop:deg-func} shows that
$\KS^{\rm d}_{p,1} \subset \KS_{p,1}$.
Let $f \in \WKS^{\rm d}_{p,1}$. Since $\Delta_f \neq 0$, we obtain
$\norm{(f(s)-f(t))/(p\,(s-t))}_p = 1$. Thus $f \in \WKS_{p,1}$.
Define $\tilde{\delta}(\nu) = 2 + \lfloor \log_p \nu \rfloor$ for $\nu \geq 3$.
Functions of $\KS_{p,2}$ and $\WKS_{p,2}$ are also $\delta$-degenerate with
the strong property $\vartheta_{f,\delta} = \infty$ and
$\delta(\nu) = \nu \geq \tilde{\delta}(\nu)$ for $\nu \geq 3$.
But conversely, $f \in \WKS^{\rm d}_{p,1} \subset \KS^{\rm d}_{p,1}$
with $\delta(\nu) = \tilde{\delta}(\nu)$ for $\nu \geq 3$
lies not in $\KS_{p,2}$.
\end{proof}

\begin{theorem} \label{thm:ks1d-zero}
If $f \in \WKS^{\rm d}_{p,1}$ and $f(0) \in p\ZZ_p$,
then $f$ has a unique simple zero $\xi \in \ZZ_p$ and
\[
   f(s) = p \, (s-\xi) \, f^*(s), \quad \norm{f(s)}_p = \norm{p \, (s-\xi)}_p,
     \quad s \in \ZZ_p,
\]
where $f^*(s) \equiv \Delta_f \pmod{p \ZZ_p}$ and $f^* \in \CZ^*$.
\end{theorem}

\begin{proof}
Proposition \ref{prop:deg-func} shows that $f/p$ satisfies the
conditions of Proposition \ref{prop:f-lip-zero}.
\end{proof}

To compute any values$\pmod{p^n \ZZ_p}$ of a $\delta$-degenerate function,
we can use again a finite Mahler expansion.

\begin{prop} \label{prop:deg-value}
Let $f$ be a $\delta$-degenerate function. For $n \geq 1$ and $s \in \ZZ_p$,
we have
\[
   f(s) \equiv \sum_{\nu = 0}^{\eta_{f,\delta}(n)-1} \Delta'_{f,\delta}(\nu)
     \, p^{\delta(\nu)} \binom{s}{\nu} \pmod{p^n \ZZ_p}.
\]
\end{prop}

\begin{proof}
This is a consequence that $\delta(\nu) \geq n$ for
$\nu \geq \eta_{f,\delta}(n)$.
\end{proof}

We can modify Algorithm \ref{alg:zero} to compute a zero of functions
$f \in \WKS^{\rm d}_{p,1}$, where $f(0) \in p\ZZ_p$, in the following way.

\begin{algo} \label{alg:deg-zero}
Let $f \in \WKS^{\rm d}_{p,1}$, $f(0) \in p\ZZ_p$, and $\xi$ be the unique
zero of $f$. Let $n \geq 1$ be fixed. Initially set $\xi_0 = 0$ and
$\tilde{\delta} \equiv -\Delta_f^{-1} \pmod{p\ZZ_p}$.
For each step $r=1,\ldots,n$ proceed as follows. Compute
\begin{equation} \label{eq:algo-deg-sum}
   \gamma_{r-1} \equiv \sum_{\nu = 0}^{\eta_{f,\delta}(r+1)-1}
     \Delta'_{f,\delta}(\nu)
     \, p^{\delta(\nu)-1} \binom{\xi_{r-1}}{\nu} \pmod{p^r \ZZ_p}.
\end{equation}
Then $\gamma_{r-1} \in p^{r-1}\ZZ_p$.
Set $s_{r-1} \equiv \gamma_{r-1} \tilde{\delta} / p^{r-1} \pmod{p\ZZ_p}$
where $0 \leq s_{r-1} < p$. Set $\xi_r = \xi_{r-1} + s_{r-1} \, p^{r-1}$
and go to the next step while $r < n$.
Finally, $\xi_n \equiv \xi \pmod{p^n\ZZ_p}$.
\end{algo}

\begin{proof}
The function $f/p$ satisfies the conditions of Proposition
\ref{prop:f-lip-zero}, which we use to compute the zero of $f/p$.
For each step \eqref{eq:sn-fn} provides the next digit of the $p$-adic
expansion of $\xi$; to compute the term $f_{r-1}(0)$ we make use of Proposition
\ref{prop:deg-value} modified for $f/p$. Note that $\vartheta_{f,\delta} \geq 3$
and $\Delta_f(0) = f(0) \in p\ZZ_p$, so \eqref{eq:algo-deg-sum} is valid.
\end{proof}

We have already seen in the proofs of Proposition \ref{prop:ks2-cntexpl} and
Lemma \ref{lem:ks2-zero0} examples of functions, that slightly violate the
conditions to be in $\KS_{p,2}$. This can be described more precisely.
First we need some lemmas.

\begin{lemma} \label{lem:mahler-transl}
If $f \in \CZ$ has the Mahler expansion
\[
   f(s) = \sum_{\nu \geq 0} a_\nu \binom{s}{\nu}, \quad s \in \ZZ_p,
\]
with coefficients $a_\nu \in \ZZ_p$ and $|a_\nu|_p \to 0$, then
\[
   f(s+t) = \sum_{\nu \geq 0} \binom{s}{\nu}
     \sum_{k \geq 0} \binom{t}{k} a_{\nu+k}, \quad s, t \in \ZZ_p.
\]
\end{lemma}

\begin{proof}
Since $|a_\nu|_p \to 0$ and the sequence of the partial sums of the Mahler
expansion of $f$ is uniformly convergent to $f$, we can rearrange the series.
Using Vandermonde's convolution identity again, we obtain
\begin{align*}
  f(s+t) &= \sum_{\nu \geq 0} a_\nu \binom{s+t}{\nu} \\
         &= \sum_{\nu \geq 0} a_\nu \sum_{j+k=\nu}
           \binom{s}{j} \binom{t}{k} \\
         &= \sum_{\nu \geq 0} \binom{s}{\nu}
           \sum_{k \geq 0} \binom{t}{k} a_{\nu+k}
\end{align*}
after rearranging the sums.
\end{proof}

\begin{lemma} \label{lem:deg-transl}
Let $f$ be a $\delta$-degenerate function. Define $\tilde{f}(s) = f(s+t)$ for
$s \in \ZZ_p$, where $t \in \ZZ_p$ is a fixed translation. Then also
$\tilde{f}$ is a $\delta$-degenerate function.
\end{lemma}

\begin{proof}
We have to show that $\tilde{f}$ has a Mahler series, that suffices
Definition \ref{def:deg-func} regarding $\delta$.
We use Lemma \ref{lem:mahler-transl} and set
$a_\nu = \Delta'_{f,\delta}(\nu) \, p^{\delta(\nu)}$. Thus
\[
   \tilde{f}(s) = f(s+t) = \sum_{\nu \geq 0} \tilde{a}_\nu \binom{s}{\nu},
     \quad s \in \ZZ_p,
\]
where
\[
   \tilde{a}_\nu = \sum_{k \geq 0} \binom{t}{k} \Delta'_{f,\delta}(\nu+k)
     \, p^{\delta(\nu+k)}.
\]
Since $\delta$ is a monotonically increasing function, we achieve that
$\tilde{a}_\nu \in p^{\delta(\nu)} \ZZ_p$.
\end{proof}

The lemma above shows the significance, that $\delta$ has to be a
monotonically increasing function. Otherwise, the lemma does not work,
since the coefficients $\tilde{a}_\nu$ take successional values
$p^{\delta(\nu+k)}$ into account.

\begin{prop} \label{prop:deg-transl-delta}
Let $f$ be a $\delta$-degenerate function. Define $\tilde{f}(s) = f(s+t)$
for $s \in \ZZ_p$, where $t \in \ZZ_p$ is a fixed translation.
Then $\vartheta_{\tilde{f},\delta} = \vartheta_{f,\delta}$ and
\[
   \Delta_{\tilde{f}}(\nu) \equiv \Delta_f(\nu) \pmod{p\ZZ_p},
     \quad 0 \leq \nu \leq \vartheta_{f,\delta}-2.
\]
\end{prop}

\begin{proof}
We use Proposition \ref{prop:deg-func} and Lemma \ref{lem:deg-transl}.
Since $\tilde{f}$ is also a $\delta$-degenerate function,
we have $\vartheta_{f,\delta} = \vartheta_{\tilde{f},\delta}$.
Assume that $\vartheta_{f,\delta} \geq 2$. Note that
$\delta(\nu)=\nu$ for $0 \leq \nu \leq \vartheta_{f,\delta}-1$
and $\delta(\vartheta_{f,\delta})=\delta(\vartheta_{f,\delta}-1)$,
since we have a defect at index $\vartheta_{f,\delta}$ of the Mahler
expansion of $f$. This transfers to $\tilde{f}$, such that
\[
   \tilde{f}(s) \equiv \sum_{\nu = 0}^{\vartheta_{\tilde{f},\delta}-2}
     \Delta_{\tilde{f}}(\nu) \, p^\nu \binom{s}{\nu}
     \pmod{p^{\vartheta_{\tilde{f},\delta}-1}\ZZ_p},
     \quad s \in \ZZ_p.
\]
On the other side, we have for $0 \leq \nu \leq \vartheta_{f,\delta}-2$ that
\[
   \Delta_{\tilde{f}}(\nu) \, p^\nu
     = \sum_{k \geq 0} \binom{t}{k} \Delta'_{f,\delta}(\nu+k)
        \, p^{\delta(\nu+k)} \\
     = \Delta_f(\nu) \, p^\nu + \LS(p^{\nu+1}). \qedhere
\]
\end{proof}

\begin{corl} \label{corl:ks2-transl-delta}
Let $f \in \KS_{p,2}$. Define $\tilde{f}(s) = f(s+t)$
for $s \in \ZZ_p$, where $t \in \ZZ_p$ is a fixed translation.
Then
\[
   \Delta_{\tilde{f}}(\nu) \equiv \Delta_f(\nu) \pmod{p\ZZ_p},
     \quad \nu \geq 0.
\]
\end{corl}

\begin{proof}
This follows by Proposition \ref{prop:deg-transl-delta} and
choosing $\delta = \id_{\NN_0}$, so that $\vartheta_{f,\delta} = \infty$.
\end{proof}

As a result the coefficients $\Delta_f(\nu)$ of functions
$f \in \KS_{p,2}$ are invariant$\pmod{p\ZZ_p}$ under translation.
If $f$ is a $\delta$-degenerate function, then this property
is valid up to index $\vartheta_{f,\delta}-2$.

\begin{prop} \label{prop:deg-ks2-zero-0}
Assume that $p > 3$. If $f \in \KS_{p,2}$ with $f(0)=0$,
then $f(s) = ps \, g(s)$ for $s \in \ZZ_p$,
where $g \in \KS^{\rm d}_{p,1}$.
\end{prop}

\begin{proof}
By Lemma \ref{lem:ks2-zero0} and \eqref{eq:ks2-zero0-g} we have
\begin{equation} \label{eq:loc-deg-ks2-zero-0}
   g(s) = \sum_{\nu \geq 0} \frac{\Delta_f(\nu+1) \, p^\nu}{\nu+1}
          \binom{s-1}{\nu}, \quad s \in \ZZ_p.
\end{equation}
Lemma \ref{lem:deg-transl} shows that we can work with $\tilde{g}(s) = g(s+1)$
instead of $g$. Thus
\begin{equation} \label{eq:loc-deg-ks2-zero-1}
   \tilde{g}(s) = \sum_{\nu \geq 0} \Delta'_{\tilde{g},\delta}(\nu) \,
     p^{\delta(\nu)} \binom{s}{\nu}, \quad s \in \ZZ_p,
\end{equation}
where
\[
   \Delta'_{\tilde{g},\delta}(\nu) \, p^{\delta(\nu)}
     = \Delta_f(\nu+1) \, p^\nu / (\nu+1).
\]
One easily observes for $p > 3$ that
$\vartheta_{\tilde{g},\delta} \geq p-1 > 3$
and a simple counting argument shows that
$\delta(\nu) \geq 2 + \lfloor \log_p \nu \rfloor$ for $\nu \geq 3$.
Finally, $\tilde{g} \in \KS^{\rm d}_{p,1}$ and equivalently
$g \in \KS^{\rm d}_{p,1}$.
\end{proof}

\begin{corl} \label{corl:deg-ks2-zero-0}
Assume that $p > 3$. Let $f \in \KS_{p,2}$ having a zero $\xi \in \ZZ_p$.
Then $f(s) = p \, (s-\xi) \, g(s)$ for $s \in \ZZ_p$,
where $g \in \KS^{\rm d}_{p,1}$.
\end{corl}

\begin{proof}
Set $\tilde{f}(s) = f(s + \xi)$ and $\tilde{g}(s) = g(s + \xi)$.
Then $\tilde{f}(s) = ps \, \tilde{g}(s)$. Note that $\tilde{f} \in \KS_{p,2}$.
This already follows by the definition of $\KS_{p,2}$, due to the fact that
$\Dop^n f(s) \equiv 0 \pmod{p^n \ZZ_p}$ for all $s \in \ZZ_p$ and $n \geq 1$.
Applying Proposition \ref{prop:deg-ks2-zero-0}, we get
$\tilde{g} \in \KS^{\rm d}_{p,1}$ and consequently $g \in \KS^{\rm d}_{p,1}$
by Lemma \ref{lem:deg-transl}.
\end{proof}

\begin{remark*}
For $p > 3$ and $f \in \KS_{p,2}$, which has a zero $\xi \in \ZZ_p$,
we have at least that $g(s) = f(s)/p(s-\xi)$ is a function of
$\KS^{\rm d}_{p,1}$. This shows that $g$ obeys at least the Kummer congruences.
Moreover, $f$ can have two roots in $\ZZ_p$ under certain conditions as follows.
\end{remark*}

\begin{defin}
We define the set
\[
   \KS^2_{p,2} = \left\{ f \in \KS_{p,2} : \Delta_f = 0,
     \Delta_f(2) \in \ZZ^*_p, f \text{\ has a zero in\ } \ZZ_p \right\},
     \quad p > 3.
\]
\end{defin}

\begin{theorem} \label{thm:ks2-zero-2}
Assume that $p > 3$. If $f \in \KS^2_{p,2}$,
then $f$ has two zeros $\xi_1, \xi_2 \in \ZZ_p$, such that
\[
   f(s) = p^2 \, (s-\xi_1)(s-\xi_2) \, f^*(s),
     \quad \norm{f(s)}_p = \norm{p^2 \, (s-\xi_1)(s-\xi_2)}_p,
     \quad s \in \ZZ_p,
\]
where $f^*(s) \equiv \Delta_f(2)/2 \pmod{p \ZZ_p}$ and $f^* \in \CZ^*$.
\end{theorem}

\begin{proof}
By assumption $f$ has a zero in $\ZZ_p$, say $\xi_1 \in \ZZ_p$.
By Corollary \ref{corl:deg-ks2-zero-0} we then have the decomposition
$f(s) = p \, (s-\xi_1) \, g(s)$, where $g \in \KS^{\rm d}_{p,1}$.
Again set $\tilde{f}(s) = f(s + \xi_1)$ and $\tilde{g}(s) = g(s + \xi_1)$.
By Corollary \ref{corl:ks2-transl-delta} we have
\[
   \Delta_{\tilde{f}}(\nu) \equiv \Delta_f(\nu) \pmod{p\ZZ_p},
     \quad \nu \geq 0.
\]
According to \eqref{eq:loc-deg-ks2-zero-0} and \eqref{eq:loc-deg-ks2-zero-1}
we obtain
\[
   \hat{g}(s) =
   \tilde{g}(s+1) = \sum_{\nu \geq 0} \frac{\Delta_{\tilde{f}}(\nu+1)
     \, p^\nu}{\nu+1} \binom{s}{\nu}, \quad s \in \ZZ_p.
\]
As a result of Proposition \ref{prop:deg-ks2-zero-0}, we know that
$\hat{g} \in \KS^{\rm d}_{p,1}$ and $\vartheta_{\hat{g},\delta} \geq 3$.
Using Proposition \ref{prop:deg-transl-delta},
the above equation yields that
\begin{alignat*}{4}
  \Delta_g(0) &\equiv \Delta_{\hat{g}}(0) &&\equiv \Delta_{\tilde{f}}(1)
    &&\equiv \Delta_f(1) &&\equiv 0 \pmod{p\ZZ_p}, \\
  \Delta_g(1) &\equiv \Delta_{\hat{g}}(1) &&\equiv \Delta_{\tilde{f}}(2)/2
    &&\equiv \Delta_f(2)/2 &&\not\equiv 0 \pmod{p\ZZ_p}.
\end{alignat*}
Therefore $g$ satisfies the conditions of Theorem \ref{thm:ks1d-zero},
which provides
\[
   g(s) = p \, (s-\xi_2) \, g^*(s), \quad s \in \ZZ_p,
\]
where $g^*(s) \equiv \Delta_g \pmod{p \ZZ_p}$ and $g^* \in \CZ^*$.
We finally set $f^* = g^*$ and observe that
$\Delta_g \equiv \Delta_f(2)/2 \pmod{p \ZZ_p}$. This gives the result.
\end{proof}
\medskip

\section{Ring properties and products}

\begin{theorem}
The $p$-adic Kummer spaces $\KS_{p,1}$ and $\KS_{p,2}$ are commutative rings,
where $\KS^*_{p,1}$ and $\KS^*_{p,2}$, as defined in Definition
\ref{def:decomp-spaces-0}, are their unit groups, respectively.
\end{theorem}

\begin{proof}
Case $\KS_{p,1}$: Multiplication:
Let $f, g \in \KS_{p,1}$ and $s, t \in \ZZ_p$.
For $s \equiv t \pmod{p^n \ZZ_p}$ we get by Definition \ref{def:kummer-space}
that $f(s)g(s) \equiv f(t)g(t) \pmod{p^{n+1} \ZZ_p}$.
Thus $f \cdot g \in \KS_{p,1}$, where $(f \cdot g)(s)=f(s)g(s)$.
Units: Since $f(0) \equiv f(s) \pmod{p\ZZ_p}$ for $s \in \ZZ_p$, we have
$f^{-1}(s) \in \ZZ^*_p$ if and only if $f(0) \in \ZZ^*_p$.
Let $f \in \KS^*_{p,1}$, then we also have
$f^{-1}(s) \equiv f^{-1}(t) \pmod{p^{n+1} \ZZ_p}$
when $s \equiv t \pmod{p^n \ZZ_p}$. Hence $f^{-1} \in \KS^*_{p,1}$.

Case $\KS_{p,2}$: Multiplication:
Let $f, g \in \KS_{p,2}$ and $w = f \cdot g$, where $w(s)=f(s)g(s)$ for
$s \in \ZZ_p$. The product of the Mahler expansions of $f$ and $g$ yields
\begin{equation} \label{eq:prod-func}
\begin{split}
   w(s) &= \sum_{\nu \geq 0} \Delta_f(\nu) \, p^\nu \binom{s}{\nu} \, \cdot \,
     \sum_{\nu \geq 0} \Delta_g(\nu) \, p^\nu \binom{s}{\nu} \\
     &= \sum_{n \geq 0} p^n \sum_{n = j+k}
        \Delta_f(j) \, \Delta_g(k) \, \binom{s}{j} \binom{s}{k}.
\end{split}
\end{equation}
Now, for a fixed $n$, the polynomials $\binom{s}{j} \binom{s}{k}$ above have
always degree $n$. By Lemma \ref{lem:delta-poly} we obtain that
$\Dop^n \, w(s) \equiv 0 \pmod{p^n \ZZ_p}$. Since this is valid for all
$n \geq 1$, we finally deduce that $w \in \KS_{p,2}$.
Units: Again $f(0) \equiv f(s) \pmod{p\ZZ_p}$ for $s \in \ZZ_p$ and
$f^{-1}(s) \in \ZZ^*_p$ if and only if $f(0) \in \ZZ^*_p$.
Let $f \in \KS^*_{p,2}$. We have to show that $f^{-1} \in \KS^*_{p,2}$
and necessarily that $\Dop^n \, f^{-1}(s) \equiv 0 \pmod{p^n \ZZ_p}$ for all
$n \geq 1$. We will construct a sequence $(f_n)_{n \geq 1}$ of functions,
such that $f_n \equiv f^{-1} \pmod{p^n\ZZ_p}$ and consequently
$\lim_{n \to \infty} f_n = f^{-1}$. Let $\eulerphi$ be Euler's totient
function, then the Euler--Fermat's theorem reads
\begin{equation} \label{eq:ext-fermat}
   a^{\eulerphi(p^r)} \equiv 1 \pmod{p^r\ZZ_p}
     \quad \text{for\ } a \in \ZZ^*_p, r \geq 1.
\end{equation}
Define $f_n = f^{\eulerphi(p^n)-1}$ for $n \geq 1$. Then $f_n \in \KS^*_{p,2}$
and we have
\[
   \normf{f}_p = \normf{f^{-1}}_p = \normf{f_n}_p = 1.
\]
Using \eqref{eq:ext-fermat} we obtain for $n \geq r \geq 1$ and
$s \in \ZZ_p$ that
\[
   f_n(s) \equiv f^{-1}(s) \pmod{p^r\ZZ_p}
\]
and consequently
\[
   0 \equiv \Dop^r f_n(s) \equiv \Dop^r f^{-1}(s) \pmod{p^r\ZZ_p}.
\]
Thus $\normf{f_n - f^{-1}}_p \leq p^{-n}$ and $\normf{f_n - f^{-1}}_p \to 0$
as $n \to \infty$. Finally $f^{-1} \in \KS^*_{p,2}$.
\smallskip

All other ring axioms are also valid, since addition and multiplication
are induced by standard operations.
\end{proof}

\begin{remark*}
Interestingly, the condition $2 \pdiv \Delta_f(2)$, as required in Definition
\ref{def:kummer-space} for functions $f \in \WKS_{p,2}$ in case $p=2$, is
preserved under ring operations among functions having this condition. We have
to show that this is compatible with multiplication and inverse mapping, while
it is trivial for addition. We use the same notation from above and adapt the
proof. We show that the condition transfers to $\Delta_w(2)$ and
$\Delta_{f^{-1}}(2)$, respectively.
Multiplication: We have $\Delta_f(2) \equiv \Delta_g(2) \equiv 0 \pmod{2\ZZ_2}$
by assumption. Evaluating $\Delta_w(2)$ in
\eqref{eq:prod-func} gives
\[
   \Delta_w(2) \equiv 2^{-2} \, \Dop^2 w(0)
     \equiv \Delta_f(0) \Delta_g(2) + 2 \Delta_f(1) \Delta_g(1)
     + \Delta_f(2) \Delta_g(0) \equiv 0 \pmod{2\ZZ_2}.
\]
Inverse mapping:
The condition $2 \pdiv \Delta_f(2)$ is equivalent to
$\Dop^2 f(0) \equiv 0 \pmod{2^3\ZZ_2}$.
Since $f(s) \in \ZZ^*_2$, the values of $f$ modulo $8$ in question are
in $\{ 1, 3, 5, 7\}$, which are inverse to themselves. Thus
$\Dop^2 f^{-1}(0) \equiv \Dop^2 f(0) \equiv 0 \pmod{2^3\ZZ_2}$
and consequently $2 \pdiv \Delta_{f^{-1}}(2)$.
\end{remark*}

\begin{remark*}
Since $\KS_{p,2} \subset \KS_{p,1}$, $\KS^*_{p,2}$ is a subgroup of
$\KS^*_{p,1}$. Moreover $\ZZ^*_p$, identified as the group of constant
functions, is a subgroup of $\KS^*_{p,1}$ and $\KS^*_{p,2}$.
\end{remark*}

\begin{prop} \label{prop:exp-func}
Let $\SU_p = 1 + p\ZZ_p$.
The group of functions
\[
   \SE_p = \{ f_a(s) = a^s : s \in \ZZ_p, \, a \in \SU_p \}
\]
is a subgroup of $\KS^*_{p,2}$.
\end{prop}

\begin{proof}
Obviously, we have $\SU_p \cong \, \SE_p$ as multiplicative groups.
Following \cite[Ch.~4.2, p.~173]{Robert:2000}, the binomial expansion can
be extended to a uniformly convergent Mahler series for $s \in \ZZ_p$:
\[
   f_a(s) = (a-1+1)^s = \sum_{\nu \geq 0} (a-1)^\nu \binom{s}{\nu}.
\]
Observing that $\Dop^\nu f_a(0) = (a-1)^\nu \in p^\nu \ZZ_p$,
we can also write
\[
   f_a(s) = \sum_{\nu \geq 0} \Delta_{f_a}(\nu) \, p^\nu \binom{s}{\nu}.
\]
Since $f_a(0) = 1$, it follows that $f_a \in \KS^*_{p,2}$.
\end{proof}

\begin{corl}
If $a \in \SU_p$, then the equation
\[
   a^\tau = \tau
\]
is uniquely solvable with $\tau \in \SU_p$,
which is effectively computable by Algorithm \ref{alg:fixed-point}.
Moreover, the map
\[
   \kappa: \SU_p \to \SU_p, \quad a \mapsto \tau,
\]
is injective.
\end{corl}

\begin{proof}
Since the corresponding function $f_a \in \KS_{p,2}$,
Theorem \ref{thm:ks2-zero-fixpnt} asserts that $f_a$ has a fixed point.
Thus, the equation above is uniquely solvable.
From $1 = f_a(0) \equiv f_a(\tau) = \tau \pmod{p\ZZ_p}$,
we deduce that $\tau \in \SU_p$. Now we show that the map $\kappa$
is injective. Let $a,b \in \SU_p$, where $a \neq b$. Assume that
$\kappa(a) = \kappa(b) = \tau$. Thus we get $(ab^{-1})^\tau = 1$, which
implies that $a=b$ by the following Lemma \ref{lem:exp-tau-1}. Contradiction.
\end{proof}

\begin{lemma} \label{lem:exp-tau-1}
Let $a, \tau \in \SU_p$, then
\[
   a^\tau = 1 \quad \Longleftrightarrow \quad a = 1.
\]
\end{lemma}

\begin{proof}
If $a=1$, then we are ready.
Assume that $a=1+p^r\alpha$ with some $\alpha \in \ZZ_p^*$ and $r \geq 1$.
Then
\[
   0 = a^\tau - 1 = \sum_{\nu \geq 1} p^{\nu r} \alpha^\nu \binom{\tau}{\nu}
     \equiv p^r \alpha \tau \pmod{p^{2r}\ZZ_p}.
\]
Since $\tau \in \SU_p$, we obtain $\alpha \equiv 0 \pmod{p^r\ZZ_p}$.
Contradiction.
\end{proof}

Using the ring properties of $\KS_{p,2}$, we can give the following
applications.

\begin{prop}
Let $f \in \KS_{p,2}$, resp., $f \in \KS^*_{p,2}$
and $r \in \NN$, resp., $r \in \ZZ$. For $s \in \ZZ_p$ and $n \geq 0$ we have
\[
   \left( \sum_{\nu = 0}^{n} f(\nu) \binom{s}{\nu} \binom{n-s}{n-\nu} \right)^r
     \equiv \,\, \sum_{\nu = 0}^{n} f(\nu)^r \binom{s}{\nu} \binom{n-s}{n-\nu}
     \pmod{p^{n+1} \ZZ_p}.
\]
\end{prop}

\begin{proof}
Since $\KS_{p,2}$ is a ring, we have $f^r \in \KS_{p,2}$, resp.,
$f^r \in \KS^*_{p,2}$. By Proposition \ref{prop:kummer-congr-value} both
sides of the congruence above are congruent to $f(s)^r$.
\end{proof}

\begin{prop} \label{prop:conv-ks2}
Let $f, g \in \KS_{p,2}$, $s \in \ZZ_p$, and $n \geq 1$. Define the convolution
\[
   \Dop^n (f \star g)(s) = \sum_{\nu=0}^n \binom{n}{\nu}
     (-1)^{n-\nu} f(s+\nu) g(s+n-\nu).
\]
Then $\Dop^n (f \star g)(s) \equiv 0 \pmod{p^n\ZZ_p}$.
\end{prop}

\begin{proof}
Define $\hat{g} = g \circ \lambda$, where $\lambda(s) = 2\hat{s} + n - s$
with a fixed $\hat{s} \in \ZZ_p$.
Then $\hat{g}(s + \nu) = g(2\hat{s} + n - s - \nu)$.
By Lemma \ref{lem:f-comp-lin-ks2} we know that $\hat{g} \in \KS_{p,2}$.
Thus $f \cdot \hat{g} \in \KS_{p,2}$. Now, let $s$ be fixed and choose
$\hat{s} = s$. We finally obtain in this case that
\[
   \Dop^n (f \star g)(s) = \Dop^n (f \cdot \hat{g})(s)
     \equiv 0 \pmod{p^n\ZZ_p}. \qedhere
\]
\end{proof}

Now, we shall examine properties of products of functions of $\KS_{p,2}$.

\begin{prop} \label{prop:prod-f-delta}
Let $n \geq 1$ and
\[
   F = \prod_{\nu=1}^n f_\nu , \quad f_\nu \in \KS_{p,2}.
\]
Then
\[
   \Delta_F(0) = \prod_{\nu=1}^n \Delta_{f_\nu}(0)
\]
and
\[
   \Delta_F(k) \equiv \sum_{\nu_1+\cdots+\nu_n=k} \binom{k}{\nu_1,\ldots,\nu_n}
     \Delta_{f_1}(\nu_1) \cdots \Delta_{f_n}(\nu_n) \pmod{p\ZZ_p}
\]
for $k \geq 1$.
\end{prop}

\begin{proof}
In view of \eqref{eq:prod-func}, we obtain more generally that
\[
   F(s) = \sum_{k \geq 0} p^k \sum_{\nu_1+\cdots+\nu_n=k} \!\!\!\!
     \Delta_{f_1}(\nu_1) \cdots \Delta_{f_n}(\nu_n)
     \binom{s}{\nu_1} \cdots \binom{s}{\nu_n}.
\]
Since $F \in \KS_{p,2}$, we have the Mahler expansion
\[
   F(s) = \sum_{k \geq 0} \Delta_F(k) \, p^k \binom{s}{k}.
\]
For $k=0$ we have $\Delta_F(0) = F(0) = f_1(0) \cdots f_n(0)
= \Delta_{f_1}(0) \cdots \Delta_{f_n}(0)$. Let $k \geq 1$ be fixed, then
$\pi(s) = (s)_{\nu_1} \cdots (s)_{\nu_n}$ is a monic polynomial of degree $k$.
Note that $\Dop (s)_m = m \, (s)_{m-1}$ and $\Dop^k \pi(0) = k!$ by Lemma
\ref{lem:delta-poly}, but in general $\Dop^r \pi(0) \neq 0$ for $k > r \geq 1$,
because $\Dop^r \pi(s)$ can have a constant term. Evaluating $\Dop^k F(0)/p^k$
we obtain
\begin{align*}
   \Delta_F(k) = & \sum_{\nu_1+\cdots+\nu_n=k} \binom{k}{\nu_1,\ldots,\nu_n}
     \Delta_{f_1}(\nu_1) \cdots \Delta_{f_n}(\nu_n) \\
     & + \sum_{k' > k} p^{k'-k} \!\!\!\!
     \sum_{\nu_1+\cdots+\nu_n=k'} \!\!\!\!
     \Delta_{f_1}(\nu_1) \cdots \Delta_{f_n}(\nu_n) \,
     \Dop^k \, \binom{s}{\nu_1} \cdots \binom{s}{\nu_n} \valueat{s=0},
\end{align*}
where the second sum converges on $\ZZ_p$ and vanishes$\pmod{p\ZZ_p}$.
\end{proof}

\begin{defin}
We define for functions $f \in \KS_{p,2}$ the following parameters
\begin{alignat*}{2}
   \lambda_f &= \min_{\nu \geq 0} \{ \nu : \Delta_f(\nu) \in \ZZ^*_p \},
     &\quad  \mu_f &= \max_{\nu \geq 0} \{ \nu : f/p^\nu \in \KS_{p,2} \} \\
\intertext{and we set $\lambda_f = \mu_f = \infty$ in case $f=0$.
   We further define}
   \KS^\prime_{p,2} &= \{ f \in \KS_{p,2} : \lambda_f < \infty \},
     &\quad \KS^{\prime\prime}_{p,2} &= \{ f/ p^{\mu_f} : f \in \KS_{p,2} \}.
\end{alignat*}
\end{defin}

\begin{lemma} \label{lem:decomp-ks2}
We have $\KS^\prime_{p,2} = \KS^{\prime\prime}_{p,2}$ and
the decomposition $\KS_{p,2} = \KS^\prime_{p,2} \times p^{\NN_0}$.
\end{lemma}

\begin{proof}
We have $\ZZ_p \subset \KS_{p,2}$, since constant functions are in $\KS_{p,2}$.
Note that $1 \in \KS^\prime_{p,2}$ and $0, 1 \in p^{\NN_0}$. Let
$f \in \KS_{p,2}$ where $f \neq 0$. In case $\lambda_f = \infty$ we can split
the prime $p$ from $f$ to get $f/p \cdot p$.
Then $\Delta_{f/p}(\nu) = \Delta_f(\nu)/p$ for all $\nu \geq 1$ and
$f/p \in \KS_{p,2}$. This procedure can be finitely repeated, say $r$ times,
until $f/p^r \in \KS^\prime_{p,2}$ and we are ready. In case
$\lambda_f < \infty$ set $r=0$. By construction $r$ is maximal, so we get
$r = \mu_f$ and $f/p^{\mu_f} \in \KS^\prime_{p,2}$. This shows the
decomposition $\KS_{p,2} = \KS^\prime_{p,2} \times p^{\NN_0}$. Now,
by the same arguments we achieve that
$\KS^\prime_{p,2} = \KS^{\prime\prime}_{p,2}$, since $r = \mu_f$ is maximal.
The case $f=0$ yields the term $0/0$, which we define to be 1.
\end{proof}

\begin{prop} \label{prop:prod-f-prime-delta}
Let $n \geq 1$ and
\[
   F = \prod_{\nu=1}^n f_\nu , \quad f_\nu \in \KS^\prime_{p,2}.
\]
Then
\[
   \Delta_F(m) \equiv \binom{m}{\lambda_{f_1},\ldots,\lambda_{f_n}}
     \prod_{\nu=1}^n \Delta_{f_\nu} ( \lambda_{f_\nu} ) \pmod{p\ZZ_p},
\]
where
\[
   \lambda_F \geq m = \sum_{\nu=1}^n \lambda_{f_\nu}.
\]
Moreover, $\lambda_F > m$ if and only if
\[
   \binom{m}{\lambda_{f_1},\ldots,\lambda_{f_n}} \in p\ZZ_p.
\]
\end{prop}

\begin{proof}
From Proposition \ref{prop:prod-f-delta} we have for $k \geq 0$ that
\[
   \Delta_F(k) \equiv \sum_{\nu_1+\cdots+\nu_n=k} \binom{k}{\nu_1,\ldots,\nu_n}
     \Delta_{f_1}(\nu_1) \cdots \Delta_{f_n}(\nu_n) \pmod{p\ZZ_p}.
\]
Case $k < m$: Since $\Delta_{f_j}(\nu) \in p\ZZ_p$ for
$\lambda_{f_j} > \nu \geq 0$ and $\Delta_{f_j}(\lambda_{f_j}) \in \ZZ^*_p$ for
$j \in \{ 1,\ldots, n \}$, we observe that
$\Delta_F(k) \in p\ZZ_p$ for $m > k \geq 0$. Case $k=m$:
All terms of the sum, except for the term where
$\nu_1 = \lambda_{f_1}, \ldots, \nu_n = \lambda_{f_n}$, vanish$\pmod{p\ZZ_p}$.
This gives the proposed formula for $\Delta_F(m)$.
Let $b = \binom{m}{\lambda_{f_1},\ldots,\lambda_{f_n}}$.
If $b \in \ZZ^*_p$, then $\lambda_F = m$. Otherwise $b \in p\ZZ_p$
implies that $\lambda_F > m$.
\end{proof}

\begin{corl} \label{corl:lambda-inv}
Let $f \in \KS^\prime_{p,2}$, $u \in \KS^*_{p,2}$, and $g=fu$. The parameter
$\lambda_f$ is invariant under multiplication of $f$ and units $u$:
\[
   \lambda_g = \lambda_f, \quad \Delta_g(\lambda_g)
     \equiv \Delta_f(\lambda_f) \, u(0) \pmod{p\ZZ_p}.
\]
\end{corl}

\begin{proof}
Using Proposition \ref{prop:prod-f-prime-delta},
we get $m = \lambda_f + \lambda_u = \lambda_f$ and
\[
   \Delta_g(m) \equiv \binom{m}{m} \Delta_f(\lambda_f) \Delta_u(\lambda_u)
     \equiv \Delta_f(\lambda_f) \, u(0) \not\equiv 0 \pmod{p\ZZ_p}.
\]
Thus $\lambda_g = m = \lambda_f$.
\end{proof}

\begin{theorem} \label{thm:poly-f-ks2}
Let $n \geq 1$ and
\[
    F = \prod_{\nu=1}^n f_\nu , \quad f_\nu \in \WKS^0_{p,2}.
\]
Then
\[
   F(s) = p^n \prod_{\nu=1}^n (s-\xi_\nu ) \cdot F^*(s),
     \quad s \in \ZZ_p,
\]
where $\xi_\nu \in \ZZ_p$ is the zero of $f_\nu$. Moreover,
\[
   F^*(s) \equiv \prod_{\nu=1}^n \Delta_{f_\nu} \pmod{p\ZZ_p},
     \quad F^* \in \CZ^*, \quad s \in \ZZ_p,
\]
and
\[
   \Delta_F(n) \equiv n! \, \prod_{\nu=1}^n \Delta_{f_\nu} \pmod{p\ZZ_p}.
\]
Moreover, $\ord_p \Delta_F(\nu) \geq n - \nu$ for $\nu=0,\dots,n$.
If $n < p$, then $\lambda_F = n$, otherwise $\lambda_F > n$.
\end{theorem}

\begin{proof}
By Theorem \ref{thm:ks2-zero-fixpnt} we have the following representation for a
function $f_\nu \in \WKS^0_{p,2}$: $f_\nu(s) = p \, (s-\xi_\nu) \, f_\nu^*(s)$,
where $f_\nu^*(s) \equiv \Delta_{f_\nu} \not\equiv 0 \pmod{p \ZZ_p}$.
Thus the product representation of $F$ and $F^*$ follows easily. Since
$\lambda_{f_\nu} = 1$ we get from Proposition \ref{prop:prod-f-prime-delta}
a simplified formula
$\Delta_F(n) \equiv n! \, \prod_{\nu=1}^n \Delta_{f_\nu} \pmod{p\ZZ_p}$.
Hence $\lambda_F = n$ if $n < p$. For $n \geq p$ we obtain
$\Delta_F(n) \equiv 0 \pmod{p\ZZ_p}$, which implies that $\lambda_F > n$
in that case. Now, $\ord_p \Delta_F(\nu) \geq n - \nu$ for $\nu=0,\dots,n$
follows by
$\Delta_F(\nu) = \Dop^\nu F(0)/p^\nu = p^{n-\nu} \, \Dop^\nu \tilde{F}(0)$,
where $\tilde{F} = F/p^n$ is a function on $\ZZ_p$ as seen above.
\end{proof}

\begin{defin}
We define the class of monic polynomial like functions by
\[
   \KS^{\rm s}_{p,2} = \left\{ f \in \KS_{p,2} :
     f = \prod_{\nu=1}^n f_\nu, \, n \geq 1, \, f_\nu \in \WKS^0_{p,2} \right\}
\]
having all roots in $\ZZ_p$.
\end{defin}

We get an analogue to the $p$-adic Weierstrass Preparation Theorem.

\begin{corl} \label{corl:decomp-poly}
If $f \in \KS^{\rm s}_{p,2}$ and $\lambda_f < p$, then
we have the decomposition
\[
   f = p^{\lambda_f} \times \pi_f \times u,
\]
where $\pi_f$ is a monic polynomial of degree $\lambda_f$, that
splits over $\ZZ_p$, and $u \in \CZ^*$.
\end{corl}

\begin{proof}
This follows from Theorem \ref{thm:poly-f-ks2}, since $\lambda_f = n < p$.
\end{proof}

Again, we have a close relation between the zeros and the fixed point of
$f \in \KS^{\rm s}_{p,2}$, which extends Lemma \ref{lem:rel-zero-fixpnt}.

\begin{corl}
If $f \in \KS^{\rm s}_{p,2}$ where $f(0) \neq 0$ and $n = \lambda_f < p$, then
\[
   (-1)^n \frac{\Delta_f(n)}{n!} \equiv \frac{\tau}{p^n \prod_{\nu=1}^n \xi_\nu}
     \pmod{p\ZZ_p}, \quad f(0)/\tau \equiv 1 \pmod{p\ZZ_p},
\]
and
\[
   \ord_p f(0) = \ord_p \tau = n + \sum_{\nu = 1}^n \ord_p \xi_\nu,
\]
where $\tau$ is the fixed point and $\xi_\nu$ are the zeros of $f$.
\end{corl}

\begin{proof}
We have excluded the case $f(0)=0$ which implies $\tau = 0$ and vice versa,
so $\tau \neq 0$ and $f(0) \neq 0$. Since $n < p$ we obtain by
Theorem \ref{thm:poly-f-ks2} that
\[
  f(0) = (-1)^n f^*(0) \, p^n \prod_{\nu=1}^n \xi_\nu, \quad
  \tau = f^*(\tau) \, p^n \prod_{\nu=1}^n (\tau - \xi_\nu),
\]
where $f^*(s) \equiv \Delta_f(n)/n! \not\equiv 0 \pmod{p\ZZ_p}$ for $s \in \ZZ_p$.
Thus
\[
   f^*(\tau)^{-1} \equiv \frac{p^n \prod_{\nu=1}^n (\tau - \xi_\nu)}{\tau}
     \equiv (-1)^n p^n \prod_{\nu=1}^n \xi_\nu \, \Big/ \, \tau \pmod{p\ZZ_p},
\]
where we have used the following argument. Expanding the product above,
we get summands $s_j$ which have a factor $\tau$. For those terms we see that
$p^n s_j / \tau$ vanishes$\pmod{p\ZZ_p}$. Consequently, it only remains the product
over the zeros as given above. The rest follows easily.
\end{proof}

Functions of $\KS^{\rm s}_{p,2}$ have a controlled but unbounded growth when
viewed in the $p$-adic norm via $\norm{\cdot}_p^{-1}$, since all roots lie in
$\ZZ_p$. We can also consider arbitrary monic polynomials, where the roots may
lie in some finite extension of $\QQ_p$.

\begin{defin} \label{def:ks2-monic}
We define the class of monic polynomial like functions by
\begin{align*}
   \KS^{\rm m}_{p,2} = \{ & f \in \KS_{p,2} :
     f = p^n \, \pi_n \, u, \, \text{monic\ } \pi_n \in \ZZ_p[s], \\
     & \deg \pi_n = n \geq 1, \, u \in \KS^*_{p,2} \}.
\end{align*}
\end{defin}

\begin{prop} \label{prop:ks2-monic}
If $f \in \KS^{\rm m}_{p,2}$ where $f = p^n \, \pi_n \, u$,
then $\lambda_f = n$ in case $n < p$, otherwise $\lambda_f > n$.
Moreover, if $n \geq 2$ and $\pi_n$ is irreducible over $\ZZ_p[s]$, then
there exists a lower bound
\[
   \norm{f(s)}_p \geq p^{-c}, \quad s \in \ZZ_p,
\]
with some constant $c$ where $n \leq c < \infty$ depending on $f$.
\end{prop}

\begin{proof}
First we look at $\tilde{f} = p^n \, \pi_n$. Since we have
$\Dop^n \pi_n(0) = n!$, we deduce that $\lambda_{\tilde{f}} = n$ in case
$n < p$, otherwise $\lambda_{\tilde{f}} > n$.
This property transfers to $\lambda_f$ by Corollary \ref{corl:lambda-inv}.
Now, let $n \geq 2$ and $\pi_n$ be irreducible over $\ZZ_p[s]$. We have
$\norm{f(s)}_p = p^{-n} \norm{\pi_n(s)}_p$.
Assume that $\min_{s \in \ZZ_p} \norm{f(s)}_p$ is unbounded.
Since $\ZZ_p$ is compact and has a discrete valuation,
we would get $\norm{f(s')}_p = 0$ for some $s' \in \ZZ_p$.
This implies that $\pi_n(s') = 0$ and $s'$ is a root of $\pi_n$ in $\ZZ_p$.
This gives a contradiction.
\end{proof}

\begin{remark*}
Products of functions $f(s)=p\,(s-\xi) \, u(s)$, where $\xi \in \ZZ_p$
and $u \in \KS^*_{p,2}$, are in $\KS^{\rm s}_{p,2} \cap \, \KS^{\rm m}_{p,2}$.
But we have $\KS^{\rm s}_{p,2} \not\subset \KS^{\rm m}_{p,2}$
as a consequence of Proposition \ref{prop:ks2-cntexpl}.
\end{remark*}

At last, we construct a class of functions that are constant
regarding the $p$-adic norm.

\begin{defin} \label{def:ks2-const}
We define the class of constant functions regarding the $p$-adic norm by
\begin{align*}
   \KS^{\rm c}_{p,2} = \{ & f \in \KS_{p,2} : \ord_p f(0) = n < \lambda_f, \\
     & \ord_p \Delta_f(\nu) > n-\nu \text{\ for\ } \nu = 1,\ldots,n \}.
\end{align*}
\end{defin}

\begin{prop} \label{prop:ks2-const}
If $f \in \KS^{\rm c}_{p,2}$, then $\norm{f(s)}_p = p^{-n}$ for
$s \in \ZZ_p$, where $n = \ord_p f(0) > 0$.
\end{prop}

\begin{proof}
Let $\ord_p f(0) = n$. Since $\lambda_f > n$ by definition, the case $n=0$
is not possible, so $n \geq 1$. Set
$\Delta_f^\prime(\nu) = \Delta_f(\nu) / p^{n+1-\nu} \in \ZZ_p$
for $\nu = 1,\ldots,n$. Then we get
\[
   f(s) = f(0) + \sum_{\nu =1}^{n} \Delta_f^\prime(\nu) \, p^{n+1}
     \binom{s}{\nu}
     + \sum_{\nu > n} \Delta_f(\nu) \, p^\nu \binom{s}{\nu}.
\]
Thus $f(s) = f(0) + \LS(p^{n+1})$. Since $\ord_p f(0) = n$,
we have $p^{n+1} \notdiv f(s)$ and consequently $\norm{f(s)}_p = p^{-n}$.
\end{proof}

\begin{tbl} Classification of $\KS_{p,2}$.
\begin{center}
\begin{tabular}{|c|c|c|c|} \hline
  $f \in$ & $\lambda_f$ & $\ord_p f(0)$ & $\norm{f(s)}_p$ \\ \hline\hline
  $\KS^*_{p,2}$   & 0 & 0 & 1 \\\hline
  $\WKS^0_{p,2}$ & 1 & $\geq 1$ &
    $\norm{p \, (s-\xi)}_p, \, \xi \in \ZZ_p$ \\\hline
  $\KS^2_{p,2}$ & 2 & $\geq 2$ &
    $\norm{p^2 \, (s-\xi_1)(s-\xi_2)}_p, \, \xi_1,\xi_2 \in \ZZ_p$ \\\hline
  $\KS^{\rm s}_{p,2}$ & $\lambda_f \geq n$ & $\geq n$ &
    $\prod_{\nu=1}^n \norm{p \, (s-\xi_\nu)}_p, \, \xi_\nu \in \ZZ_p$ \\\hline
  $\KS^{\rm m}_{p,2}$ & $\lambda_f \geq n$ & $\geq n$ &
    $\norm{p^n \, \pi_n(s)}_p, \, \deg \pi_n = n$ \\\hline
  $\KS^{\rm c}_{p,2}$ & $\lambda_f > n > 0$ & $n$ & $p^{-n}$ \\\hline
\end{tabular}
\end{center}

\end{tbl}
\medskip

Functions $f \in \KS_{p,2}$, that have the property $\ord_p f(0) = 1$, can be
described as follows. As a result, such functions, not being in $\WKS^0_{p,2}$,
are constant regarding the $p$-adic norm. We exclude the case $p=2$, since the
additional condition $2 \pdiv \Delta_f(2)$ makes some difficulties.

\begin{prop} \label{prop:ks2-f-ord-1}
Let $p \geq 3$. If $f \in \KS_{p,2}$ with $\ord_p f(0) = 1$, then $f$ has
exactly one of the following forms:
\begin{center}
\begin{tabular}{|c|c|c|c|} \hline
  $f \in$ & $\mu_f$ & $\lambda_f$ & $\norm{f(s)}_p$, $s \in \ZZ_p$ \\ \hline\hline
  $\WKS^0_{p,2}$ & $0$ & $1$ & $\norm{p \, (s-\xi)}_p$, $\xi \in \ZZ^*_p$ \\\hline
  $\KS^{\rm c}_{p,2}$ & $0$ & $\geq 2$ & $p^{-1}$ \\\hline
  $p\KS^*_{p,2}$ & $1$ & $\infty$ & $p^{-1}$ \\\hline
\end{tabular}
\end{center}
\end{prop}

\begin{proof}
First, if we have $\mu_f > 0$, then only $f \in p\KS^*_{p,2}$ is possible,
since $\ord_p f(0) = 1$. Now, we can assume that $\mu_f = 0$ and
$1 \leq \lambda_f < \infty$. The Mahler expansion shows that
\[
   f(s) \equiv f(0) + \Delta_f(1) \, p s \pmod{p^2 \ZZ_p}, \quad s \in \ZZ_p.
\]
We have the cases $\Delta_f(1) \in \ZZ_p^*$ and $\Delta_f(1) \in p\ZZ_p$.
The first case implies that $\lambda_f = 1$ and $f \in \WKS^0_{p,2}$.
By Theorem \ref{thm:ks2-zero-fixpnt} we then get
$\norm{f(s)}_p = \norm{p \, (s-\xi)}_p$ with some $\xi \in \ZZ_p$.
Since $\norm{f(0)}_p = \norm{p \, \xi}_p = p^{-1}$, it even follows that
$\xi \in \ZZ^*_p$. The second case provides that $\lambda_f \geq 2$ and
$f$ suffices the conditions of Definition \ref{def:ks2-const} and
Proposition \ref{prop:ks2-const} with $n=1$.
\end{proof}

We shall also show the more complicated case $\ord_p f(0) = 2$ of functions
$f \in \KS_{p,2}$. Since we have a decomposition of $\KS_{p,2}$ by Lemma
\ref{lem:decomp-ks2}, we only consider those cases where $\mu_f = 0$ to
simplify the results.

\begin{prop} \label{prop:ks2-f-ord-2}
Let $p \geq 3$. If $f \in \KS^\prime_{p,2}$ with $\ord_p f(0) = 2$,
then $f$ has one of the following forms
\begin{center}
\begin{tabular}{|c|c|c|} \hline
  $f \in$ & $\lambda_f$ & $\norm{f(s)}_p$, $s \in \ZZ_p$ \\ \hline\hline
  $\WKS^0_{p,2}$ & $1$ & $\norm{p \, (s-\xi)}_p$, $\xi \in \ZZ_p$,
    $\ord_p \xi = 1$ \\\hline
  $\KS^2_{p,2}, $ & $2$ & $\norm{p^2 \, (s-\xi_1)(s-\xi_2)}_p$,
    $\xi_1,\xi_2 \in \ZZ^*_p$, \\
  $\KS^{\rm s}_{p,2}$ & & $p > 3$ if $f \in \KS^2_{p,2}$  \\\hline
  $\KS^{\rm m}_{p,2}$ & $2$ & $\norm{p^2 \, \pi_2(s)}_p \geq p^{-c}$,
    $\pi_2(0) \in \ZZ^*_p$, \\
    & & $\pi_2$ irreducible, $c \geq 2$. \\ \hline
  $\KS^{\rm c}_{p,2}$ & $\geq 3$ & $p^{-2}$ \\\hline
\end{tabular}
\end{center}
or $f$ has the behavior that $\norm{f(s)}_p \leq p^{-2}$ for $s \in \ZZ_p$.
\end{prop}

\begin{proof}
The Mahler expansion provides that
\[
   f(s) \equiv f(0) + \Delta_f(1) \, p s + \Delta_f(2) \, p^2 \binom{s}{2}
     \pmod{p^3 \ZZ_p}, \quad s \in \ZZ_p.
\]
Again, the case $\lambda_f = 1$ yields $\norm{f(s)}_p = \norm{p \, (s-\xi)}_p$
with $\ord_p \xi = 1$, caused by $\norm{f(0)}_p = \norm{p \, \xi}_p = p^{-2}$.
The case $\lambda_f \geq 2$ implies that $\Delta_f(1) \in p\ZZ_p$ and we know
at least that $\norm{f(s)}_p \leq p^{-2}$ for $s \in \ZZ_p$.
Due to $\ord_p f(0) = 2$, we obtain the supplementary classification as above
by Theorem \ref{thm:ks2-zero-2}, Corollary \ref{corl:decomp-poly}, and
Propositions \ref{prop:ks2-monic} and \ref{prop:ks2-const}.
\end{proof}

\section{$p$-adic interpolation of $L$-functions}

\begin{defin}
Let the function $f: \NN_0 \to \ZZ_p$ satisfy the Kummer type congruences
\[
   \Dop^n f(0) \equiv 0 \pmod{p^n \ZZ_p} \quad \text{for all\ } n \geq 0.
\]
Then we call $f$ a Kummer function.
\end{defin}

The standard way to extend a function, which is defined for nonnegative
integer arguments, to the domain $\ZZ_p$ is the following, cf.
\cite[Ch.~2]{Koblitz:1996}. A function $f: \NN_0 \to \ZZ_p$ can be uniquely
extended to a function $\tilde{f}: \ZZ_p \to \ZZ_p$, which interpolates values
$\tilde{f}(n) = f(n)$ for nonnegative integers $n$. Then define for
$s \in \ZZ_p$ that $\tilde{f}(s) = \lim_{t_\nu \to s} f(t_\nu)$ for any sequence
$(t_\nu)_{\nu \geq 1}$ of nonnegative integers which $p$-adically converges to
$s$. Since $\ZZ$ is dense in $\ZZ_p$, there exists at most one function
$\tilde{f}$ with these properties. Finally, we can identify $f = \tilde{f}$.
Note that, for example, $f(-1) = \lim_{n \to \infty} f(p^n-1)$. If the function
$f$ satisfies the Kummer congruences, then $f$ is continuous on $\ZZ_p$ and
$f \in \KS_{p,1}$.

\begin{prop} \label{prop:p-adic-cont}
Let $f$ be a Kummer function. Then $f$ can be uniquely extended to a
continuous $p$-adic function on $\ZZ_p$ such that $f \in \KS_{p,2}$.
\end{prop}

\begin{proof}
According to Definition \ref{def:kummer-space}, define the Mahler expansion
\[
   \tilde{f}(s) = \sum_{\nu \geq 0} \Delta_{f}(\nu) \, p^\nu \binom{s}{\nu},
     \quad s \in \ZZ_p.
\]
By construction $\tilde{f} = f$ restricted on $\NN_0$. Since $\ZZ$ is dense in
$\ZZ_p$, the same arguments from above are valid here. We get a $p$-adic
function $\tilde{f}$ on $\ZZ_p$, such that the Mahler expansions of $\tilde{f}$
and $f$ are equal and $\tilde{f}$ extends $f$ uniquely to $\ZZ_p$. Lemma
\ref{lem:ks2-kummer-congr-0} shows that $\tilde{f} \in \KS_{p,2}$. Finally, we
identify $f = \tilde{f}$ as a function of $\KS_{p,2}$.
\end{proof}

\begin{remark*}
It should be noted that Sun \cite{Sun:2000} introduced so-called $p$-regular
functions which are Kummer functions in this context. He proved some
special congruences, which easily follow here in general and have a full
interpretation in $\KS_{p,2}$. However, his proofs are completely different,
complicated, and lengthy using properties of Stirling numbers and the binomial
inversion theorem.
\end{remark*}

\begin{defin} \label{def:gen-bn}
The generalized Bernoulli numbers $B_{n,\chi}$ are defined by the generating
function
\[
   \sum_{a=1}^{\ff_\chi} \chi(a) \frac{z e^{az}}{e^{\ff_\chi z}-1} =
     \sum_{n \geq 0} B_{n,\chi} \frac{z^n}{n!},
     \quad \norm{z} < \frac{2\pi}{\ff_\chi},
\]
where $\chi$ is a primitive Dirichlet character$\pmod{\ff_\chi}$ and $\ff_\chi$
is a positive integer, which is called the conductor of $\chi$. Choose
$\delta_\chi \in \{0,1\}$ such that $\chi(-1)=(-1)^{\delta_\chi}$ and
$\delta_\chi$ corresponds to even, resp., odd characters. The Dirichlet
$L$-functions are defined by
\[
   L(z,\chi) = \sum_{\nu \geq 1} \chi(\nu) \nu^{-z} ,
     \quad z \in \CC, \, \real z > 1.
\]
\end{defin}

The numbers $B_{n,\chi}$ were introduced and studied by Leopoldt
\cite{Leopoldt:1958} and subsequently examined by Carlitz \cite{Carlitz:1959}
and others. They fulfill the following basic properties,
cf.~\cite[Ch.~4]{Washington:1997}, where we have to exclude the case
$n=1$ when $\chi=1$ since $B_{1,1} = - B_1 = \tfrac12$:
\begin{alignat*}{2}
   \ff_\chi B_{n,\chi}  &\in \ZZ[\chi], &\quad & \chi \neq 1, \\
   B_{n,\chi}  &\neq 0, &&  n \geq 1, \, n \equiv \delta_\chi
     \!\!\!\! \pmod{2}, \\
   B_{n,\chi}  &= 0, &&  n \geq 1, \, n \not\equiv \delta_\chi
     \!\!\!\! \pmod{2}, \\
   L(1-n,\chi) &= -\frac{B_{n,\chi}}{n}, &&  n \geq 1.
\end{alignat*}

The values of $B_{n,\chi}$ are given by the Bernoulli polynomials $B_n(\cdot)$:
\begin{alignat*}{2}
   B_{1,\chi} &= \frac{1}{\ff_\chi} \sum_{a=1}^{\ff_\chi} \chi(a) a,
                 &\quad& \chi \neq 1, \\
   B_{n,\chi} &= \ff_\chi^{n-1} \sum_{a=1}^{\ff_\chi} \chi(a)
                 B_n \!\left( \frac{a}{\ff_\chi} \right), &\quad& n > 1.
\end{alignat*}

Note that $B_{n,\chi}$ reduces to the Bernoulli numbers $B_n$ for $n>1$, if
$\chi = 1$ is the principal character with conductor 1, where $\zeta(z) = L(z,1)$
is the Riemann zeta function. If $\chi = \chi_{-4}$ is the non-principal
character$\pmod{4}$, then $-2B_{n+1,\chi}/(n+1)$ reduces to the Euler numbers $E_n$.

Kubota and Leopoldt \cite{Kubota&Leopoldt:1964} constructed $p$-adic
$L$-functions, that interpolate values of $L$-functions, modified by an Euler
factor, at negative integer arguments. Here we regard their first construction
of $p$-adic $L$-functions, that are defined on certain residue classes,
cf.\ Koblitz \cite{Koblitz:1996}. Note that their second construction is
connected with Iwasawa theory.

Recall Euler's totient function $\eulerphi$ and set $q = p$ for $p \geq 3$
and $q = 4$ for $p=2$, which we use in the following.

\begin{defin} \label{def:mod-l-func}
Define the modified $L$-functions by
\[
   L_p(1-n,\chi) = (1-\chi(p) p^{n-1}) L(1-n,\chi), \quad n \geq 1.
\]
We further define the modified $L$-functions on
residue classes$\pmod{\eulerphi(q)}$ by
\[
   L_{p,l}(s,\chi) = L_p(1-(\delta_\chi+l+\eulerphi(q)s),\chi),
     \quad s \in \NN_0,
\]
where $l$ is a fixed integer and $0 \leq l \leq \eulerphi(q)-2$.
If $l=\delta_\chi=0$, then we exclude the case $s=0$.
We write $\zeta_{p,l}(s)$ for $L_{p,l}(s,1)$.
Define the backward variable substitution
\[
   s_{p,l}(n) = (n-l)/\eulerphi(q),
\]
where we briefly write $s_{p,l}$ in case of no ambiguity.
\end{defin}

Note that $L_{p,l}(\cdot,\chi)$ is defined regardless of the parity of the
character $\chi$, such that $L_{p,l}(\cdot,\chi)$ is the zero function for odd
$l$. The generalized Bernoulli numbers $B_{n,\chi}/n$, resp., the $L$-functions
$L_p(\cdot,\chi)$ at negative integer arguments satisfy the Kummer type
congruences. As usual we have to omit the prime $p$ where $\ff_\chi = p^e$ is a
prime power with $e \geq 1$.

\begin{theorem}[{Carlitz \cite{Carlitz:1959}, Fresnel \cite{Fresnel:1967}}]
\label{thm:lp-congr}
Let $\chi = 1$ or $\chi$ be a primitive non-principal
character$\pmod{\ff_\chi}$. Assume that $p^e \neq \ff_\chi$, $e \geq 1$.
Let $k,n,r$ be positive integers and $h = k \eulerphi(p^r)$ be even. Then
\[
   \Doph{h}{n} \, L_p( 1-s, \chi ) \valueat{s=m} \equiv 0 \pmod{p^{nr}},
     \quad m \in \delta_\chi + 2\NN_0, \, m \geq 1,
\]
where in case $\chi = 1$ additionally suppose that $p>3$ and
$m \not\equiv 0 \pmod{p-1}$.
\end{theorem}

It should be noted that Carlitz rarely used Euler factors, so his congruences
are restricted to$\pmod{(p^{nr},p^{m-1})}$ here. Generally, since
$L_p(1-m,\chi) \in \QQ(\chi) \subset \overline{\QQ}$, one can also view
$L_p(1-m,\chi)$ in a finite extension of $\QQ_p$ to obtain a $p$-adic
$L$-function, cf.~\cite[Ch.~5]{Washington:1997}. Thereby one has to choose
a fixed embedding of $\overline{\QQ}$ into $\CC_p$, the completion of the
algebraic closure $\overline{\QQ}_p$ of $\QQ_p$. Here we keep the focus on
functions on $\ZZ_p$.

\begin{prop} \label{prop:lpl-func-ks2}
Let $\chi = 1$ or $\chi$ be a primitive quadratic character$\pmod{\ff_\chi}$.
Assume that $p > 3$ in case $\chi = 1$, otherwise $p^e \neq \ff_\chi$, $e \geq 1$.
Let $l \in 2\NN_0$, where $0 \leq l < \eulerphi(q)$ and $l \neq \delta_\chi$.
Then $L_{p,l}(\cdot,\chi)$ can be uniquely extended to $\ZZ_p$ such that
$L_{p,l}(\cdot,\chi) \in \KS_{p,2}$.
\end{prop}

\begin{proof}
The conditions above satisfy Theorem \ref{thm:lp-congr}, which we use in a
weaker form to obtain
\[
   \Dop^n L_{p,l}(0,\chi) = \Doph{\eulerphi(q)}{n} \, L_p(1-s,\chi)
     \valueat{s=\delta_\chi+l} \equiv 0 \pmod{p^n}
\]
for all $n \geq 0$. Thus $L_{p,l}(\cdot,\chi)$ is a Kummer function and
Proposition \ref{prop:p-adic-cont} gives the result.
\end{proof}

It remains the case $l = \delta_\chi = 0$ and $\chi \neq 1$. Here we have
the situation that $L_{p,l}(0,\chi)$ is not defined. In spite of that
$L_{p,l}(s,\chi)$ can be uniquely extended to $\ZZ_p$ by removing
the discontinuity at $s=0$.

\begin{prop} \label{prop:lp0-func-ks2}
Let $\chi$ be an even primitive quadratic character$\pmod{\ff_\chi}$. Assume
that $p^e \neq \ff_\chi$, $e \geq 1$. Then $L_{p,0}(\cdot,\chi)$ can be
uniquely extended to $\ZZ_p$ such that $L_{p,0}(s,\chi) \in \KS_{p,2}$ with
a removable discontinuity at $s=0$.
\end{prop}

\begin{proof}
By Theorem \ref{thm:lp-congr} we have
\[
   \Dop^n L_{p,0}(1,\chi) = \Doph{\eulerphi(q)}{n} \, L_p(1-s,\chi)
     \valueat{s=\eulerphi(q)} \equiv 0 \pmod{p^n}
\]
for all $n \geq 0$. Define $f(s) = L_{p,l}(s+1,\chi)$ for $s \in \NN_0$.
Then $f$ is a Kummer function and Proposition \ref{prop:p-adic-cont} shows
that $f \in \KS_{p,2}$. Now define $L_{p,l}(s,\chi) = f(s-1)$
for $s \in \ZZ_p$. Consequently we get an extended function
$L_{p,l}(s,\chi) \in  \KS_{p,2}$, which is defined at $s=0$.
\end{proof}

Henceforth we regard the functions $L_{p,l}(\cdot,\chi)$ as $p$-adic
functions lying in $\KS_{p,2}$ as a result of Propositions
\ref{prop:lpl-func-ks2} and \ref{prop:lp0-func-ks2}.
\smallskip

The definition of irregular primes and irregular pairs, which is usually
introduced in the context of Bernoulli numbers and the class number of
cyclotomic fields, was generalized to generalized Bernoulli numbers by several
authors, cf. Ernvall \cite{Ernvall:1979}, Hao and Parry \cite{Hao&Parry:1984},
and Holden \cite{Holden:1998}, who explicitly studied irregular primes over
real quadratic fields.

Our definition of $\chi$-irregular primes differs somewhat, such that we
associate a $\chi$-irregular pair $(p,l)$ with a $p$-adic $L$-function
$L_{p,l}(\cdot,\chi)$, where $l$ is always even. Moreover, we exclude primes
that divide the conductor of $\chi$, which are considered separately.

\begin{defin} \label{def:irr-prime}
Let $\mathfrak{X}_2$ be the set of primitive quadratic
characters$\pmod{\ff_\chi}$ including $\chi = 1$. We define for
$\chi \in \mathfrak{X}_2$ the set of $\chi$-irregular primes by
\[
   \IRR_\chi = \{ (p,l) : L_{p,l}(0,\chi) \in p \ZZ_p, p > 3,
     p \notdiv \ff_\chi, 2 \leq l \leq p-3, 2 \pdiv l \}.
\]
The index of $\chi$-irregularity of $p$ is defined to be
\[
   i_\chi(p) = \# \{ (p,l) \in \IRR_\chi : 2 \leq l \leq p-3, 2 \pdiv l \}.
\]
\end{defin}
\medskip

The aim of this new definition is that we have a correspondence between
\[
   (p,l) \in \IRR_\chi \qquad \longleftrightarrow \qquad
     L_{p,l}(\cdot,\chi) \in \KS^0_{p,2},
\]
which enables us to study the behavior of the $\chi$-irregular prime $p$ and
its powers as divisors of $L_{p,l}(\cdot,\chi)$. Collecting information about
all $(p,l) \in \IRR_\chi$, we then achieve a description of the structure of
the underlying $L$-function at negative integer arguments.

\begin{theorem} \label{thm:decomp-l-func}
Let $\chi \in \mathfrak{X}_2$ and $n \in 2\NN$. Set $s_{p,l} = s_{p,l}(n)$.
Then
\[
   \norm{L(1-(\delta_\chi+n),\chi)}_\infty = \mathfrak{I}(n,\chi) \,
     \mathfrak{S}(n,\chi) \, \mathfrak{D}(n,\chi),
\]
where
\begin{alignat*}{2}
   \mathfrak{I}(n,\chi) &= \prod_{\substack{(p,l) \in \IRR_\chi\\
     l \equiv n \!\!\! \pmod{p-1}}}
     && \hspace*{-3ex} \norm{L_{p,l}(s_{p,l},\chi)}_p^{-1}, \\
   \mathfrak{S}(n,\chi) &= \prod_{\substack{p \pdiv \ff_\chi\\
     l \equiv n \!\!\! \pmod{p-1}}}
     && \hspace*{-3ex} \norm{L_{p,l}(s_{p,l},\chi)}_p^{-1}, \\
   \mathfrak{D}(n,\chi) &= \hspace*{3.5ex}
     \prod_{\substack{p \, \notdiv \, \ff_\chi\\ p-1 \pdiv n}}
     && \hspace*{-3ex} \norm{L_{p,0}(s_{p,0},\chi)}_p^{-1}.
\end{alignat*}
Moreover, if $\chi$ is odd, then
\[
   \mathfrak{D}(n,\chi) = \prod_{\substack{\chi(p) = 1\\ p-1 \pdiv n}}
     \norm{L_{p,0}(s_{p,0},\chi)}_p^{-1}
     \prod_{\substack{\chi(p) = -1\\ p \pdiv 2 B_{1,\chi}\\ p-1 \pdiv n}}
     \hspace*{-0.8ex}
     \norm{L_{p,0}(s_{p,0},\chi)}_p^{-1}.
\]
\end{theorem}

\begin{proof}
Note that $L(1-(\delta_\chi+n),\chi) \in \QQ^*$.
The product formula states that
\[
   \prod_{p \, \in \, \PP \, \cup \, \{ \infty \}} \hspace*{-1.6ex}
     \norm{L(1-(\delta_\chi+n),\chi)}_p = 1.
\]
Since $\delta_\chi+n \geq 2$, we have for all primes $p$ that
\[
   \norm{L(1-(\delta_\chi+n),\chi)}_p = \norm{L_p(1-(\delta_\chi+n),\chi)}_p
     = \norm{L_{p,l}(s_{p,l},\chi)}_p,
\]
where $l \equiv n \pmod{p-1}$.
Further we divide $\PP$ into the disjoint sets
$I_1 = \{ p : p \pdiv \ff_\chi \}$,
$I_2 = \{ p : p \notdiv \ff_\chi, p-1 \pdiv n \}$,
and $I_3 = \{ p : p \notdiv \ff_\chi, p-1 \notdiv n \}$. Thus
\[
   \norm{L(1-(\delta_\chi+n),\chi)}_\infty
     = \hspace*{-2ex} \prod_{\substack{p \, \in \, I_1 \cup I_2 \cup I_3\\
       l \equiv n \!\!\! \pmod{p-1}}}
       \hspace*{-3ex} \norm{L_{p,l}(s_{p,l},\chi)}_p^{-1}.
\]
We can split the above product into three products over $I_1$, $I_2$, and
$I_3$. The product over $I_1$, resp., $I_2$ equals $\mathfrak{S}(n,\chi)$,
resp., $\mathfrak{D}(n,\chi)$. It remains the product over $I_3$. We have to
show that
\[
   \prod_{\substack{p \, \notdiv \ff_\chi, \, l \neq 0\\
     l \equiv n \!\!\! \pmod{p-1}}}
       \hspace*{-3ex} \norm{L_{p,l}(s_{p,l},\chi)}_p^{-1}
     = \hspace*{-2ex} \prod_{\substack{(p,l) \in \IRR_\chi\\
       l \equiv n \!\!\! \pmod{p-1}}}
       \hspace*{-3ex} \norm{L_{p,l}(s_{p,l},\chi)}_p^{-1}.
\]
By Proposition \ref{prop:lpl-func-ks2} the left product above consists only of
functions $L_{p,l}(\cdot,\chi) \in \KS_{p,2}$. From Definition
\ref{def:irr-prime} we deduce for these functions that $(p,l) \notin \IRR_\chi$
implies $L_{p,l}(\cdot,\chi) \in \KS^*_{p,2}$.

Now, assume that $\chi$ is odd. Regarding the product of $\mathfrak{D}(n,\chi)$
we can also write $I_2 = I_2^+ \cup I_2^-$, where
$I_2^\pm = \{ p : \chi(p) = \pm 1, p-1 \pdiv n \}$.
If $p \in I_2^-$, then we have
\[
   L_{p,0}(0,\chi) = -(1-\chi(p)) B_{1,\chi} = -2 B_{1,\chi} \neq 0.
\]
Here we use the non-trivial fact that $B_{1,\chi} \neq 0$ for odd $\chi$,
cf.~\cite[Thm.~4.9, p.~38]{Washington:1997}.
Since $L_{p,0}(\cdot,\chi) \in \KS_{p,2}$, it follows that
$L_{p,0}(\cdot,\chi) \in \KS^*_{p,2}$ when $p \notdiv 2 B_{1,\chi}$.
\end{proof}

Our main interest is focused on the product of $\mathfrak{I}(\cdot,\chi)$,
since the $\chi$-irregular primes and their powers are the fundamental elements,
building mainly the values of $L(\cdot,\chi)$ at negative integer arguments.
Theorem \ref{thm:decomp-l-func} shows that
\[
   \mathfrak{I}(n,\chi) = \prod_{\substack{(p,l) \in \IRR_\chi\\
     l \equiv n \!\!\! \pmod{p-1}}}
     \hspace*{-3ex} \norm{L_{p,l}(s_{p,l},\chi)}_p^{-1}, \quad n \in 2\NN,
\]
where the functions, lying in $\KS_{p,2}$, can have a different behavior, such that
\[
   \norm{L_{p,l}(s_{p,l},\chi)}_p^{-1} \text{\ is \ } \left\{
     \begin{array}{l}
       \text{constant,} \\
       \text{bounded,} \\
       \text{unbounded,}
     \end{array} \right.
\]
as a result of the last section. As supported by computations, mainly functions
of $\WKS^0_{p,2}$ have been found. Thus, powers of $\chi$-irregular primes,
unbounded as $n \to \infty$ for $n \in 2\NN$, seem to give contribution to the
values of $L(1-(\delta_\chi+n),\chi)$.

\begin{remark} \label{rem:irr-results}
Buhler et al \cite{Buhler&others:2001} calculated irregular pairs and
cyclotomic invariants for all primes below 12 million. Due to their results,
we deduce for these pairs that
$\zeta_{p,l} \in \WKS^0_{p,2}$, $p^2 \notdiv \zeta_{p,l}(0)$,
and the zero $\xi \in \ZZ_p^*$ by Proposition \ref{prop:ks2-f-ord-1}.
Holden \cite{Holden:1998} showed that there are examples of $\chi$-irregular
pairs, $\chi$ a primitive quadratic character, such that
$p^2 \pdiv L_{p,l}(0,\chi)$. However, we have recalculated these
examples to demonstrate that the functions in question have
always $\lambda_f = 1$ and lie in $\WKS^0_{p,2}$; consequently the zero
$\xi \in p\ZZ_p$ by Proposition \ref{prop:ks2-f-ord-2}.
These and further computational results are given in \cite{Kellner:2009};
see also Example \ref{expl:chi-irr}.
\end{remark}

These aspects lead to the following conjecture about the \textit{irregular part}
of the values of $L(\cdot,\chi)$ at negative integer arguments.

\begin{conj} \label{conj:irr-prod}
Assume the conditions of Theorem \ref{thm:decomp-l-func}. We postulate the
following conjecture, which may hold either in weak or strong form, for a
given $L$-function.

\begin{enumerate}
\item Weak form:
  \[
     \mathfrak{I}(n,\chi) = \prod_{\substack{(p,l) \in \IRR_\chi\\
       l \equiv n \!\!\! \pmod{p-1}}} \hspace*{-3ex}  p^{\lambda_f} \,
       \prod_{\nu = 1}^{\lambda_f} \norm{s_{p,l}-\xi^{(\nu)}_{p,l}}_p^{-1},
       \quad n \in 2\NN,
  \]
  where $\xi^{(\nu)}_{p,l} \in \ZZ_p$ are the roots of the corresponding
  function $f = L_{p,l}(\cdot,\chi) \in \KS^{\rm s}_{p,2}$ with $\lambda_f < p$.
\item Strong form:
  \[
     \mathfrak{I}(n,\chi) = \prod_{\substack{(p,l) \in \IRR_\chi\\
       l \equiv n \!\!\! \pmod{p-1}}} \hspace*{-3ex} p \,
       \norm{s_{p,l}-\xi_{p,l}}_p^{-1},
       \quad n \in 2\NN,
  \]
  where $\xi_{p,l} \in \ZZ_p$ is the root of the corresponding function
  $L_{p,l}(\cdot,\chi) \in \WKS^0_{p,2}$.
\end{enumerate}
\end{conj}

Clearly, products of $L$-functions provide examples, trivially take
$L(\cdot,\chi)^2$, where we then get a conjectural product representation as
above in \textit{weak form}. Therefore we define a product of $L$-functions,
which only makes sense when the characters have the same parity. Otherwise we
would get a zero function at negative integer arguments.

\begin{defin}
Let $\chi_1, \chi_2 \in \mathfrak{X}_2$. Assume that $\chi_1$ and $\chi_2$
have the same parity. Define the product $L(\cdot,\chi_1 \otimes \chi_2) =
L(\cdot,\chi_1) L(\cdot,\chi_2)$, which transfers to all other definitions.
Further define $\IRR_{\chi_1,\chi_2} = \IRR_{\chi_1} \cap \IRR_{\chi_2}$.
\end{defin}

\begin{prop} \label{prop:prod-l-func}
Let $\chi_1, \chi_2 \in \mathfrak{X}_2$ having the same parity.
If $\IRR_{\chi_1,\chi_2} \neq \emptyset$, then there exist functions
$f = L_{p,l}(\cdot,\chi_1 \otimes \chi_2) \in \KS_{p,2}$ with
$(p,l) \in \IRR_{\chi_1,\chi_2}$ and $\lambda_f \geq 2$.
Assuming Conjecture \ref{conj:irr-prod}, these functions lie in
$\KS^{\rm s}_{p,2}$ and $\mathfrak{I}(n,\chi_1 \otimes \chi_2)$
has a weak form.
\end{prop}

\begin{proof}
We use the notation of Theorem \ref{thm:decomp-l-func}. Let $n \in 2\NN$.
We then obtain
\[
  \mathfrak{I}(n,\chi_1 \otimes \chi_2)
    = \mathfrak{I}(n,\chi_1) \, \mathfrak{I}(n,\chi_2)
    = \mathfrak{I}_1 \, \mathfrak{I}_2 \, \mathfrak{I}_3,
\]
where
\begin{alignat*}{2}
  \mathfrak{I}_\nu &=
    \prod_{\substack{(p,l) \in \IRR_{\chi_\nu} \backslash \IRR_{\chi_1,\chi_2}\\
    l \equiv n \!\!\! \pmod{p-1}}}
    && \hspace*{-3ex} \norm{L_{p,l}(s_{p,l},\chi_\nu)}_p^{-1}, \quad \nu = 1,2,\\
  \mathfrak{I}_3 &= \hspace*{1ex}
    \prod_{\substack{(p,l) \in \IRR_{\chi_1,\chi_2}\\
    l \equiv n \!\!\! \pmod{p-1}}}
    && \hspace*{-3ex} \norm{L_{p,l}(s_{p,l},\chi_1)L_{p,l}(s_{p,l},\chi_2)}_p^{-1}.
\end{alignat*}
If $\IRR_{\chi_1,\chi_2} \neq \emptyset$, then the product of $\mathfrak{I}_3$
cannot be trivial for all $n$. For $(p,l) \in \IRR_{\chi_1,\chi_2}$ we get
\[
   L_{p,l}(\cdot,\chi_1)L_{p,l}(\cdot,\chi_2)
     = L_{p,l}(\cdot,\chi_1 \otimes \chi_2) \in \KS_{p,2}.
\]
Since both functions $L_{p,l}(\cdot,\chi_\nu) \in \KS^0_{p,2}$, the product
$f = L_{p,l}(\cdot,\chi_1 \otimes \chi_2)$ has $\lambda_f \geq 2$
by Proposition \ref{prop:prod-f-prime-delta}.
Assuming Conjecture \ref{conj:irr-prod}, all these functions lie in
$\KS^{\rm s}_{p,2}$ and so do their products. Since we have $\lambda_f \geq 2$
for some functions, $\mathfrak{I}(n,\chi_1 \otimes \chi_2)$ has a weak form.
\end{proof}

Next, we consider the connection with the Dedekind zeta function.

\begin{defin}
Let $K$ be an algebraic number field. The Dedekind zeta function is defined by
\[
   \zeta_K(z) = \sum_{\mathfrak{a}} N(\mathfrak{a})^{-z},
     \quad z \in \CC, \, \real z > 1,
\]
where the sum runs over all nonzero integral ideals of $K$ and
$N(\mathfrak{a})$ denotes the norm of the ideal $\mathfrak{a}$.
\end{defin}

We need the well known decomposition for quadratic fields,
cf. \cite[Thm.~4.3, p.~34]{Washington:1997}.

\begin{theorem}
Let $D$ be the fundamental discriminant of the quadratic field
$K = \QQ(\sqrt{D})$. Then
\[
   \zeta_K(z) = \zeta(z) L(z,\chi_D), \quad z \in \CC,
\]
where $\chi_D(\cdot) = \left( \frac{D}{\cdot} \right)$ is the Kronecker symbol.
\end{theorem}

\begin{corl}
Let $D > 0$ be the fundamental discriminant of the real quadratic field
$K = \QQ(\sqrt{D})$. Then
\[
   \zeta_K(1-n) = L(1-n,1 \otimes \chi_D), \quad n \geq 2.
\]
The irregular part of $\zeta_K(1-n)$ is described by
$\mathfrak{I}(n,1 \otimes \chi_D)$ for $n \in 2\NN$.
\end{corl}

\begin{proof}
This follows by Proposition \ref{prop:prod-l-func},
since $\chi = 1$ and $\chi_D$ have the same parity.
\end{proof}

Holden \cite{Holden:1998} studied $\zeta_K$ in case of real quadratic fields in
context of $\chi_D$-irregular primes and their distribution. We use here his
computational results, which we have recalculated and extended for our purpose.
He found an example for $D=77$ such that $(37,32)$ is both an irregular pair
and a $\chi_D$-irregular pair; this implies $\IRR_{1,\chi_D} \neq \emptyset$.
As mentioned earlier in Remark \ref{rem:irr-results}, there are examples where
$p^2 \pdiv L_{p,l}(0,{\chi_D})$, e.g., for $D=5$ and $(p,l)=(443,216)$.
Holden mentions that there are other examples, particularly for $D=5$ and
$p < 50$, but without publishing these data. Therefore we have used
the tables of $\chi_D$-irregular primes for $D=5$ of Hao and Parry
\cite{Hao&Parry:1984}, to find the example for $D=5$ and $(p,l)=(19,8)$,
which is more suitable for our calculations below. The computations are
reported in \cite{Kellner:2009}.

\begin{expl} \label{expl:chi-irr}
Computed zeros $\xi$ and fixed points $\tau$$\pmod{p^{10}}$.
\begin{enumerate}
\item Case $D = 5$, $(p,l)=(19,8)$,
$f = L_{p,l}(\cdot,\chi_D) \in \WKS^0_{p,2}$:
\begin{center}
\begin{tabular}{|c|l|} \hline
  $f$ & values $(s=0, 1)$ / $p$-adic digits $s_0, \ldots, s_9$ \\ \hline\hline
  $\Delta_f, \lambda_f$ & 16, 1 \\ \hline
  $\ord_p f(s)$ & 2, 1 \\ \hline
  $\xi$  & 0, 7, 18, 11, 12, 10, 10, 8, 14, 0 \\ \hline
  $\tau$ & 0, 0, 2, 13, 11, 4, 15, 6, 12, 16 \\ \hline
\end{tabular}
\end{center}

\item Case $D = 77$, $(p,l)=(37,32)$,
$f = L_{p,l}(\cdot,1 \otimes \chi_D) \in \KS^{\rm s}_{p,2}$:
\begin{center}
\begin{tabular}{|c|l|} \hline
  $f$ & values $(s=0, 1)$ / $p$-adic digits $s_0, \ldots, s_9$ \\ \hline\hline
  $\Delta_f, \lambda_f$ & 0, 2 \\ \hline
  $\ord_p f(s)$ & 2, 2 \\ \hline
  $\xi_1$ & 7, 28, 21, 30, 4, 17, 26, 13, 32, 35 \\ \hline
  $\xi_{\chi_D}$ & 9, 36, 26, 31, 25, 30, 21, 36, 30, 33 \\ \hline
  $\tau$  & 0, 0, 14, 35, 13, 27, 30, 3, 22, 29 \\ \hline
\end{tabular}
\end{center}

\end{enumerate}
\end{expl}
\smallskip

\begin{remark*}
Surely, there are several authors, who already computed zeros of $p$-adic
$L$-func\-tions, mostly in context of Iwasawa theory. These calculations
were performed by searching a start solution$\pmod{p^r}$ for some $r \geq 1$
and further using Newton's method. In contrast, we give here a necessary and
sufficient condition ($f(0) \in p\ZZ_p$ and $\Delta_f \neq 0$, or
$\lambda_f = 1$ for both), so that $f \in \WKS^0_{p,2}$ which shows the
existence of a zero and $|f(s)|_p$ reduces to a linear term
(for $p=2$ we also need the condition $2 \pdiv \Delta_f(2)$).
\end{remark*}
\medskip

At the end, we consider the \textit{non-irregular part}
$\mathfrak{D}(\cdot,\chi)$ in case $\chi$ is odd.
Here we have the interesting situation, that the functions
$L_{p,0}(s,\chi)$ have a zero at $s=0$ when $\chi(p) = 1$.

\begin{defin} \label{def:except-prime}
Let $\chi \in \mathfrak{X}_2$ where $\chi$ is odd. We define
\[
   \EXC_\chi = \{ (p,0) : L_{p,0}(1,\chi) \in p^2\ZZ_p, p > 3,
     \chi(p) = 1 \}
\]
as the set of $\chi$-exceptional pairs.
We further define for $(p,0) \in \EXC_\chi$ the functions
\[
   \tilde{L}_{p,0}(s,\chi) = \left\{
     \begin{array}{cc}
       L_{p,0}(s,\chi)/ps, & s \neq 0, \vspace*{1ex}\\
       L'_{p,0}(s,\chi)/p, & s = 0,
     \end{array}
   \right. \quad s \in \ZZ_p.
\]
\end{defin}

\begin{remark*}
As a result of Proposition \ref{prop:deg-ks2-zero-0} and
Lemma \ref{lem:ks2-zero0}, we have for $p > 3$ that
\[
   \tilde{L}_{p,0}(\cdot,\chi) \in \KS^{\rm d}_{p,1}
     \quad \text{and} \quad
     \tilde{L}_{p,0}(0,\chi) = \int_{\ZZ_p} \KSop L_{p,0}(s,\chi) \, ds.
\]
The value of $\tilde{L}_{p,0}(0,\chi)$ is easily computable by
\eqref{eq:ks2-zero0-g0} of Lemma \ref{lem:ks2-zero0}.
\end{remark*}

The next theorem shows that Definition \ref{def:except-prime} is well defined.
The set $\EXC_\chi$ can be seen as an analogue to $\IRR_\chi$, where all
functions $f = L_{p,0}(\cdot,\chi)$ with $(p,0) \in \EXC_\chi$ have the
property that $\lambda_f > 1$.

\begin{theorem} \label{thm:except-l-func}
Let $\chi \in \mathfrak{X}_2$ where $\chi$ is odd.
Let $p > 3$ where $\chi(p) = 1$.
For $f = L_{p,0}(\cdot,\chi)$ we have the following statements:
\begin{enumerate}
\item $\lambda_f = 1 \iff \ord_p L_{p,0}(1,\chi) = 1
       \iff (p,0) \notin \EXC_\chi$.
\item $\lambda_f > 1 \iff \ord_p L_{p,0}(1,\chi) \geq 2
       \iff (p,0) \in \EXC_\chi$.
\item If $\lambda_f = 1$, then $\norm{L_{p,0}(s,\chi)}_p = \norm{ps}_p$
      for $s \in \ZZ_p$.
\item If $\lambda_f = 2$, then $\norm{L_{p,0}(s,\chi)}_p =
      \norm{p^2s(s-\xi_{p,0})}_p$ for $s \in \ZZ_p$,
      where $\xi_{p,0}$ is the unique simple zero of
      $\tilde{L}_{p,0}(\cdot,\chi) \in \KS^{\rm d}_{p,1}$.
\end{enumerate}
\end{theorem}

\begin{proof}
(1)-(3):
Since $\chi$ is odd, we have
\[
   L_{p,0}(0,\chi) = -(1-\chi(p)) B_{1,\chi} = 0.
\]
Therefore $L_{p,0}(\cdot,\chi)$ has a zero at $s=0$.
By Lemma \ref{lem:ks2-zero0} and Definition \ref{def:except-prime} we obtain
\[
   L_{p,0}(s,\chi) = ps \, \tilde{L}_{p,0}(s,\chi), \quad s \in \ZZ_p.
\]
Using \eqref{eq:ks2-zero0-g} of Lemma \ref{lem:ks2-zero0}, we have
\[
   \tilde{L}_{p,0}(s,\chi) \equiv \Delta_f \pmod{p \ZZ_p}.
\]
If $\lambda_f = 1$, then $\norml{\tilde{L}_{p,0}(s,\chi)}_p = 1$
and $\norm{L_{p,0}(s,\chi)}_p = \norm{ps}_p$ for $s \in \ZZ_p$,
thus $\ord_p L_{p,0}(1,\chi) = 1$.
Conversely, $\lambda_f > 1$ implies that $\norml{\tilde{L}_{p,0}(s,\chi)}_p < 1$
for $s \in \ZZ_p$ and $p^2 \pdiv L_{p,0}(1,\chi)$.

(4): Since $L_{p,0}(s,\chi)$ has a zero at $s=0$ and $\lambda_f = 2$,
this function satisfies the conditions of Theorem \ref{thm:ks2-zero-2}.
It follows that $\norm{L_{p,0}(s,\chi)}_p = \norm{p^2s(s-\xi_{p,0})}_p$
for $s \in \ZZ_p$ and $\xi_{p,0}$ is the unique simple zero of
$\tilde{L}_{p,0}(\cdot,\chi)$.
Proposition \ref{prop:deg-ks2-zero-0} shows that
$\tilde{L}_{p,0}(\cdot,\chi) \in \KS^{\rm d}_{p,1}$.
\end{proof}

We achieve a more detailed decomposition of $\mathfrak{D}(n,\chi)$ as follows.

\begin{theorem} \label{thm:decomp-d-prod}
Assume the conditions of Theorem \ref{thm:decomp-l-func} where $\chi$ is odd.
We have
\[
   \mathfrak{D}(n,\chi) = \mathfrak{D}_{2,3}(n,\chi) \, \mathfrak{D}_{+}(n,\chi)
     \, \mathfrak{D}_{-}(n,\chi) \, \mathfrak{D}_{0}(n,\chi),
\]
where
\begin{alignat*}{3}
   \mathfrak{D}_{2,3}(n,\chi) &=
     \prod_{\substack{p \in I_\chi^{2,3}}}
     && \hspace*{-1.1ex} \norm{L_{p,0}(s_{p,0},\chi)}_p^{-1},
     \quad & I_\chi^{2,3} &= \{ p \leq 3 : \chi(p) \neq 0,
     p \pdiv (1-\chi(p))B_{1,\chi} \}, \\
   \mathfrak{D}_+(n,\chi) &=
     \prod_{\substack{p \in I_\chi^+\\ p-1 \pdiv n}}
     && \hspace*{-1.1ex} \norm{p n}_p^{-1},
     \quad & I_\chi^+ &= \{ p > 3 : \chi(p) = 1 \}, \\
   \mathfrak{D}_-(n,\chi) &=
     \prod_{\substack{p \in I_\chi^-\\ p-1 \pdiv n}}
     && \hspace*{-1.1ex} \norm{L_{p,0}(s_{p,0},\chi)}_p^{-1},
     \quad & I_\chi^- &= \{ p > 3 : \chi(p) = -1, p \pdiv B_{1,\chi} \}, \\
   \mathfrak{D}_0(n,\chi) &= \hspace*{-1.5ex}
     \prod_{\substack{(p,0) \in \EXC_\chi\\ p-1 \pdiv n}}
     && \hspace*{-0.8ex} \norml{\tilde{L}_{p,0}(s_{p,0},\chi)}_p^{-1}.
\end{alignat*}
\end{theorem}

\begin{proof}
By Theorem \ref{thm:decomp-l-func} we have
\[
   \mathfrak{D}(n,\chi) = \prod_{\substack{\chi(p) = 1\\ p-1 \pdiv n}}
     \norm{L_{p,0}(s_{p,0},\chi)}_p^{-1}
     \prod_{\substack{\chi(p) = -1\\ p \pdiv 2 B_{1,\chi}\\ p-1 \pdiv n}}
     \hspace*{-0.8ex}
     \norm{L_{p,0}(s_{p,0},\chi)}_p^{-1},
\]
where all functions lie in $\KS_{p,2}$.
First we separate the factors for $p=2$ and $p=3$ from the products above.
This defines $\mathfrak{D}_{2,3}(n,\chi)$ and $I_\chi^{2,3}$, where we use
the original condition $p \pdiv L_{p,0}(0,\chi) = -(1-\chi(p)) B_{1,\chi}$,
such that $L_{p,0}(\cdot,\chi) \notin \KS^*_{p,2}$. The second product above
for $p > 3$ defines $\mathfrak{D}_{-}(n,\chi)$, which covers
the case $\chi(p) = -1$ and the modified condition $p \pdiv B_{1,\chi}$.
Now, we consider the remaining case $\chi(p) = 1$ for $p > 3$.
We use Theorem \ref{thm:except-l-func} to obtain
\[
   \prod_{\substack{p > 3\\ \chi(p) = 1\\ p-1 \pdiv n}}
     \norm{L_{p,0}(s_{p,0},\chi)}_p^{-1}
     = \mathfrak{D}_{+}(n,\chi) \, \mathfrak{D}_{0}(n,\chi),
\]
where we have used that
\[
   \norm{L_{p,0}(s_{p,0},\chi)}_p^{-1} = \norm{p s_{p,0}}_p^{-1}
     = \norm{p n/(p-1)}_p^{-1} = \norm{p n}_p^{-1},
     \quad (p,0) \notin \EXC_\chi,
\]
and
\[
   \norm{L_{p,0}(s_{p,0},\chi)}_p^{-1} = \norm{p n}_p^{-1}
     \norml{\tilde{L}_{p,0}(s_{p,0},\chi)}_p^{-1},
     \quad (p,0) \in \EXC_\chi. \qedhere
\]
\end{proof}

We examine now examples of $\chi$-exceptional primes. Let $\chi_{-3}$ be the
non-principal character$\pmod{3}$ associated with the imaginary quadratic field
$\QQ(\sqrt{-3})$. Ernvall \cite{Ernvall:1979} studied the $\chi$-irregular
pairs of the so-called $D$-numbers, which are given by
\[
   D_n = 3 L(-n, \chi_{-3}), \quad n \geq 1.
\]

He states that $p^2 \pdiv D_{p-1}$ occurs only for $p=13$, $181$, and $2521$
below $10^4$, where he also remarks that the primes $p=13$ and $p=181$ were
already found by Ferrero \cite{Ferrero:1978}. We have found exactly two further
primes below $10^6$: $p=76543$ and $p=489061$. The condition $p^2 \pdiv
D_{p-1}$ is equivalent to $L_{p,0}(1,\chi_{-3}) \in p^2\ZZ_p$ and therefore
$(p,0) \in \EXC_{\chi_{-3}}$ for these five primes. Certainly, these
primes $p$ satisfy $\chi_{-3}(p)=1$ or equivalently $p \equiv 1 \pmod{3}$.

Regarding the primes mentioned above, the following table shows that these
functions $f = L_{p,0}(\cdot,\chi_{-3})$ have always $\lambda_f = 2$;
the corresponding functions $\tilde{L}_{p,0}(\cdot,\chi_{-3})$
have each time a unique simple zero by Theorem \ref{thm:except-l-func}.
More results are given in \cite{Kellner:2009}.

\begin{tbl}
Computed parameters of functions $f = L_{p,0}(\cdot,\chi_{-3}) \in \KS^2_{p,2}$:
\begin{center}
\begin{tabular}{|c|c|c|c|} \hline
  $p$ & $\lambda_f$ & $\ord_p f(1)$ & $\ord_p \Delta_f(2)$  \\ \hline\hline
  13   & 2 & 2 & 0 \\ \hline
  181  & 2 & 2 & 0 \\ \hline
  2521 & 2 & 2 & 0 \\ \hline
  76543 & 2 & $\geq$ 2 & 0 \\ \hline
  489061 & 2 & $\geq$ 2 & 0 \\ \hline
\end{tabular}
\end{center}
\end{tbl}

Using Algorithm \ref{alg:deg-zero}, we have computed the zero of the
$\delta$-degenerate function $\tilde{L}_{p,0}(\cdot,\chi_{-3})$ for $p=13$.

\begin{expl}
Computed zero $\xi$$\pmod{p^{10}}$ of
$f = \tilde{L}_{p,0}(\cdot,\chi_{-3}) \in \KS^{\rm d}_{p,1}$
for the case $p=13$:
\begin{center}
\begin{tabular}{|c|l|} \hline
  $f$ & values / $p$-adic digits $s_0, \ldots, s_9$ \\ \hline\hline
  $\Delta_f, \lambda_f$ & 3, 1 \\ \hline
  $\xi$  & 3, 8, 2, 11, 1, 1, 10, 12, 7, 1 \\ \hline
\end{tabular}
\end{center}
\end{expl}

The functions $\tilde{L}_{p,0}(\cdot,\chi)$ seem to behave like the functions
of the irregular part. In contrast, these functions lie in $\KS^{\rm d}_{p,1}$
having a defect in their Mahler expansion. Therefore we raise the following
conjecture about the functions $\tilde{L}_{p,0}(\cdot,\chi)$ in the case of
$\chi$-exceptional pairs.

\begin{conj} \label{conj:except-prod}
Assume the conditions of Theorem \ref{thm:decomp-d-prod}. Then
\[
   \mathfrak{D}_0(n,\chi) = \prod_{\substack{(p,0) \in \EXC_\chi\\
     p-1 \pdiv n}} \hspace*{-1ex}
     \norm{p \, (s_{p,0} - \xi_{p,0})}_p^{-1},
     \quad n \in 2\NN,
\]
where $\xi_{p,0}$ is the zero of
$\tilde{L}_{p,0}(\cdot,\chi) \in \KS^{\rm d}_{p,1}$.
\end{conj}
\medskip

\section{Bernoulli and Euler numbers}

The Bernoulli and Euler numbers are defined by
\begin{alignat*}{2}
   \frac{z}{e^z-1} &= \sum_{n=0}^\infty B_n \frac{z^n}{n!},
     &\quad |z| &< 2 \pi, \\
   \frac{2}{e^z+e^{-z}} &= \sum_{n=0}^\infty E_n \frac{z^n}{n!},
     &\quad |z| &< 2 \pi.
\end{alignat*}
The numbers $B_n$ are rational, whereas the numbers $E_n$ are integers.
It easily follows that
\begin{alignat*}{3}
   E_n &= -2 \frac{B_{n+1,\chi_{-4}}}{n+1} &&= 2L(-n,\chi_{-4}),
          &\quad n &\geq 0, \\
   B_n &= B_{n,1} &&= -n\zeta(1-n), & n &\geq 2.
\end{alignat*}
We also have the connection with the Dedekind zeta function of $\QQ(i)$ that
\[
   \zeta_{\QQ(i)}(z) = \zeta(z) L(z,\chi_{-4}), \quad z \in \CC.
\]

In 1850 Kummer introduced congruences about Bernoulli and Euler numbers in
the following form, which have been greatly generalized after that and now
are called Kummer congruences. For the sake of completeness we cite Kummer's
theorem.

\begin{theorem}[{Kummer \cite{Kummer:1851}}] \label{thm:kummer}
Let $n, r$ be positive integers, where $n$ is even.
\begin{enumerate}
\item If $p-1 \notdiv n$ and $n > r$, then
\[
   \sum_{\nu=0}^r \binom{r}{\nu} (-1)^{r-\nu}
     \frac{B_{n+\nu(p-1)}}{n+\nu(p-1)} \equiv 0 \pmod{p^r}.
\]
\item If $n > r$, then
\[
   \sum_{\nu=0}^r \binom{r}{\nu} (-1)^{r-\nu} E_{n+\nu(p-1)}
     \equiv 0 \pmod{p^r}.
\]
\end{enumerate}
\end{theorem}

Theorem \ref{thm:kummer}, formulated with Euler factors to remove the
restriction $n > r$, would already be sufficient to define Kummer functions,
which, extended to $\ZZ_p$, lie in $\KS_{p,2}$. Now, one can apply all results
about $\KS_{p,2}$ to $B_n/n$ and $E_n$, always modified by an Euler factor,
since the last section has shown that $\zeta_{p,l}$ and
$L_{p,l}(\cdot,\chi_{-4})$ are functions of $\KS_{p,2}$.
Proposition \ref{prop:kummer-congr-value} enables us to compute $B_n/n$ and
$E_n$$\pmod{p^r}$ for arbitrary even integers $n$. Algorithm \ref{alg:zero},
resp., Algorithm \ref{alg:fixed-point} shows how to compute a zero, resp.,
a fixed point$\pmod{p^r}$ of $\zeta_{p,l}$ and $L_{p,l}(\cdot,\chi_{-4})$.
As an application, we can sharpen the usual Kummer congruences for the
Bernoulli numbers for those cases where the converse also holds.

\begin{prop}[Strong Kummer congruences]
Let $p > 3$ and $l \in 2\NN$ where $0 < l < p-1$.
Set $\EF{l} = 1-p$ in case $l=2$, otherwise $\EF{l} = 1$.
Then
\[
   \frac{B_{l+p-1}}{l+p-1} \not\equiv \EF{l} \frac{B_l}{l} \pmod{p^2}
\]
if and only if
\[
   n \equiv m \pmod{\eulerphi(p^r)} \quad \Longleftrightarrow \quad
     ( 1 - p^{n-1} ) \frac{B_n}{n} \equiv ( 1 - p^{m-1} ) \frac{B_m}{m}
     \pmod{p^r}
\]
for $n,m \in 2\NN$ such that $n \equiv m \equiv l \pmod{p-1}$ and
$1 \leq r \leq 1 + \ord_p (n-m)$.
\end{prop}

\begin{proof}
We observe that $\Dop \zeta_{p,l}(0) \equiv \EF{l} B_l/l -
B_{l+p-1}/(l+p-1) \pmod{p^2\ZZ_p}$, where the Euler factors vanish except
for $l=2$ and $\zeta_{p,l}(0)$. Therefore the condition above is equivalent to
$\Delta_{\zeta_{p,l}} \neq 0$ to ensure that $\zeta_{p,l} \in \WKS_{p,2}$.
By Corollary \ref{corl:ks2-strong-kummer-congr} we then have
\[
   \norm{p \, (s-t)}_p = \norm{\zeta_{p,l}(s)-\zeta_{p,l}(t)}_p,
     \quad s, t \in \ZZ_p.
\]
Conversely, $\Delta_{\zeta_{p,l}} = 0$ implies that
\[
   \norm{p \, (s-t)}_p > \norm{\zeta_{p,l}(s)-\zeta_{p,l}(t)}_p,
     \quad s, t \in \ZZ_p,
\]
as a result of Proposition \ref{prop:f-quod-delta}.
Transferring this back to the Bernoulli numbers with
$n = s(p-1) + l, m = t(p-1) + l \in \NN$ gives the result.
\end{proof}

\begin{remark*}
One cannot omit the Euler factor for $l=2$ in the condition of the proposition
above. For example, if $p=13$, then $B_{14}/14 = B_2/2$, but
$B_{14}/14 - (1-p)B_2/2 = 13/12 \not\equiv 0 \pmod{p^2}$.
Without the condition, we only have the implication "$\Rightarrow$",
which equals the usual Kummer congruences.
\end{remark*}

\begin{expl} \label{expl:zero-bn-en}
Computed zeros $\xi$ and fixed points $\tau$$\pmod{p^{10}}$ of functions
of $\WKS^0_{p,2}$.
\begin{enumerate}

\item Case $(p,l)=(37,32)$ and $f = \zeta_{p,l} \in \WKS^0_{p,2}$:
\begin{center}
\begin{tabular}{|c|l|} \hline
  $f$ & values / $p$-adic digits $s_0, \ldots, s_9$ \\ \hline\hline
  $\Delta_f, \lambda_f$ & 16, 1 \\ \hline
  $\xi$  & 7, 28, 21, 30, 4, 17, 26, 13, 32, 35 \\ \hline
  $\tau$ & 0, 36, 28, 6, 26, 35, 27, 23, 10, 11 \\ \hline
\end{tabular}
\end{center}
It follows that the smallest indices for $\ord_p(B_{n_\nu}/{n_\nu}) = \nu$
are, e.g., $n_1 = 32$, $n_2 = 284$, $n_3 = 37\,580$,
$n_4 = 1\,072\,544$, and $n_5 = 55\,777\,784$.

\vfill \pagebreak

\item Case $(p,l)=(19,10)$ and $f = L_{p,l}(\cdot,\chi_{-4}) \in \WKS^0_{p,2}$:
\begin{center}
\begin{tabular}{|c|l|} \hline \allowdisplaybreaks
  $f$ & values / $p$-adic digits $s_0, \ldots, s_9$ \\
    \hline\hline
  $\Delta_f,\lambda_f$ & 5, 1 \\ \hline
  $\xi$  & 17, 6, 13, 18, 17, 10, 6, 18, 12, 14 \\ \hline
  $\tau$ & 0, 10, 8, 17, 15, 1, 4, 9, 14, 18 \\ \hline
\end{tabular}
\end{center}
It follows that the smallest indices for $\ord_p E_{n_\nu} = \nu$
are, e.g., $n_1 = 10$, $n_2 = 316$, $n_3 = 2\,368$,
$n_4 = 86\,842$, and $n_5 = 2\,309\,158$.

\end{enumerate}
\end{expl}

Definition \ref{def:irr-prime} of $\chi$-irregular primes for $\chi = 1$,
resp., $\chi=\chi_{-4}$ agrees with the usual definition of irregular primes
regarding $B_n$, resp., $E_n$. The latter are often called $E$-irregular
primes, cf. Carlitz \cite{Carlitz:1954b}, Ernvall and Mets{\"a}nkyl{\"a}
\cite{Ernvall&Met:1979}. As a result of Carlitz \cite{Carlitz:1954b}, there
are infinitely many irregular primes regarding $B_n$ and $E_n$. Equivalently,
Jensen \cite{Jensen:1915} showed a more special result for the Bernoulli
numbers before, that there are infinitely many irregular primes $p \equiv 3
\pmod{4}$. Ernvall \cite{Ernvall:1979} later showed that there are infinitely
many $E$-irregular primes $p \not\equiv \pm 1 \pmod{8}$.
Therefore $\#\IRR_1 = \infty$ and $\#\IRR_{\chi_{-4}} = \infty$.

We will derive a conjectural formula for the structure of the Bernoulli and
Euler numbers. Recall the notations of Theorem \ref{thm:decomp-l-func},
which we use in the following. Considering Remark \ref{rem:irr-results} and
Conjecture \ref{conj:irr-prod}, we may assume that the corresponding products
$\mathfrak{I}(\cdot,1)$ and $\mathfrak{I}(\cdot,\chi_{-4})$ fulfil the strong
form. First, we consider the Bernoulli numbers, where we need the famous fact
about their denominator.

\begin{theorem}[{von Staudt \cite{Staudt:1840}, Clausen \cite{Clausen:1840}}]
\label{thm:staudt-clausen}
Let $n \in 2\NN$. Then
\[
   B_n + \sum_{p-1 \pdiv n} \frac{1}{p} \in \ZZ,
\]
which implies that the denominator of $B_n$ equals $\prod_{p-1 \pdiv n} p$.
\end{theorem}

\begin{prop} \label{prop:bn-decomp}
We have
\[
   \mathfrak{S}(n,1) = 1 \quad \text{and} \quad
   \mathfrak{D}(n,1) = \prod_{p-1 \pdiv n} \norm{pn}_p, \quad n \in 2\NN.
\]
\end{prop}

\begin{proof}
We make use of  Theorem \ref{thm:decomp-l-func}.
Since the conductor $\ff_1 = 1$, the product of $\mathfrak{S}(n,1)$
is always trivial. By Definition \ref{def:mod-l-func} and
Theorem \ref{thm:staudt-clausen}, we obtain
\[
   \mathfrak{D}(n,1) = \prod_{\substack{p-1 \pdiv n}}
     \; \norm{\zeta_{p,0}(s_{p,0})}_p^{-1}
     = \prod_{\substack{p-1 \pdiv n}} \norm{\frac{B_n}{n}}_p^{-1}
     = \prod_{\substack{p-1 \pdiv n}} \norm{pn}_p. \qedhere
\]
\end{proof}

Combining Theorem \ref{thm:decomp-l-func}, Proposition \ref{prop:bn-decomp},
and Conjecture \ref{conj:irr-prod}, we deduce the following.

\begin{conj}
The structure of the Bernoulli numbers is given by
\[
   \norm{\frac{B_n}{n}}_\infty = \prod_{p-1 \pdiv n} \norm{pn}_p
     \hspace*{-1ex} \prod_{\substack{(p,l) \in
     \IRR_1\\ l \equiv n \!\!\! \pmod{p-1}}}
     \hspace*{-3.5ex} \norm{p \, (s_{p,l} - \xi_{p,l})}_p^{-1},
     \quad n \in 2\NN,
\]
where $\xi_{p,l}$ is the zero of $\zeta_{p,l}$.
\end{conj}

\begin{remark*}
This conjecture about the Bernoulli numbers was already given by the author,
see \cite[Rem.~4.17, p.~421]{Kellner:2007}, where the numerator of $B_n/n$,
assuming the conjecture, can be described by zeros and the denominator of
$B_n/n$ can be described, without any assumption, by poles of $p$-adic zeta
functions. Since the Bernoulli numbers and the Riemann zeta function can be
viewed as the prototype of the generalized Bernoulli numbers and $L$-functions,
one may speculate, whether this \textit{simple} formula above holds generally.
\end{remark*}

Secondly, we consider the Euler numbers. Here we have the more complicated
behavior of the non-irregular part, since we do not have a denominator as in
the case of the Bernoulli numbers.

\begin{theorem}[{Frobenius \cite{Frobenius:1910}$^1$,
Carlitz \cite{Carlitz:1954a}$^2$, \cite{Carlitz:1959}$^3$}] \label{thm:en-prop}
Let $n \in 2\NN$. Then
\begin{equation} \label{eq:en-prop-1}
   E_n \equiv \left\{
     \begin{array}{cl}
       0, & (p \equiv 1 \pmod{4}), \\
       2, & (p \equiv 3 \pmod{4}),
     \end{array} \right. \pmod{p^r} \quad^{1,2}
\end{equation}
when $\eulerphi(p^r) \pdiv n$. Moreover,
\begin{equation} \label{eq:en-prop-2}
   \frac{E_n}{2} = - \frac{B_{n+1,\chi_{-4}}}{n+1} \equiv \frac12 \pmod{1}.
   \quad^3
\end{equation}
\end{theorem}

We can sharpen the result of the theorem above by determining the exact
$p$-power, that divides $E_n$ in \eqref{eq:en-prop-1}, for most cases.

\begin{prop}
Let $p \equiv 1 \pmod{4}$ and $n \in 2\NN$ where $p-1 \pdiv n$.
We have the following statements:
\begin{enumerate}
\item If $\ord_p E_{p-1} = 1$, then
  \[
     \norm{E_n}_p = \norm{pn}_p, \quad \text{resp.,} \quad
       \ord_p E_n = 1 + \ord_p n.
  \]
  Otherwise, we have
  \[
     \norm{E_n}_p < \norm{pn}_p, \quad \text{resp.,} \quad
       \ord_p E_n \geq 2 + \ord_p n.
  \]
\item If $\ord_p E_{p-1} \geq 2$ and $\ord_p (E_{2(p-1)} - 2E_{p-1}) = 2$,
  then
  \[
     \norm{E_n}_p = \norm{p^2s(s-\xi)}_p,
  \]
  where $s=n/(p-1)$ and $\xi \in \ZZ_p$ is the unique simple zero of
  $\tilde{L}_{p,0}(\cdot,\chi_{-4})$.
\end{enumerate}
\end{prop}

\begin{proof}
Note that $p > 3$. We apply Theorem \ref{thm:except-l-func} to
$f = L_{p,0}(\cdot,\chi_{-4})$, where we have the connection with the
Euler numbers by
\[
   \norm{L_{p,0}(s,\chi_{-4})}_p = \norm{\tfrac12 E_{s(p-1)}}_p
     = \norm{E_{s(p-1)}}_p, \quad s \in \NN,
\]
and especially $\norm{L_{p,0}(1,\chi_{-4})}_p = \norm{E_{p-1}}_p$.
Moreover, we have
\[
   L_{p,0}(s,\chi_{-4}) = ps \, \tilde{L}_{p,0}(s,\chi_{-4}),
     \quad s \in \ZZ_p.
\]

(1): Since
\[
   \ord_p E_{p-1} = 1 \iff \ord_p L_{p,0}(1,\chi_{-4}) = 1 \iff \lambda_f = 1,
\]
we obtain $\norm{L_{p,0}(s,\chi_{-4})}_p = \norm{ps}_p$ for $s \in \ZZ_p$.
It follows for $n=s(p-1) \in \NN$ that
\[
   \norm{E_n}_p = \norm{L_{p,0}(s,\chi_{-4})}_p = \norm{ps}_p
     = \norm{pn/(p-1)}_p = \norm{pn}_p.
\]
Otherwise, we have the case that $\lambda_f > 1$,
$\ord_p E_{p-1} \geq 2$, and
$\norml{\tilde{L}_{p,0}(s,\chi_{-4})}_p < 1$ for $s \in \ZZ_p$.
This implies the inequalities given above.

(2): We show that the conditions $\ord_p E_{p-1} \geq 2$ and
$\ord_p (E_{2(p-1)} - 2E_{p-1}) = 2$ give the
necessary and sufficient conditions for $\lambda_f = 2$.
As seen above, $\ord_p E_{p-1} \geq 2$ implies that $\lambda_f > 1$.
Thus it remains the condition
$\Delta_f(2) = \Dop^2 L_{p,0}(0,\chi_{-4}) / p^2 \in \ZZ^*_p$
to ensure that $\lambda_f = 2$. Since $L_{p,0}(0,\chi_{-4}) = 0$,
we obtain the condition
$\ord_p(L_{p,0}(2,\chi_{-4}) - 2L_{p,0}(1,\chi_{-4})) = 2$,
which is equivalent to $\ord_p (E_{2(p-1)} - 2E_{p-1}) = 2$,
since the Euler factors have no effect.
Now, we have established that $\lambda_f = 2$. Hence
\[
   \norm{E_n}_p = \norm{L_{p,0}(s,\chi_{-4})}_p = \norm{p^2s(s-\xi)}_p,
\]
where $n=s(p-1) \in \NN$ and $\xi \in \ZZ_p$ is the unique simple zero of
$\tilde{L}_{p,0}(\cdot,\chi_{-4})$.
\end{proof}

\begin{remark*}
Ernvall \cite{Ernvall:1979} notices that the exception such that
$p^2 \pdiv E_{p-1}$ occurs first for $p = 29\,789$, which is the only known
example for the Euler numbers yet. Therefore $(29\,789,0) \in \EXC_{\chi_{-4}}$.
The table below represents our computations for this prime showing that
$\tilde{L}_{p,0}(\cdot,\chi_{-4})$ has a unique simple zero. Moreover, using a
congruence of \cite[p.~628]{Ernvall&Met:1979}, which also follows as a special
case of Theorem \ref{thm:tpl-gen-ks2} for the functions $T_{p,l}(\cdot,\chi)$
for $\chi = \chi_{-4}$, we have not found any further exceptional prime
$p \equiv 1 \pmod{4}$ below $10^6$. These exceptional primes seem to be very
rarely in the case of the Euler numbers. The result of \cite{Ernvall&Met:1979}
is the computation of cyclotomic invariants and $E$-irregular pairs below
$10^4$. We have extended the computations of $E$-irregular pairs up to
$p < 10^5$ to show that $L_{p,l}(\cdot,\chi_{-4}) \in \WKS^0_{p,2}$ for
$(p,l) \in \IRR_{\chi_{-4}}$ in this range. We give all details of the
computations and used methods in \cite{Kellner:2009}.
\end{remark*}

\begin{tbl}
Computed parameters of function $f = L_{p,0}(\cdot,\chi_{-4}) \in \KS_{p,2}$
for $p= 29\,789$:
\begin{center}
\begin{tabular}{|c|c|c|c|} \hline
  $p$ & $\lambda_f$ & $\ord_p f(1)$ & $\ord_p \Delta_f(2)$ \\ \hline\hline
  29\,789 & 2 & 2 & 0 \\ \hline
\end{tabular}
\end{center}
verified by
\begin{center}
\begin{tabular}{|c|c|c|c|} \hline
  & $\pmod{p^3}$ & $\pmod{p^3} / p^2$ & $\ord_p$ \\ \hline\hline
  $E_{p-1}$     & 22651377283046 & 25526 & 2 \\ \hline
  $E_{2(p-1)}$  & 20830464245954 & 23474 & 2 \\ \hline
  $E_{2(p-1)} - 2E_{p-1}$ & 1962007175931 & 2211 & 2 \\ \hline
\end{tabular}
\end{center}
\end{tbl}

\begin{prop} \label{prop:en-decomp}
We have
\[
   \mathfrak{S}(n,\chi_{-4}) = \frac12 \quad \text{and} \quad
   \mathfrak{D}(n,\chi_{-4}) = \mathfrak{D}_{0}(n,\chi_{-4}) \hspace*{-2ex}
     \prod_{\substack{p-1 \pdiv n\\ p \equiv 1 \!\!\! \pmod{4}}}
     \hspace*{-2ex} \norm{pn}_p^{-1}, \quad n \in 2\NN.
\]
\end{prop}

\begin{proof}
We apply Theorems \ref{thm:decomp-l-func}, \ref{thm:decomp-d-prod},
and \ref{thm:en-prop}.
Since $\ff_{\chi_{-4}} = 4$, we obtain by \eqref{eq:en-prop-2} that
\[
   \mathfrak{S}(n,\chi_{-4}) = \norm{L_{2,0}(s_{2,0},\chi_{-4})}_2^{-1}
     = \norm{\frac{E_n}{2}}_2^{-1} = \frac12.
\]
Observing that $\chi_{-4}$ is odd and $2B_{1,\chi_{-4}} = -E_0 = -1$,
the product of $\mathfrak{D}(n,\chi_{-4})$ reduces to
\[
  \mathfrak{D}_{0}(n,\chi_{-4}) \mathfrak{D}_+(n,\chi_{-4}) =
   \mathfrak{D}_{0}(n,\chi_{-4}) \hspace*{-2ex}
     \prod_{\substack{p-1 \pdiv n\\ p \equiv 1 \!\!\! \pmod{4}}}
     \hspace*{-2ex} \norm{pn}_p^{-1},
\]
where we have used that $\chi_{-4}(p) = 1 \iff p \equiv 1 \pmod{4}$.
\end{proof}

Note that the factor $\mathfrak{S}(n,\chi_{-4}) = \frac12$
is cancelled regarding $E_n = 2L(-n,\chi_{-4})$ for even $n$.
Combining Theorems \ref{thm:decomp-l-func} and \ref{thm:decomp-d-prod},
Proposition \ref{prop:en-decomp},
and Conjectures \ref{conj:irr-prod} and \ref{conj:except-prod},
we deduce the following.

\begin{conj}
The structure of the Euler numbers is given by
\[
   \norm{E_n}_\infty = \prod_{\substack{p-1 \pdiv n\\
     p \equiv 1 \!\!\! \pmod{4}}}
     \hspace*{-2.5ex} \norm{pn}_p^{-1}
     \prod_{\substack{(p,l) \in \IRR_{\chi_{-4}} \cup \EXC_{\chi_{-4}}\\
     l \equiv n \!\!\! \pmod{p-1}}}
     \hspace*{-4.5ex} \norm{p \, (s_{p,l} - \xi_{p,l})}_p^{-1},
     \quad n \in 2\NN,
\]
where $\xi_{p,l}$ is the zero of $L_{p,l}(\cdot,\chi_{-4})$ when $l \neq 0$ and
$\xi_{p,0}$ is the zero of $\tilde{L}_{p,0}(\cdot,\chi_{-4})$.
\end{conj}

\begin{remark*}
It is quite remarkable that the Bernoulli and Euler numbers, defined by simple
generating functions and lying in $\QQ$, resp., $\ZZ$, can be completely
described only with the aid of the theory of $p$-adic functions. In other
words, the Riemann zeta function and $L$-functions at negative integer
arguments encode $p$-adic information about zeros of $p$-adic functions.
\end{remark*}
\medskip

\section{Fermat quotients}

In this section, let $p$ be an odd prime.
Let $\SU_p = 1 + p\ZZ_p$. As usual, we have the decomposition
\[
   a = \omega(a) \bracket{a}, \quad a \in \ZZ_p^*,
\]
where $\bracket{a} \in \SU_p$ and
$\omega: \ZZ_p^* \to \ZZ_p^*$ is the Teichm\"uller character,
that gives the $(p-1)$-th roots of unity in $\QQ_p$ and satisfies
$\omega(a) \equiv a \pmod{p\ZZ_p}$. Define the Fermat quotient by
\[
   q(a) = \frac{a^{p-1}-1}{p}, \quad a \in \ZZ_p^*.
\]
Further define
\[
   u(a): \ZZ_p^* \to \SU_p, \quad a \mapsto a^{p-1}.
\]
Alternatively, we also have $u(a) = \bracket{a}^{p-1} = 1 + p\,q(a)$.
Basic properties of the Fermat quotient were given by Lerch \cite{Lerch:1905};
one of these is the similar behavior like the $\log$ function:
\[
   q(ab) \equiv q(a)+q(b) \pmod{p\ZZ_p}, \quad a, b \in \ZZ^*_p.
\]

Now, we introduce the functions $T_{p,l}$ as follows.

\begin{defin}
Let $l,r$ be integers where $r \geq 1$. We define
\[
   T_{p,l}^r(s) = \sum_{a=1}^{p-1} a^l q(a)^r u(a)^s, \quad s \in \ZZ_p.
\]
We also write $T_{p,l}(s)$ for $T_{p,l}^1(s)$.
\end{defin}

\begin{prop} \label{prop:tpl-prop}
Let $\nu,l,r$ be integers where $\nu \geq 0$ and $r \geq 1$.
We have the following statements:
\begin{enumerate}
\item $T_{p,l}^r \in \KS_{p,2}$.
\item $\KSop^\nu \, T_{p,l}^r = T_{p,l}^{r+\nu}$.
\item The Mahler expansion is given by
  \[
     T_{p,l}^r(s) = \sum_{\nu \geq 0} \Delta_{T_{p,l}^r}(\nu) \,
       p^\nu \binom{s}{\nu}, \quad s \in \ZZ_p,
  \]
  where $\Delta_{T_{p,l}^r}(\nu) = T_{p,l}^{r+\nu}(0)$.
\item We have
  \[
     T_{p,l}^r(s) \equiv T_{p,l+p-1}^r(s) \equiv T_{p,l}^{r+p-1}(s)
       \pmod{p\ZZ_p}, \quad s \in \ZZ_p.
  \]
\item We have a certain recurrent behavior of the coefficients such that
  \[
     \Delta_{T_{p,l}^r}(\nu) \equiv \Delta_{T_{p,l}^r}(\nu+p-1)
       \equiv \Delta_{T_{p,l+p-1}^r}(\nu) \equiv \Delta_{T_{p,l}^{r+p-1}}(\nu)
       \pmod{p\ZZ_p}.
  \]
\item Define $S_n(m) = 1^n + \cdots + (m-1)^n$ for $m \geq 1$, $n \geq 0$.
If $l \geq 0$, then
  \[
     T_{p,l}^r(s) = p^{-r} \, \Doph{p-1}{r} S_t(p) \valueat{t=l+s(p-1)},
       \quad s \in \NN_0.
  \]
\end{enumerate}
\end{prop}

\begin{proof}
Let $s \in \ZZ_p$.
(1): From Proposition \ref{prop:exp-func} we deduce that the functions
\[
   a^l q(a)^r u(a)^s \in \KS_{p,2}
\]
for all $a \in \{1,\ldots,p-1\}$. Thus the sum $T_{p,l}^r(s) \in \KS_{p,2}$.
(2): A simple calculation shows that
\[
   \KSop \, u(a)^s = \Dop u(a)^s/p = q(a) u(a)^s.
\]
So we get $\KSop \, T_{p,l}^r(s) = T_{p,l}^{r+1}(s)$
and iteratively $\KSop^\nu \, T_{p,l}^r(s) = T_{p,l}^{r+\nu}(s)$.
(3): This follows by (2) and Lemma \ref{lem:oper-ks2}.
(4),(5): This is a consequence of Fermat's little theorem.
The case (5) follows by $\Delta_{T_{p,l}^r}(\nu) = T_{p,l}^{r+\nu}(0)$.
(6): Using the binomial expansion
\[
   a^l q(a)^r u(a)^s = p^{-r} a^{l+s(p-1)} (a^{p-1}-1)^r = p^{-r} \,
     \Doph{p-1}{r} a^t \valueat{t=l+s(p-1)}
\]
and summing over $a$ yield the result.
\end{proof}

The next theorem shows a link between the behavior of the functions $T_{p,l}$
and $\zeta_{p,l}$.

\begin{theorem} \label{thm:tpl-ks2}
Let $p > 3$ and $l \in 2\NN$ where $0 < l < p-1$.
We have the relations
\begin{alignat*}{2}
   T_{p,l}(s) &\equiv \zeta_{p,l}(s) &&\pmod{p\ZZ_p}, \quad s \in \ZZ_p, \\
   \Delta_{T_{p,l}} &\equiv 2 \Delta_{\zeta_{p,l}} &&\pmod{p\ZZ_p},
     \quad l \neq 2.
\end{alignat*}
The functions $T_{p,l}$ and $\zeta_{p,l}$ have the same classification
in $\KS_{p,2}$ such that
\begin{align*}
   T_{p,l} \in \KS^*_{p,2}  \quad &\iff \quad \zeta_{p,l} \in \KS^*_{p,2}, \\
   T_{p,l} \in \WKS^0_{p,2} \quad &\iff \quad \zeta_{p,l} \in \WKS^0_{p,2}.
\end{align*}
\end{theorem}

\begin{proof}
We use the well known congruences, giving a relation between Fermat quotients
and Bernoulli numbers, which can be deduced by Proposition \ref{prop:tpl-prop}
(6), cf.~\cite[Prop.~1,2, p.~855]{Ernvall&Met:1991}:
\begin{alignat}{2}
   T_{p,l}^1(0) &= \sum_{a=1}^{p-1} a^l q(a) &\equiv& - \frac{B_l}{l}
     \pmod{p\ZZ_p} \label{eq:loc-q-Bn-1}, \\
   T_{p,l}^2(0) &= \sum_{a=1}^{p-1} a^l q(a)^2 \,\, &\equiv&
     - \frac{2}{p} \left( \frac{B_{l+p-1}}{l+p-1} - \frac{B_l}{l} \right)
     \pmod{p\ZZ_p}, \quad l \neq 2. \label{eq:loc-q-Bn-2}
\end{alignat}
From \eqref{eq:loc-q-Bn-1} it follows that
\[
   T_{p,l}(0) = T_{p,l}^1(0) \equiv \zeta_{p,l}(0) \pmod{p\ZZ_p},
\]
since the Euler factor of $\zeta_{p,l}$ vanishes for $l \geq 2$.
This also shows that $T_{p,l}(s) \equiv \zeta_{p,l}(s) \pmod{p\ZZ_p}$
for $s \in \ZZ_p$ and consequently that
$T_{p,l} \in \KS^*_{p,2} \iff \zeta_{p,l} \in \KS^*_{p,2}$.
Because $\frac{B_2}{2} = \frac{1}{12} \not\equiv 0 \pmod{p\ZZ_p}$,
we always have $T_{p,2}, \zeta_{p,2} \in \KS^*_{p,2}$.
Similarly we deduce by \eqref{eq:loc-q-Bn-2} and
Proposition \ref{prop:tpl-prop} (2), (3) that
\[
   \Delta_{T_{p,l}} \equiv \KSop \, T_{p,l}(0) = T_{p,l}^2(0)
     \equiv 2 \KSop \, \zeta_{p,l}(0) \equiv 2 \Delta_{\zeta_{p,l}}
     \pmod{p\ZZ_p},
\]
observing that the Euler factors of $\KSop \, \zeta_{p,l}$ vanish for $l > 2$.
Now, we have $T_{p,l} \notin \KS^*_{p,2} \iff \zeta_{p,l} \notin \KS^*_{p,2}$,
which implies $l > 2$. In these cases we also have
$\Delta_{T_{p,l}} \equiv 2 \Delta_{\zeta_{p,l}} \pmod{p\ZZ_p}$
and consequently $T_{p,l} \in \WKS^0_{p,2} \iff \zeta_{p,l} \in \WKS^0_{p,2}$.
\end{proof}

We can generalize the results in the following way.

\begin{defin} \label{def:tpl-gen}
Let $\chi \in \mathfrak{X}_2$, $\chi \neq 1$, and $p \notdiv \ff_\chi$.
Let $l,r$ be integers where $r \geq 1$. We define
\[
   T_{p,l}^r(s,\chi) = \frac{1}{\ff_\chi} \sum_{\substack{a=1\\(a,p) = 1}}
     ^{p \, \ff_\chi}
     \chi(a) a^{l+\delta_\chi} q(a)^r u(a)^s, \quad s \in \ZZ_p.
\]
We write $T_{p,l}(s,\chi)$ for $T_{p,l}^1(s,\chi)$. Further define
\[
   S_{n,\chi}(m) = \sum_{a=1}^m \chi(a) a^n, \quad
   S^*_{n,\chi}(m) = \hspace*{-1ex} \sum_{\substack{a=1\\(a,p) = 1}}^m
     \hspace*{-1ex} \chi(a) a^n, \quad m \geq 1, n \geq 0.
\]
The generalized Bernoulli polynomials are given by
\[
   B_{n,\chi}(x) = \sum_{\nu = 0}^n \binom{n}{\nu} B_{\nu,\chi} \, x^{n-\nu},
     \quad n \geq 1, x \in \RR.
\]
\end{defin}

\begin{prop} \label{prop:tpl-gen-prop}
Let $\chi \in \mathfrak{X}_2$, $\chi \neq 1$, and $p \notdiv \ff_\chi$.
Let $\nu,l,r$ be integers where $\nu \geq 0$ and $r \geq 1$.
We have the following statements:
\begin{enumerate}
\item $T_{p,l}^r(\cdot,\chi) \in \KS_{p,2}$.
\item $\KSop^\nu \, T_{p,l}^r(\cdot,\chi) = T_{p,l}^{r+\nu}(\cdot,\chi)$.
\item The Mahler expansion is given by
  \[
     T_{p,l}^r(s,\chi) = \sum_{\nu \geq 0} \Delta_{T_{p,l}^r(\cdot,\chi)}
      (\nu) \, p^\nu \binom{s}{\nu}, \quad s \in \ZZ_p,
  \]
  where $\Delta_{T_{p,l}^r(\cdot,\chi)}(\nu) = T_{p,l}^{r+\nu}(0,\chi)$.
\item We have
  \[
     T_{p,l}^r(s,\chi) \equiv T_{p,l+p-1}^r(s,\chi)
       \equiv T_{p,l}^{r+p-1}(s,\chi) \pmod{p\ZZ_p}, \quad s \in \ZZ_p.
  \]
\item We have a certain recurrent behavior of the coefficients such that
  \begin{align*}
     \Delta_{T_{p,l}^r(\cdot,\chi)}(\nu)
       &\equiv \Delta_{T_{p,l}^r(\cdot,\chi)}(\nu+p-1) \\
       &\equiv \Delta_{T_{p,l+p-1}^r(\cdot,\chi)}(\nu)
       \equiv \Delta_{T_{p,l}^{r+p-1}(\cdot,\chi)}(\nu)
       \pmod{p\ZZ_p}.
  \end{align*}
\item If $l \geq 0$, then
  \[
     T_{p,l}^r(s,\chi) = \ff_\chi^{-1} p^{-r} \, \Doph{p-1}{r} S^*_{t,\chi}
       (p\,\ff_\chi) \valueat{t=l+\delta_\chi+s(p-1)}, \quad s \in \NN_0.
  \]
\end{enumerate}
\end{prop}

\begin{proof}
The proof is exactly derived as the proof of Proposition
\ref{prop:tpl-prop} by considering the additional factor $\chi(a)$ in the
sum of $T_{p,l}^r$ and excluding those $a$ where $p \pdiv a$.
\end{proof}

\begin{theorem} \label{thm:tpl-gen-ks2}
Let $\chi \in \mathfrak{X}_2$, $\chi \neq 1$, $p > 3$, and $p \notdiv \ff_\chi$.
Let $l \in 2\NN_0$ where $0 \leq l < p-1$. We have the relations
\begin{alignat*}{3}
   T_{p,l}(s,\chi) &\equiv L_{p,l}(s,\chi) &&\pmod{p\ZZ_p}, \quad s \in \ZZ_p, \\
   \Delta_{T_{p,l}(\cdot,\chi)} &\equiv 2 \Delta_{L_{p,l}(\cdot,\chi)}
     &&\pmod{p\ZZ_p}.
\end{alignat*}
The functions $T_{p,l}(\cdot,\chi)$ and $L_{p,l}(\cdot,\chi)$ have the same
classification in $\KS_{p,2}$ such that
\begin{align*}
   T_{p,l}(\cdot,\chi) \in \KS^*_{p,2}  \quad &\iff \quad L_{p,l}(\cdot,\chi)
     \in \KS^*_{p,2}, \\
   T_{p,l}(\cdot,\chi) \in \WKS^0_{p,2} \quad &\iff \quad L_{p,l}(\cdot,\chi)
     \in \WKS^0_{p,2}.
\end{align*}
\end{theorem}

The case $\chi=1$ is compatible with the former results of Proposition
\ref{prop:tpl-prop} and Theorem \ref{thm:tpl-ks2} since
$T_{p,l}^r(s,1) = T_{p,l}^r(s)$. The difference between the cases
$\chi = 1$ and $\chi \neq 1$ is only caused by the von Staudt-Clausen Theorem
\ref{thm:staudt-clausen}, while the latter case implies that we already have
$p$-integrality of the numbers $B_{n,\chi}$ in question. We need some
preparations to prove Theorem \ref{thm:tpl-gen-ks2}.

\begin{prop}[{\cite[p.~463]{Neukirch:1992}}$^1$, \cite{Carlitz:1959}$^2$]
\label{prop:bn-gen-prop}
Let $\chi \neq 1$ be a primitive non-principal character where
$p \notdiv \ff_\chi$ and $m,n \geq 1$. We have the following statements:
\begin{enumerate}
\item
  \[
     S_{n,\chi}(m) = \frac{1}{n+1} \left( B_{n+1,\chi}(m) - B_{n+1,\chi}
       \right), \quad \ff_\chi \pdiv m. \quad ^1
  \]
\item
  \[
     B_{0,\chi} = 0, \quad B_{n,\chi}/n \text{\ is $p$-integral}. \quad ^2
  \]
\item
  \[
     S^*_{n,\chi}(m) = S_{n,\chi}(m) - \chi(p)p^n \, S_{n,\chi}(m/p),
       \quad p \pdiv m.
  \]
\item If $p \, \ff_\chi \pdiv m$ and $n \equiv \delta_\chi \pmod{2}$, then
  \begin{alignat*}{2}
     S_{n,\chi}(m)/m &\equiv B_{n,\chi} && \pmod{p^2}, \quad n \geq 1, \\
     S_{n,\chi}(m)/m &\equiv B_{n,\chi} + \binom{n}{3} \frac{B_{n-2,\chi}}{n-2}
       \, m^2 && \pmod{p^4}, \quad n \geq 3.\\
  \end{alignat*}
\end{enumerate}
\end{prop}

\begin{proof}
(3): This follows by
\[
   S_{n,\chi}(m) = \hspace*{-1ex} \sum_{\substack{a=1\\(a,p) = 1}}^m
     \hspace*{-1ex} \chi(a) a^n + \sum_{a=1}^{m/p} \chi(pa) (pa)^n.
\]

(4): By Definition \ref{def:tpl-gen} and using (1) and (2), we obtain
\[
   S_{n,\chi}(m)/m = \frac{1}{n+1} \sum_{\nu = 0}^n \binom{n+1}{\nu}
     B_{\nu,\chi} \, m^{n-\nu} = \sum_{\nu = 1}^n \binom{n}{\nu-1}
     \frac{B_{\nu,\chi}}{\nu} \, m^{n-\nu},
\]
where the numbers $B_{\nu,\chi}/\nu$ are $p$-integral. Since $n \equiv
\delta_\chi \pmod{2}$, we have $B_{n,\chi} \neq 0$ and $B_{n-1,\chi} = 0$;
also $B_{n-2,\chi} \neq 0$ and $B_{n-3,\chi} = 0$ if $n \geq 3$.
By assumption $p \pdiv m$ and this implies the congruences in $\QQ(\chi)$.
\end{proof}

\begin{lemma} \label{lem:congr-bn-gen}
Let $\chi \neq 1$ be a primitive non-principal character where $p > 3$ and
$p \notdiv \ff_\chi$. Let $n, r$ be integers with $n, r \geq 1$ and
$n \equiv \delta_\chi \pmod{2}$. We have
\[
   \Doph{p-1}{r} \, (1 - \chi(p) p^{t-1}) B_{t,\chi} \valueat{t=n}
     \equiv -r \, \Doph{p-1}{r-1} \, (1 - \chi(p) p^{t-1}) \frac{B_{t,\chi}}{t}
     \valueat{t=n} \pmod{p^r}.
\]
Moreover the congruence above vanishes$\pmod{p^{r-1}}$.
\end{lemma}

\begin{proof}
Note that the congruences are valid in $\QQ(\chi)$. For brevity we write
$\widetilde{B}_{t,\chi} = (1 - \chi(p) p^{t-1}) B_{t,\chi}/t$.
As a consequence of the Kummer congruences, cf. Theorem \ref{thm:lp-congr},
we have
\begin{equation} \label{eq:loc-congr-bn-gen}
   \Doph{p-1}{r} \, \widetilde{B}_{t,\chi} \valueat{t=n} \equiv 0 \pmod{p^r}.
\end{equation}
Thus we can write
\begin{align*}
  &\Doph{p-1}{r} \, (1 - \chi(p) p^{t-1}) B_{t,\chi} \valueat{t=n} \\
    &\quad\equiv \Doph{p-1}{r} \left( t \widetilde{B}_{t,\chi}
      - (n + r(p-1)) \widetilde{B}_{t,\chi} \right) \valueat{t=n} \\
    &\quad\equiv \sum_{\nu=0}^r \binom{r}{\nu} (-1)^{r-\nu} (-(r-\nu)(p-1))
      \widetilde{B}_{n+\nu(p-1),\chi} \\
    &\quad\equiv r(p-1) \sum_{\nu=0}^{r-1} \binom{r-1}{\nu} (-1)^{r-1-\nu}
      \widetilde{B}_{n+\nu(p-1),\chi} \\
    &\quad\equiv -r \, \Doph{p-1}{r-1} \, (1 - \chi(p) p^{t-1})
      \frac{B_{t,\chi}}{t} \valueat{t=n} \pmod{p^r},
\end{align*}
where we have used that $\binom{r}{\nu} = \frac{r}{r-\nu} \binom{r-1}{\nu}$
and the last part of the congruences is divisible by $p^{r-1}$ in view of
\eqref{eq:loc-congr-bn-gen}.
\end{proof}

\begin{proof}[Proof of Theorem \ref{thm:tpl-gen-ks2}]
Set $l' = l + \delta_\chi$ and $m = p \, \ff_\chi$.
Note that $\ff_\chi \in \ZZ^*_p$, $l' \geq 0$, and $p \geq 5$. Let $r \geq 1$.
We write $\EF{t} = (1 - \chi(p) p^{t-1})$ for the Euler factors.
By Propositions \ref{prop:tpl-gen-prop} (6), \ref{prop:bn-gen-prop} (3)
we obtain
\begin{equation} \label{eq:loc-tpl-gen-ks2-m}
\begin{split}
   p^{r-1} \, T_{p,l}^r(s,\chi) &= \Doph{p-1}{r} S^*_{t,\chi}(m)/m
      \valueat{t=l'+s(p-1)} \\
   &= \Doph{p-1}{r} S_{t,\chi}(m)/m \valueat{t=l'+s(p-1)} \\
      & \quad - \chi(p) \, \Doph{p-1}{r} p^{t-1} S_{t,\chi}
      (\ff_\chi)/\ff_\chi \valueat{t=l'+s(p-1)}, \quad s \in \NN_0.
\end{split}
\end{equation}
To avoid complemental Euler factors in the cases where $s=0$ and $l'$ is small,
caused by the second summand above, we shall shift the index $l'$ to
$l'_1 = l'+p-1$.
This simplifies the congruences and we can add Euler factors when needed. Recall
Corollary \ref{corl:ks2-transl-delta} which shows that the coefficients
$\Delta_f(\nu)$ of functions $f \in \KS_{p,2}$ are invariant$\pmod{p\ZZ_p}$
under translation.

Case $r=1$: Using \eqref{eq:loc-tpl-gen-ks2-m}, Propositions
\ref{prop:tpl-gen-prop}, \ref{prop:bn-gen-prop} (4), and
Lemma \ref{lem:congr-bn-gen}, we deduce that
\begin{align*}
   T_{p,l}(0,\chi) &\equiv T_{p,l}(1,\chi)
     \equiv \Doph{p-1}{} S_{t,\chi}(m)/m \valueat{t=l'_1} \\
     &\equiv \Doph{p-1}{} B_{t,\chi} \valueat{t=l'_1}
     \equiv \Doph{p-1}{} \, \EF{t} B_{t,\chi} \valueat{t=l'_1} \\
     &\equiv -\EF{l'_1} \frac{B_{l'_1,\chi}}{l'_1}
     \equiv L_{p,l}(1,\chi) \equiv L_{p,l}(0,\chi) \pmod{p\ZZ_p}.
\end{align*}
Since $T_{p,l}(\cdot,\chi), L_{p,l}(\cdot,\chi) \in \KS_{p,2}$, we also
have $T_{p,l}(s,\chi) \equiv L_{p,l}(s,\chi) \pmod{p\ZZ_p}$ for $s \in \ZZ_p$.
This shows that $T_{p,l}(\cdot,\chi) \in \KS^*_{p,2} \iff
L_{p,l}(\cdot,\chi) \in \KS^*_{p,2}$.
\smallskip

Case $r=2$: First, we have by Proposition \ref{prop:tpl-gen-prop} that
\[
   \Delta_{T_{p,l}(\cdot,\chi)} \equiv \KSop \, T_{p,l}(0,\chi)
     \equiv T_{p,l}^2(0,\chi) \equiv T_{p,l}^2(1,\chi) \pmod{p\ZZ_p}.
\]
Again, using \eqref{eq:loc-tpl-gen-ks2-m}, Proposition
\ref{prop:bn-gen-prop} (4), and Lemma \ref{lem:congr-bn-gen}, we derive that
\begin{align*}
  p \, T_{p,l}^2(1,\chi)
    &\equiv \Doph{p-1}{2} S_{t,\chi}(m)/m \valueat{t=l'_1} \\
    &\equiv \Doph{p-1}{2} B_{t,\chi} \valueat{t=l'_1}
      \equiv \Doph{p-1}{2} \, \EF{t} B_{t,\chi} \valueat{t=l'_1} \\
    &\equiv -2 \Doph{p-1}{} \, \EF{t} \frac{B_{t,\chi}}{t} \valueat{t=l'_1}
      \pmod{p^2\ZZ_p}
\end{align*}
and further that
\begin{align*}
   T_{p,l}^2(1,\chi)
     &\equiv - \frac{2}{p}
       \Doph{p-1}{} \, \EF{t} \frac{B_{t,\chi}}{t} \valueat{t=l'_1} \\
     &\equiv 2 \KSop L_{p,l}(1,\chi) \equiv 2 \KSop L_{p,l}(0,\chi)
       \equiv 2 \Delta_{L_{p,l}(\cdot,\chi)} \pmod{p\ZZ_p}.
\end{align*}
Thus $\Delta_{T_{p,l}(\cdot,\chi)} \equiv 2 \Delta_{L_{p,l}(\cdot,\chi)}
\pmod{p\ZZ_p}$. Considering case $r=1$ we deduce that
$T_{p,l}(\cdot,\chi) \in \WKS^0_{p,2} \iff
L_{p,l}(\cdot,\chi) \in \WKS^0_{p,2}$.
\end{proof}

\begin{remark*}
Let $\chi \in \mathfrak{X}_2$, $p \notdiv \ff_\chi$, $l \in 2\NN$, and
$0 < l < p-1$. The function $T_{p,l}(\cdot,\chi)$ has a unique simple zero if
and only if $L_{p,l}(\cdot,\chi)$ has a unique simple zero. This can only
happen, when $(p,l) \in \IRR_\chi$ and $\Delta_{T_{p,l}(\cdot,\chi)} \equiv
2 \Delta_{L_{p,l}(\cdot,\chi)} \not\equiv 0 \pmod{p\ZZ_p}$. Example
\ref{expl:zero-tpl} shows the analogy to Example \ref{expl:zero-bn-en}.
Moreover, we have a kind of reciprocity relation as follows.
\end{remark*}

\begin{corl}
Let $\chi \in \mathfrak{X}_2$, $p \notdiv \ff_\chi$, $l \in 2\NN$,
$0 < l < p-1$, and $(p,l) \in \IRR_\chi$.
If $T_{p,l}(\cdot,\chi) \in \WKS^0_{p,2}$ or
$L_{p,l}(\cdot,\chi) \in \WKS^0_{p,2}$, then
\[
   \frac{\tau_T}{\tau_L} \cdot \frac{\xi_L}{\xi_T} \equiv 2 \pmod{p\ZZ_p},
\]
where $\tau_T$, $\tau_L \in p\ZZ_p$ is the fixed point and
$\xi_T$, $\xi_L \in \ZZ_p$ is the zero of
$T_{p,l}(\cdot,\chi)$, $L_{p,l}(\cdot,\chi)$, respectively.
\end{corl}

\begin{proof}
This follows by Lemma \ref{lem:rel-zero-fixpnt}, Theorem \ref{thm:tpl-ks2}
for $\chi = 1$, and Theorem \ref{thm:tpl-gen-ks2} for $\chi \neq 1$.
\end{proof}

\begin{expl} \label{expl:zero-tpl}
Computed zeros $\xi$ and fixed points $\tau$$\pmod{p^{10}}$ of functions
of $\WKS^0_{p,2}$.
\begin{enumerate}

\item Case $(p,l)=(37,32)$ and $f = T_{p,l} \in \WKS^0_{p,2}$:
\begin{center}
\begin{tabular}{|c|l|} \hline
  $f$ & values / $p$-adic digits $s_0, \ldots, s_9$ \\ \hline\hline
  $\Delta_f, \lambda_f$ & 32, 1 \\ \hline
  $\xi$  & 19, 1, 24, 12, 16, 24, 22, 26, 12, 33 \\ \hline
  $\tau$ & 0, 21, 31, 31, 14, 25, 15, 2, 10, 27 \\ \hline
\end{tabular}
\end{center}

\item Case $(p,l)=(19,10)$ and $f = T_{p,l}(\cdot,\chi_{-4}) \in \WKS^0_{p,2}$:
\begin{center}
\begin{tabular}{|c|l|} \hline \allowdisplaybreaks
  $f$ & values / $p$-adic digits $s_0, \ldots, s_9$ \\
    \hline\hline
  $\Delta_f,\lambda_f$ & 10, 1 \\ \hline
  $\xi$  & 4, 8, 6, 1, 18, 14, 8, 3, 3, 3 \\ \hline
  $\tau$ & 0, 17, 8, 16, 0, 2, 7, 2, 10, 14 \\ \hline
\end{tabular}
\end{center}

\end{enumerate}
\end{expl}

\begin{theorem} \label{thm:tpl-lpl-d2}
Let $\chi \in \mathfrak{X}_2$, $\chi \neq 1$, $p > 3$, and $p \notdiv \ff_\chi$.
Let $l \in 2\NN_0$ where $0 \leq l < p-1$. We have the relations
\begin{alignat*}{2}
   3 \Delta_{L_{p,l}(\cdot,\chi)}(2)
     &\equiv \Delta_{T_{p,l}(\cdot,\chi)}(2)
       - \ff_\chi^2 \, \Delta_{T_{p,l-2}(\cdot,\chi)}(0) \\
     &\equiv T^3_{p,l}(0,\chi)
       - \ff_\chi^2 \, T_{p,l-2}(0,\chi) &&\quad \pmod{p\ZZ_p}, \\
   4 \Delta_{L_{p,l}(\cdot,\chi)}(3)
     &\equiv \Delta_{T_{p,l}(\cdot,\chi)}(3)
       - 2 \ff_\chi^2 \, \Delta_{T_{p,l-2}(\cdot,\chi)} \\
     &\equiv T^4_{p,l}(0,\chi)
       - 2 \ff_\chi^2 \, T^2_{p,l-2}(0,\chi) &&\quad \pmod{p\ZZ_p}.
\end{alignat*}
\end{theorem}

\begin{proof}
Set $l' = l + \delta_\chi$ and $m = p \, \ff_\chi$.
Define $\ell_2 = p-3$ for $l = 0$ otherwise $\ell_2 = l-2$.
Note that $\ff_\chi \in \ZZ^*_p$, $l' \geq 0$, and $p \geq 5$. We write
$\EF{t} = (1 - \chi(p) p^{t-1})$. We use the same arguments given in the
proof of Theorem \ref{thm:tpl-gen-ks2}. Thus we shift the index $l'$ to
$l'_2 = l'+2(p-1)$, which is sufficient for the following congruences since
$p \geq 5$. We now consider the cases $r=3$ and $r=4$ simultaneously.
By Proposition \ref{prop:tpl-gen-prop} we have
\[
   \Delta_{T_{p,l}(\cdot,\chi)}(r-1)
     \equiv \KSop^{r-1} \, T_{p,l}(0,\chi)
     \equiv T_{p,l}^r(0,\chi)
     \equiv T_{p,l}^r(2,\chi) \pmod{p\ZZ_p}.
\]
From \eqref{eq:loc-tpl-gen-ks2-m} and Proposition
\ref{prop:bn-gen-prop} (4) we deduce that
\begin{align*}
   p^{r-1} \, T_{p,l}^r(2,\chi)
     &\equiv \Doph{p-1}{r} S_{t,\chi}(m)/m \valueat{t=l'_2} \\
     &\equiv \Doph{p-1}{r} \, \EF{t} B_{t,\chi} \valueat{t=l'_2} \\
     &\quad   + p^2 \ff_\chi^2 \, \Doph{p-1}{r} \binom{t}{3}
       \, \EF{t-2} \frac{B_{t-2,\chi}}{t-2} \valueat{t=l'_2} \pmod{p^r\ZZ_p},
\end{align*}
where we have already added Euler factors in the last part of the congruence.
Applying Lemma \ref{lem:congr-bn-gen} to the first summand provides that
\[
   \Doph{p-1}{r} \, \EF{t} B_{t,\chi} \valueat{t=l'_2}
     \equiv -r \Doph{p-1}{r-1} \, \EF{t} \frac{B_{t,\chi}}{t} \valueat{t=l'_2}
       \pmod{p^r\ZZ_p}.
\]
With the help of Corollary \ref{corl:ks2-transl-delta} we further obtain that
\begin{align*}
   -r p^{-(r-1)} \Doph{p-1}{r-1} \, \EF{t} \frac{B_{t,\chi}}{t} \valueat{t=l'_2}
     &\equiv r \, \KSop^{r-1} L_{p,l}(2,\chi) \\
     &\equiv r \, \KSop^{r-1} L_{p,l}(0,\chi)
       \equiv r \, \Delta_{L_{p,l}(\cdot,\chi)}(r-1) \pmod{p\ZZ_p}.
\end{align*}
Now we have to separate the cases. Case $r=3$:
By use of the Kummer congruences the second summand turns into
\begin{align*}
   \Doph{p-1}{3} \binom{t}{3} \EF{t-2} \frac{B_{t-2,\chi}}{t-2} \valueat{t=l'_2}
     &\equiv - \EF{l'_2-2} \frac{B_{l'_2-2,\chi}}{l'_2-2} \\
     &\equiv \left\{
       \begin{array}{ll}
         L_{p,p-3}(1,\chi), & (l' < 2),\\
         L_{p,l-2}(2,\chi), & (l' \geq 2),
       \end{array} \right. \\
     &\equiv L_{p,\ell_2}(0,\chi) \pmod{p\ZZ_p},
\end{align*}
where we have used Lemma \ref{lem:delta-poly} to derive that
\begin{equation} \label{eq:loc-tpl-lpl-d2}
   \Doph{p-1}{n} \binom{t}{n} = (p-1)^n, \quad t \in \ZZ_p, n \geq 1.
\end{equation}
Theorem \ref{thm:tpl-gen-ks2} and Proposition \ref{prop:tpl-gen-prop} (4)
show that
\[
   L_{p,\ell_2}(0,\chi) \equiv T_{p,\ell_2}(0,\chi) \equiv T_{p,l-2}(0,\chi)
     \equiv \Delta_{T_{p,l-2}(\cdot,\chi)}(0) \pmod{p\ZZ_p}.
\]
Combining the results from above we finally achieve that
\[
   \Delta_{T_{p,l}(\cdot,\chi)}(2)
     \equiv 3\Delta_{L_{p,l}(\cdot,\chi)}(2)
       + \ff_\chi^2 \Delta_{T_{p,l-2}(\cdot,\chi)}(0)
       \pmod{p\ZZ_p}.
\]
Case $r=4$: Using $L_{p,\ell_2}(\cdot,\chi)$ as above, we can write
\[
   -\Doph{p-1}{4} \binom{t}{3} \EF{t-2} \frac{B_{t-2,\chi}}{t-2} \valueat{t=l'_2}
     \equiv \Doph{p-1}{4} \binom{t}{3}
       L_{p,\ell_2}(s(t),\chi) \valueat{t=l'_2} \pmod{p^2\ZZ_p},
\]
where we use the variable substitution $s(t) = s_{l'_2} + (t-l'_2)/(p-1)$ and
$s_{l'_2}$ is a suitable constant depending on $l'_2$.
Since $f = L_{p,\ell_2}(\cdot,\chi) \in \KS_{p,2}$, we have
\[
   f(s) \equiv f(0) + p \, \Delta_f \, s \pmod{p^2\ZZ_p}, \quad s \in \ZZ_p.
\]
Thus we obtain
\begin{align*}
   \Doph{p-1}{4} \binom{t}{3} f(s(t)) \valueat{t=l'_2}
     &\equiv \Doph{p-1}{4} \binom{t}{3} \left( f(0) + p \, \Delta_f
       \left( s_{l'_2} + \frac{t-l'_2}{p-1} \right)
       \right) \valueat{t=l'_2} \\
     &\equiv \left( f(0) + p \, \Delta_f \left( s_{l'_2} + \frac{3-l'_2}{p-1}
       \right) \right) \Doph{p-1}{4} \binom{t}{3} \valueat{t=l'_2} \\
     &\quad + p \, \frac{4\Delta_f}{p-1} \Doph{p-1}{4} \binom{t}{3}
       \frac{t-3}{4} \valueat{t=l'_2} \\
     &\equiv - p \, 4\Delta_f \pmod{p^2\ZZ_p},
\end{align*}
where the first summand vanishes by Lemma \ref{lem:delta-poly} and the second
summand is reduced by \eqref{eq:loc-tpl-lpl-d2} observing that $\binom{t}{4} =
\binom{t}{3} \frac{t-3}{4}$. By Theorem \ref{thm:tpl-gen-ks2} and Proposition
\ref{prop:tpl-gen-prop} (5) we have
\[
   2 \Delta_f \equiv 2 \Delta_{L_{p,\ell_2}(\cdot,\chi)}
     \equiv \Delta_{T_{p,\ell_2}(\cdot,\chi)}
     \equiv \Delta_{T_{p,l-2}(\cdot,\chi)} \pmod{p\ZZ_p}.
\]
Putting all together, we finally get
\[
   \Delta_{T_{p,l}(\cdot,\chi)}(3)
     \equiv 4\Delta_{L_{p,l}(\cdot,\chi)}(3)
       + 2\ff_\chi^2 \Delta_{T_{p,l-2}(\cdot,\chi)} \pmod{p\ZZ_p}.
   \qedhere
\]
\end{proof}

As a result, we get criteria for exceptional primes regarding
$L_{p,0}(\cdot,\chi)$ and for the existence of unique simple zeros of the
corresponding functions $\tilde{L}_{p,0}(\cdot,\chi)$ in case $\chi$ is odd.

\begin{corl} 
Let $\chi \in \mathfrak{X}_2$ and $p > 3$ where $\chi$ is odd and
$\chi(p) = 1$. For $f = L_{p,0}(\cdot,\chi)$ we have the following
statements:
\begin{enumerate}
\item We have
  \[
     T_{p,0}^2(0,\chi) \equiv 0 \pmod{p\ZZ_p} \quad \iff \quad
       (p,0) \in \EXC_\chi.
  \]
\item If $(p,0) \in \EXC_\chi$, then
  \[
     T^3_{p,0}(0,\chi) \not\equiv \ff_\chi^2 \, T_{p,p-3}(0,\chi)
       \pmod{p\ZZ_p} \quad \iff \quad \lambda_f = 2.
  \]
  If $\lambda_f = 2$, then $\tilde{L}_{p,0}(\cdot,\chi)$ has a
  unique simple zero.
\end{enumerate}
\end{corl}

\begin{proof}
This is a consequence of Proposition \ref{prop:tpl-gen-prop} and
Theorems \ref{thm:except-l-func}, \ref{thm:tpl-gen-ks2}, and
\ref{thm:tpl-lpl-d2}.
\end{proof}

\bibliographystyle{amsplain}

\begin{thebibliography}{10}

\bibitem{Buhler&others:2001}
J.~Buhler, R.~Crandall, R.~Ernvall, T.~Mets\"ankyl\"a, and M.~A. Shokrollahi,
\newblock \emph{Irregular primes and cyclotomic invariants to 12 million},
\newblock J. Symb. Comput. {\bf 31} (2001), no.~1--2, 89--96.

\bibitem{Carlitz:1954a}
L.~Carlitz,
\newblock \emph{A note on Euler numbers and polynomials},
\newblock Nagoya Math. J. {\bf 7} (1954), 35--43.

\bibitem{Carlitz:1954b}
L.~Carlitz,
\newblock \emph{Note on irregular primes},
\newblock Proc. Amer. Math. Soc. {\bf 5} (1954), 329--331.

\bibitem{Carlitz:1959}
L.~Carlitz,
\newblock \emph{Arithmetic properties of generalized Bernoulli numbers},
\newblock J. Reine Angew. Math. {\bf 202} (1959), 174--182.

\bibitem{Clausen:1840}
T.~Clausen,
\newblock \emph{Lehrsatz aus einer Abhandlung \"uber die Bernoullischen
  Zahlen},
\newblock Astr. Nachr. {\bf 17} (1840), 351--352.

\bibitem{Davis&Webb:1990}
K.~S. Davis and W.~A. Webb,
\newblock \emph{Lucas' theorem for prime powers},
\newblock Eur. J. Comb. {\bf 11} (1990), no.~3, 229--233.

\bibitem{Ernvall:1979}
R. Ernvall,
\newblock \emph{Generalized Bernoulli numbers, generalized irregular primes,
  and class number},
\newblock Ann. Univ. Turku., Ser. A I {\bf 178} (1979), 72 pp.

\bibitem{Ernvall&Met:1979}
R.~Ernvall and T.~Mets\"ankyl\"a,
\newblock \emph{Cyclotomic invariants and E-irregular primes},
\newblock Math. Comp. {\bf 32} (1978), 617--629. Corr., {\bf 33} (1979), 433.

\bibitem{Ernvall&Met:1991}
R.~Ernvall and T.~Mets\"ankyl\"a,
\newblock \emph{Cyclotomic invariants for primes between 125000 and 150000},
\newblock Math. Comp. {\bf 56} (1991), 851--858.

\bibitem{Ferrero:1978}
B.~Ferrero,
\newblock \emph{Iwasawa invariants of abelian number fields},
\newblock Math. Ann. {\bf 234} (1978), 9--24.

\bibitem{Fresnel:1967}
J. Fresnel,
\newblock \emph{Nombres de Bernoulli et fonctions $L$ $p$-adiques},
\newblock Ann. Inst. Fourier {\bf 17} (1967), no.~2, 281--333.

\bibitem{Frobenius:1910}
G. Frobenius,
\newblock \emph{\"Uber die Bernoullischen Zahlen und die Eulerschen Polynome},
\newblock Sitzungsber. Preuss. Akad. Wiss. (1910), no.~2, 809--847.

\bibitem{Graham&others:1994}
R.~L. Graham, D.~E. Knuth, O.~Patashnik,
\newblock \emph{Concrete Mathematics},
\newblock Addison-Wesley, Reading, MA, USA, 1994.

\bibitem{Hao&Parry:1984}
F.~H. Hao and C.~J. Parry,
\newblock \emph{Generalized Bernoulli numbers and $m$-regular primes},
\newblock Math. Comp. {\bf 43} (1984), 273--288.

\bibitem{Holden:1998}
J. Holden,
\newblock \emph{Irregularity of prime numbers over real quadratic fields},
\newblock Algorithmic number theory, Third International Symposium,
  Proceedings (J. P. Buhler, ed.),
  Springer Lect. Notes Comput. Sci. {\bf 1423} (1998), 454--462.

\bibitem{Jensen:1915}
K.~L. Jensen,
\newblock \emph{Om talteoretiske Egenskaber ved de Bernoulliske Tal},
\newblock Nyt Tidsskr. for Math. {\bf 26} (1915), 73--83.

\bibitem{Kellner:2007}
B.~C. Kellner,
\newblock \emph{On irregular prime power divisors of the Bernoulli numbers},
\newblock Math. Comp. {\bf 76} (2007), 405–-441.

\bibitem{Kellner:2009}
B.~C. Kellner,
\newblock \emph{Zeros of $p$-adic $L$-functions associated with irregular
  and exceptional pairs},
\newblock in preparation, 2009.

\bibitem{Koblitz:1996}
N.~Koblitz,
\newblock \emph{$p$-adic Numbers, $p$-adic Analysis and Zeta-Functions},
  GTM {\bf 58},
\newblock Springer--Verlag, 2nd edition, 1996.

\bibitem{Kubota&Leopoldt:1964}
T.~Kubota and H.~W. Leopoldt,
\newblock \emph{Eine $p$-adische Theorie der Zetawerte,
  I: Einf\"uhrung der $p$-adischen Dirichletschen $L$-Funktionen},
\newblock J. Reine Angew. Math. {\bf 214/215} (1964), 328--339.

\bibitem{Kummer:1851}
E.~E. Kummer,
\newblock \emph{\"Uber eine allgemeine Eigenschaft der rationalen
  Entwicklungscoefficienten einer bestimmten Gattung analytischer
  Functionen},
\newblock J. Reine Angew. Math. {\bf 41} (1851), 368--372.

\bibitem{Leopoldt:1958}
H.~W. Leopoldt,
\newblock \emph{Eine Verallgemeinerung der Bernoullischen Zahlen},
\newblock Abh. Math. Semin. Univ. Hamb. {\bf 22} (1958), 131--140.

\bibitem{Lerch:1905}
M.~Lerch,
\newblock \emph{Zur Theorie des Fermatschen Quotienten $\frac{a^{p-1}-1}p=q(a)$},
\newblock Math. Ann. {\bf 60} (1905), 471--490.

\bibitem{Mahler:1958}
K.~Mahler,
\newblock \emph{An Interpolation Series for Continuous Functions
  of a $p$-adic Variable},
\newblock J. Reine Angew. Math. {\bf 199} (1958), 23-–34. Corr.,
  {\bf 208} (1961), 70–-72.

\bibitem{Neukirch:1992}
J.~Neukirch,
\newblock \emph{Algebraische Zahlentheorie}, Springer--Verlag, 1992.

\bibitem{Robert:2000}
A.~M. Robert,
\newblock \emph{A Course in $p$-adic Analysis}, GTM {\bf 198},
\newblock Springer--Verlag, 2000.

\bibitem{Staudt:1840}
K.~G.~C. von Staudt,
\newblock \emph{Beweis eines Lehrsatzes die Bernoulli'schen Zahlen betreffend},
\newblock J. Reine Angew. Math. {\bf 21} (1840), 372--374.

\bibitem{Sun:2000}
Z.~H. Sun,
\newblock \emph{Congruences concerning Bernoulli numbers and Bernoulli polynomials},
\newblock Discrete Appl. Math. {\bf 105} (2000), 193--223.

\bibitem{Washington:1997}
L.~C. Washington,
\newblock \emph{Introduction to Cyclotomic Fields}, GTM {\bf 83},
\newblock Springer--Verlag, 2nd edition, 1997.

\end{thebibliography}

\end{document}